\documentclass[11pt]{article}%
\pdfoutput=1 %
\usepackage{hyperref}

\usepackage[algo2e,linesnumbered,vlined,algoruled]{algorithm2e}
\hypersetup{colorlinks=true, linkcolor=blue!80!black, urlcolor=blue!80!black, citecolor = blue!80!black}

\usepackage{enumitem}
\usepackage{lipsum}

\newcommand{\Hdef}{\mathfrak{H}} %
\newcommand{\Act}[1]{\mathbb{A} \mathopen{} \left( #1 \right)\mathclose{}}    %
\newcommand{\Gen}{\mathbb{G}}    %
\newcommand{\Reals}{\mathcal{R}}
\newcommand{\Integers}{\mathbb{N}}

\newcommand{\succh}{\mathbb{K}_1}
\newcommand{\succm}{\mathbb{K}_2}

\newcommand{\cS}{\mathcal{S}}    %
\newcommand{\Level}{L}    %
\newcommand{\cB}{\mathcal{B}}      %
\newcommand{\Levelmax}{L_{\rm max}}      %
\newcommand{\Deltamax}{\Delta_{\rm max}}
\newcommand{\defined}{\triangleq}
\newcommand{\minimize}{\operatornamewithlimits{minimize}}
\newcommand{\maximize}{\operatornamewithlimits{maximize}}
\newcommand{\Lset}{\mathbb{L}_{\infty}}
\newcommand{\Uset}{\mathbb{U}_{\infty}}
\newcommand{\lambdaa}{\lambda_a}
\newcommand{\lambdal}{\lambda_{\ell}}
\newcommand{\lambdau}{\lambda_u}
\newcommand{\kappaieg}{\kappa_{i, {\rm eg}}}
\newcommand{\kappabmh}{\kappa_{{\rm H}}}
\newcommand{\kappag}{\kappa_{\rm g}}
\newcommand{\kappafcd}{\kappa_{\rm fcd}}
\newcommand{\partialC}{\partial_{\rm C}}
\newcommand{\Lh}{K_h}
\newcommand{\gammad}{\gamma_{\rm d}}
\newcommand{\gammai}{\gamma_{\rm i}}
\newcommand{\cG}{\mathcal{G}}
\newcommand{\cH}{\mathcal{H}}
\newcommand{\LnF}{K_{\nabla \! F}}
\newcommand{\Lnh}{K_{\nabla \! h}}
\newcommand{\maxhgrad}{K_{\nabla h}}
\newcommand{\LnFi}{K_{\nabla \! F_i}}
\newcommand{\Lag}{\mathcal{L}}
\newcommand{\rhoGOOMBAH}{\tilde{\rho}}
\newcommand{\sGOOMBAH}{\tilde{s}}

\newcommand{\goombahref}{\hyperref[alg:GOOMBAH]{\texttt{GOOMBAH}}\xspace}

\newcommand{\mspshortref}{\hyperref[alg:manifold_sampling]{\rm \texttt{MS-P}}\xspace}
\newcommand{\msdshortref}{\rm \texttt{MS-D}\xspace}
\newcommand{\msploopref}{\hyperref[alg:manifold_sampling_loop]{\rm \texttt{MS-P loop}}\xspace}

\newcommand{\lb}{\ell}
\newcommand{\ub}{u}

\newcommand{\ds}{\displaystyle}
\newcommand{\tn}{\textnormal}
\newcommand{\proj}[2]{\mathbf{proj}\left( #1 , #2 \right)}
\newcommand{\cop}[1]{\mathbf{co}\left(#1\right)} %
\newcommand{\clp}[1]{\mathbf{cl}\left( #1 \right)} %
\newcommand{\intp}[1]{\mathbf{int}\left( #1 \right)} %
\newcommand{\image}[1]{\mathbf{im}_F\left( #1 \right)} %
\newcommand{\hdom}{\image{\Levelmax \cap \Omega}} %

\newcommand{\subassref}[2]{{\Cref{ass:#1}.\ref{subass:#2}}}

\usepackage{amssymb,amsthm,amsmath}
\usepackage{array,booktabs}
\usepackage{amstext} %
\usepackage{multirow}
\usepackage{xcolor}

\usepackage[colorinlistoftodos]{todonotes}

\usepackage{subfigure}
\usepackage[normalem]{ulem} %

\usepackage[margin=0.9in]{geometry}
\usepackage{caption}
\usepackage{thmtools}
\usepackage{nameref}
\usepackage[nameinlink]{cleveref}

\newtheorem{theorem}{Theorem}%
\newtheorem{lemma}[theorem]{Lemma}
\newtheorem{proposition}[theorem]{Proposition}%

\newtheorem{remark}{Remark}%

\newtheorem{defi}{Definition}
\newtheorem{assumption}{Assumption}%

\DeclareRobustCommand{\rchi}{{\mathpalette\irchi\relax}}
\newcommand{\irchi}[2]{\raisebox{\depth}{$#1\chi$}} %

\crefname{defi}{definition}{definitions}
\Crefname{defi}{Definition}{Definitions}
\crefname{lemma}{lemma}{lemmas}
\Crefname{lemma}{Lemma}{Lemmas}
\crefname{assumption}{assumption}{assumptions}
\Crefname{assumption}{Assumption}{Assumptions}

\usepackage{multicol}

\usepackage{authblk}
\providecommand{\keywords}[1]{\noindent\textbf{Keywords:} #1}

\usepackage[maxnames=5,style=numeric-comp,firstinits=true,sorting=none,backend=bibtex]{biblatex}
\AtBeginBibliography{\small}
\renewbibmacro{in:}{}
\newcommand{\jefflabel}[1]{\label{#1}}
\newcommand{\newref}[1]{\Cpageref{#1}}

\begin{document}

\title{Structure-Aware Methods for Expensive Derivative-Free Nonsmooth Composite Optimization}

\author[1]{Jeffrey Larson}

\author[1]{Matt Menickelly}

\affil[1]{Mathematics and Computer Science Division, Argonne National Laboratory 
\linebreak[4] \texttt{\small jmlarson@anl.gov}; \texttt{\small mmenickelly@anl.gov}
{\small (both authors contributed equally)}}

\maketitle

\abstract{We present new methods for solving a broad class of bound-constrained nonsmooth composite minimization problems.
These methods are specially designed for objectives that are some known mapping of outputs from a computationally expensive function.
We provide accompanying implementations of these methods: in particular,
a novel manifold sampling algorithm (\mspshortref) with subproblems that are in a sense primal versions of the dual problems solved by previous manifold sampling methods
and a method (\goombahref) that employs more difficult optimization subproblems.
For these two methods, we provide rigorous convergence analysis and guarantees.
We demonstrate extensive testing of these methods.
Open-source implementations of the methods developed in this manuscript can be found at 
\url{github.com/POptUS/IBCDFO/}.}
~\\


\section{Introduction}
We consider optimization problems of the form
\begin{equation}
\label{eq:prob_statement}
    \begin{array}{rl}
    \displaystyle\minimize_{x \in \Omega} & f(x),\\
    \end{array}
\end{equation}
where $f:\Reals^n\to\Reals$, and $\Omega$ is a subset of the $n$-fold Cartesian product of the extended reals defined by bound constraints, namely, $\Omega\defined\left\{x:\lb \leq x \leq \ub \right\}$.\jefflabel{Omega_def}
We additionally assume $f$ in~\eqref{eq:prob_statement} is a \emph{composite} function, that is, $f$ satisfies the following.
\begin{assumption}\label{ass:f}\
The function $f$ in~\eqref{eq:prob_statement} has the form $f(x) \defined h(F(x))$, where $h:\Reals^p\to\Reals$ and $F:\Reals^n\to\Reals^p$.
\end{assumption}
In particular, we are interested in the case where $h$ is a known function that is cheap to evaluate
(with a known subdifferential) but a single evaluation of $F$ requires
considerable time or computational resources.
In this paper we consider $h$ in \Cref{ass:f} to be from a fairly broad family of functions called \emph{continuous selections}.

\begin{defi}\label{def:continuous_selection}
  The function $h$ is a \emph{continuous selection} on a set $U$ if it is
  continuous on $U$ and
  \[
  h(z) \in \{h_j(z):h_j\in\Hdef\}, \quad \forall z \in U,
  \]
  where $\Hdef$ is a finite set of functions $h_j: \Reals^p\to\Reals$, called \emph{selection functions}.
\end{defi}

The composite functions represented by continuous selections are extensive,
encapsulating virtually all practical use cases of composite optimization.
For example,
$h$ can be the 1-norm~\cite{LMW16}, a quantile function (e.g., a minimum or
maximum of entries in $F$)~\cite{womersley1986algorithm,Hare2017b,MMSMW2018}, a
piecewise-affine function~\cite{fletcher1982second,yuan1985superlinear,KLW18}, or even a piecewise-nonlinear function~\cite{Larson2020}. As one
example of the latter, most general case, particle accelerator physicists often
seek parameters $x$ that minimizes the minimum spread
of a simulated beam over a finite set of locations $J$ along a beam line~\cite{Eldred2022}.
High-fidelity simulations of such beam lines may require many thousands
of compute hours, producing copious amounts of output.
One way of quantifying the beam spread is the \emph{normalized emittance}~\cite{Wiedemann2015},
which takes some computed quantities $F_{1,j}(x), F_{2,j}(x), F_{3,j}(x)$ for each $j \in J$
and combines them with the mapping $\sqrt{F_{1,j}(x) F_{2,j}(x) - F_{3,j}(x)^2}$.
Therefore, the outer function is
$h(z) \defined \min_{j \in J} \sqrt{z_{1,j} z_{2,j} - z_{3,j}^2}$.

In this paper we provide convergence results for---and implementations
of---various methods for solving~\eqref{eq:prob_statement}. We begin by
extending past work in \emph{manifold sampling methods}, providing an
empirically superior ``primal variant'' (\mspshortref) of the manifold sampling algorithm
proposed in~\cite{Larson2020}. That algorithm %
was developed
for only the unconstrained version of~\eqref{eq:prob_statement}, and involves
subproblems that are dual to the subproblems involved in the present work. We
will demonstrate theoretical convergence results for this new primal variant of
manifold sampling under reasonable assumptions. Additionally, we will
demonstrate a method (\goombahref) that uses more difficult trust-region subproblems than
are typically analyzed;
in general, there are no guarantees that these subproblems can be (approximately) solved in finite time.
However, by
safeguarding this optimization method with the primal variant, we will
yield a provably convergent method that we find to perform exceptionally well
in practice, particularly when function evaluations are expensive.

\paragraph{Terminology and notation}
Before proceeding, we record the following definition of \emph{essentially active selection functions} pertinent to continuous selections.
\begin{defi}[adapted from~\cite{Scholtes2012}] \label{def:pc1manifold}
  If $h$ is a continuous selection function on $U$, define
  \begin{align*}
    \cS_i \defined \left\{ z: h(z) = h_j(z) \right\}, \quad
    \tilde{\cS}_i \defined \clp{\intp{\cS_i}}, \quad
    \Act{z} \defined \left\{ i: z \in \tilde{\cS}_i \right\}.
  \end{align*}
  Elements of $\Act{z}$ are called \emph{essentially active indices}; any
  $h_j$ for which $j\in \Act{z}$ is an \emph{essentially active selection
  function for $h$ at $z$}.
\end{defi}
With these definitions, we can make assumptions on $h$. Of course, these
assumptions need hold only at points where $h$ could possibly be
evaluated over the course of an optimization run.
For a constant $\Deltamax > 0$\jefflabel{Deltamax_def} bounding the possible trust-region radii $\Delta > 0$\jefflabel{Delta_def} and a starting point $x^0$, define
\begin{equation}\label{eq:Lmaxdef}
  \Levelmax \defined \ds \bigcup_{x \in \Level(x^0)} \cB(x;\Deltamax),
\end{equation}
where $\cB(x;\Delta) \defined \{ y: \left\| x - y \right\| \le \Delta \}$\jefflabel{cBdef} and
$\Level(x^0)$ is the $f(x^0)$ level set of $f$:\jefflabel{Leveldef}  
$\Level(x^0) \defined \{x\in\Reals^n: f(x) \leq f(x^0)\}$.

\begin{assumption}\label{ass:h}\
  With \Crefrange{def:continuous_selection}{def:pc1manifold} we assume the following about $h$.

  \begin{enumerate}[label=\tn{\textbf{\Alph*}.},ref=\tn{\Alph*},leftmargin=*]
    \item The function $h$
    is a continuous selection\footnote{
For practical purposes, one should attempt to define/construct $\Hdef$ such
that $\Hdef$ contains only functions that are essentially active somewhere in
the domain of $h$, $\hdom$.
} on $\hdom$.
    \label{subass:h}

  \item For any $z \in \hdom$, the essentially active indices $\Act{z}$ are computable. \label{subass:Ieh}
  \end{enumerate}
\end{assumption}
\begin{remark}
In \subassref{h}{Ieh}, we use the word ``computable" because in many
instances of continuous selections, some computation---ideally not much
more than $\mathcal{O}(p)$ arithmetic operations and much less than the
cost of evaluating $F$---is likely required to determine $\Act{z}$.
For a simple example, if
$h(z) \defined \max_{j\in\{1,\dots,p\}} h_j(z)$,
then to determine $\Act{z}$, the obvious way to implement the continuous
selection is to sort $\{h_j(z)\}_{j\in\{1,\dots,p\}}$ in ascending order---an
$\mathcal{O}(p\log p)$ computation---and then return the sorted index
(or indices, in the case of a tie) corresponding to the maximum value among the
values of $h_j(z)$.
\end{remark}

In this manuscript we use subscripts to index scalars and superscripts to index vectors.
All norms are assumed to be Euclidean unless otherwise stated.
The \emph{closure}, \emph{interior}, and \emph{convex hull} of a set
$\cS$ are denoted $\clp{\cS}$, $\intp{\cS}$, and
$\cop{\cS}$, respectively.
For ease of reference, we maintain a glossary of notation in the supplemental material in \Cref{sec:table}.

\section{Motivation and Background}\label{sec:motivation_and_background}
In this section we first discuss the foundational definitions and assumptions
for model-based derivative-free optimization (DFO) methods. We then present and discuss specialized
trust-region subproblems for use in algorithms
solving~\eqref{eq:prob_statement}. Furthermore we provide a high-level overview
of the
methods to be analyzed and contrast them with other methods for nonsmooth
optimization.

\subsection{Model-based methods}\label{sec:model-based_methods}
Manifold sampling methods belong to the class of \emph{model-based trust-region
methods}. Although manifold sampling
methods need not be derivative-free methods
(see~\cite{MMSMW2017,MMSMW2018,BHMMW20}), the inspiration and analysis are
heavily influenced by these
model-based trust-region methods from DFO (see~\cite{Conn2009} for a
complete treatment). The software attached to this paper is intended for DFO,
but all that is needed is that models of each component function $F_i$
satisfy the following definition.

\begin{defi} \label{def:flmodels}
  A function $m^{F_i} \colon \Reals^n \to \Reals$ is said to be a \emph{gradient-accurate}
  model of $F_i$ on $\cB(x;\Delta)$ if there
  exists a constant $\kappaieg$
  independent of $x$ and $\Delta$, so that
  \[
  \left\|\nabla F_i(x+s) - \nabla m^{F_i}(x+s)\right\|\leq \kappaieg\Delta \qquad \forall s\in \cB(0;\Delta).
  \]
\end{defi}
\Cref{def:flmodels} is similar to the definition of fully linear models that is
common in model-based DFO (e.g.,
\cite[Definition 6.1]{Conn2009a}) except it does not require accuracy of
function values.
As we will see, our method's subproblems requires only accurate gradients.
\Cref{def:flmodels} also resembles the concept of ``order-1 subgradient accuracy''
in~\cite[Definition 4.2]{audet2020model} with the exception that
\Cref{def:flmodels} is a statement about gradient (as opposed to arbitrary
subgradient) accuracy.
In model-based DFO, algorithms exist for constructing
fully linear (and hence, gradient-accurate) models $m^{F_i}$ by interpolation or regression on values of
$F_i$. Our software will maintain gradient-accurate models by employing the
same interpolation/regression subroutines as found in the software
\texttt{POUNDERS}~\cite{SWCHAP14}.

\subsection{Manifold sampling subproblems}\label{sec:manifold_sampling_suproblems}

Manifold sampling methods (e.g.,~\cite{LMW16,KLW18,Larson2020}) evaluate $F$
at various points $y$ in the domain $\Omega$.
Because the evaluation of $F$ is assumed to involve a non-negligible expense, all past
points $y$ and their corresponding values $F(y)$ are typically stored in memory for the purpose of model building.
Given that the continuous selection structure of $h$ is known (via \Cref{ass:h}), the values of $h_j(z)$ and $\nabla h_j(z)$ (for
any $h_j\in\Hdef$) are readily available at any point $z$, even if $h_j$ is not an essentially active selection function for $h$ at $z$.
Various manifold sampling implementations differ in
how they determine which indices $j$ corresponding to $h_j\in\Hdef$ to employ in subproblems, and how the
information $h_j(F(y))$ and $\nabla h_j(F(y))$ are used to produce putative iterates.
We will now describe how, in the novel manifold sampling method introduced in this paper, we choose indices of selection functions and subsequently choose putative iterates.

Over the run of our manifold sampling method, we will include \emph{any} point
$y\in\Reals^n$ for which $f(y)$ has been previously evaluated in a set $Y$,\jefflabel{Ydef}
along with the corresponding indices $\Act{F(y)}$.
On each iteration of the method, we will use an algorithmically determined current iterate $x^k$, an algorithmically updated radius $\Delta_k$, and this set $Y$ to determine a set of indices $\Gen^k$.
As shorthand notation, let $f_j(\cdot) \defined h_j(F(\cdot))$ for a given
selection index $j$. Moreover, with the model Jacobian $\nabla M_k(x^k)$\jefflabel{Mdef} obtained
by concatenating the model gradients $\nabla m^{F_i}_k$ in the $k$th iteration,
we abbreviate $g_j^k \defined \nabla M_k(x^k)\nabla h_j(F(x^k))$.\jefflabel{gen_def}
Then, we define $\Gen^k$ via
\begin{equation}
\label{eq:gen_k}
\begin{array}{rl}
\Gen^k & \defined \left\{j: j\in\Act{F(y)} \text{ for some } y\in Y, f_j(x^k) > f(x^k) \text{, and } \|x^k-y\|\leq c_1\Delta_k^2\right\}\\
& \bigcup
\left\{j: j\in\Act{F(y)} \text{ for some } y\in Y, f_j(x^k) \leq f(x^k) \text{, and } \|x^k-y\|\leq c_2\Delta_k\right\},
\end{array}
\end{equation}
where $c_1,c_2\geq 0$ are algorithmic parameters.
The presence of the term $c_1\Delta_k^2$ in~\eqref{eq:gen_k} is different from past manifold sampling work;
previous work would simply include $j\in\Gen^k$ provided there existed $y\in Y\cap
\cB(x^k;c_1\Delta_k)$ such that $j\in\Act{F(y)}$.
It will become clear in the proof of \Cref{lem:success} why this choice was made in the present work.

Note that in the extreme case where $c_1=c_2=0$ in~\eqref{eq:gen_k}, $\Gen^k$ includes indices corresponding only to selection functions
active at $F(x^k)$. This extreme case corresponds to the generator set
construction employed in \texttt{CMS} of~\cite{LMW16}.
The extreme choice of $c_1=c_2=0$ corresponds to an approximation of the Clarke subdifferential $\partialC f(x^k)$.\jefflabel{C_def}
By allowing strictly positive values of $c_1, c_2$, we are in principle constructing an inner approximation of the Clarke $\epsilon$-subdifferential ${\partialC}_{\epsilon} f(x^k) = \cop{\bigcup_{y\in \cB(x^k; \epsilon)} \partialC f(y)}$.

With this set of indices $\Gen^k$, we can now define a subproblem for use in the $k$th iteration to compute trial steps.
Additionally defining
\begin{equation*}
\beta_{j,k} \defined \displaystyle\max\{0, f_j(x^k)- f(x^k)\},
\end{equation*}
we define a \emph{primal model} via
\begin{equation}
    \label{eq:primal_model}
    m_{k}(s) = \displaystyle\max_{j\in\Gen^k}\{f_j(x^k) + (g^k_j)^\top s - \beta_{j,k}\} + \frac{1}{2}s^\top H^k s - f(x^k),
\end{equation}
where $H^k\in\Reals^{n\times n}$ is a symmetric matrix.
The subproblem employed in each iteration is then given by\footnote{For later reference, observe that
in the special case of $\Gen^k$ where $c_1=c_2=0$, a null step $s=0$
in~\eqref{eq:subproblem} implies that $v=f(x^k)$, which in turn implies the
objective value of~\eqref{eq:subproblem} is 0.}
\begin{equation}
\label{eq:subproblem}
\begin{array}{rl}
\displaystyle\minimize_{s} & m_{k}(s) \\
\text{subject to:}  & \|s\|\leq\Delta_k \\
            & \lb \leq x^k + s \leq \ub.\\
\end{array}
\end{equation}

We now provide some intuition concerning~\eqref{eq:subproblem} -- in
particular, explaining why the model of an arbitrary selection function $h$ is
replaced with a piecewise maximum \eqref{eq:primal_model} -- and connect it
with past manifold sampling work.
We first equivalently reformulate~\eqref{eq:subproblem} as
\begin{equation}
\label{eq:equiv_subproblem}
\begin{array}{rl}
\displaystyle\minimize_{v,s} & v+ \frac{1}{2}s^\top H^k s - f(x^k)\\
\text{subject to:} & v \geq f_j(x^k) + (g_j^k)^\top s - \beta_{j,k} \quad \forall j\in\Gen^k\\
            & \|s\|\leq\Delta_k \\
            & \lb \leq x^k + s \leq \ub.\\
\end{array}
\end{equation}
To handle bounds $\lb,\ub\in(\Reals\cup\{\infty\})^n$,\jefflabel{LandUdef} we introduce two sets,
\begin{equation*}
    \Lset\triangleq \{i\in 1,\dots,n: \lb_i = -\infty\}, \quad
    \Uset\triangleq \{i\in 1,\dots,n: \ub_i = \infty\}.
\end{equation*}
We also introduce, for brevity, a bilinear function
$$\Lambda(\lambda; x,a) \triangleq \lambdaa^\top a + \lambdal^\top(x-\lb) + \lambdau^\top(\ub-x),$$
where $a^k\in\Reals^{|\Gen^k|}$ \jefflabel{a_def} is defined entrywise by
$$ [a^k]_j = f(x^k) - f_j(x^k) + \beta_{j,k}.$$
With this notation,
one can show (see~\cite[Proposition 6]{MMSMW2017} for the case where $\Omega = \Reals^n$) that the strong Lagrangian dual of~\eqref{eq:subproblem} is
\begin{equation}
\label{eq:dual}
\begin{array}{rrl}
\displaystyle\maximize_{\lambdaa, \lambdal,\lambdau,\mu,\nu} & \frac{1}{2}\nu^\top (G^k\lambdaa -\lambdal + \lambdau) -\mu\displaystyle\frac{\Delta_k^2}{2} -\Lambda(\lambda; x^k,a^k) & \\
\text{subject to:} & \lambdaa, \lambdal, \lambdau, \mu \geq 0 &\\
        & [\lambdal]_i = 0 & , i\in \Lset\\
        & [\lambdau]_i = 0 & , i \in \Uset\\
    & e^\top\lambdaa = 1 & \\
    & H^k + \mu I_n \succeq 0 & \\
    & (H^k + \mu I_n)\nu = -G^k\lambdaa +\lambdal - \lambdau & , \\
\end{array}
\end{equation}
where
$G^k\in\Reals^{n\times|\Gen^k|}$ is the matrix with columns $\{g_j^k: j\in\Gen^k\}$,
$I_n\in\Reals^{n\times n}$ is an identity matrix, and
$e\in\Reals^{|\Gen^k|}$ is the vector of all ones.
As one might expect, our manifold sampling method derives a stationary measure
from~\eqref{eq:dual}; in particular, by setting $\Delta_k=1$ and
$H^k=0$, we arrive at a stationary measure $\rchi_k$  given by
\begin{equation}
  \small
\label{eq:chi}
\rchi_k \triangleq \displaystyle\minimize_{\lambdaa,\lambdal,\lambdau}\left\{
\begin{array}{l}
\|G^k\lambdaa - \lambdal + \lambdau\| +
\Lambda(\lambda; x^k,a^k) : \;\;
\lambdaa, \lambdal, \lambdau \geq 0, \\
e^\top\lambdaa = 1,\;
[\lambdal]_i = 0 \; i\in \Lset,\;
[\lambdau]_i = 0 \; i\in \Uset
\end{array}
\right\}.
\end{equation}
Because $H_k=0$, the constraint $\mu I_n \succeq 0$ coupled with the
maximization of $-\mu\frac{\Delta_k^2}{2}$ in \eqref{eq:dual} implies that
$\mu=0$. As a sanity check, notice that each entry of $a^k$ satisfies $f(x^k) -
f_j(x^k) + \beta_{j,k} \geq 0$. This, coupled with the observation that
$\lambdaa,\lambdal,\lambdau, x^k-\lb, \ub-x^k\geq 0$ for any $x^k$ feasible
with respect to the constraints of~\eqref{eq:subproblem}, gives the expected
characteristic that $\rchi_k\geq 0$.

We observe that in the absence of bound constraints (so that, effectively, $\lambdal = \lambdau =0)$, and replacing $a^k$ with a vector of zeros, the subproblem~\eqref{eq:chi} amounts to
\begin{equation}
    \small
    \label{eq:msd}
    \displaystyle\minimize_{\lambda} \left\{
    \begin{array}{l}
    \|G^k\lambda\| : \;\;
    \lambda \geq 0, e^\top\lambda = 1
    \end{array}
    \right\}.
\end{equation}
We observe that~\eqref{eq:msd} is precisely the direction-finding subproblem employed in previous manifold sampling works~\cite{LMW16,KLW18,Larson2020};
the $\lambda^*$ solving~\eqref{eq:msd} would be used to define a \emph{master model gradient} $g^k = G^k\lambda^*$, which in turn would be used to define a smooth \emph{master model}
\begin{equation}
    \label{eq:dual_model}
    m^D_k(s) \triangleq (g^k)^\top s + \frac{1}{2}s^\top H^k s.
\end{equation}
For this reason we refer to the model $m_k(s)$ in~\eqref{eq:primal_model} as the \emph{primal model}, while we refer to $m^D_k(s)$ in~\eqref{eq:dual_model} as the \emph{dual model}.

Having established this connection between previous and present work, we next address the question of what the primal model~\eqref{eq:primal_model} is actually modeling.
We illustrate in \Cref{fig:cartoons} the difference between~\eqref{eq:primal_model} and the dual model~\eqref{eq:dual_model} in the context of a trust-region subproblem.
We remark (importantly!) that the cartoon
function $h$ in \Cref{fig:cartoons} is the piecewise maximum function, and as
such,~\eqref{eq:primal_model} is a particularly good model of
$h\circ F$. Although, as we will show, employing the primal
model~\eqref{eq:primal_model} in the subproblem~\eqref{eq:subproblem} leads to
desirable convergence properties, the primal
model need not be a great local model of $h\circ F$ on any given iteration of a
manifold sampling problem in general. These considerations motivate our solver
\goombahref,  discussed in the next section.
\begin{figure}[t]
    \centering
    \includegraphics[width=.45\textwidth]{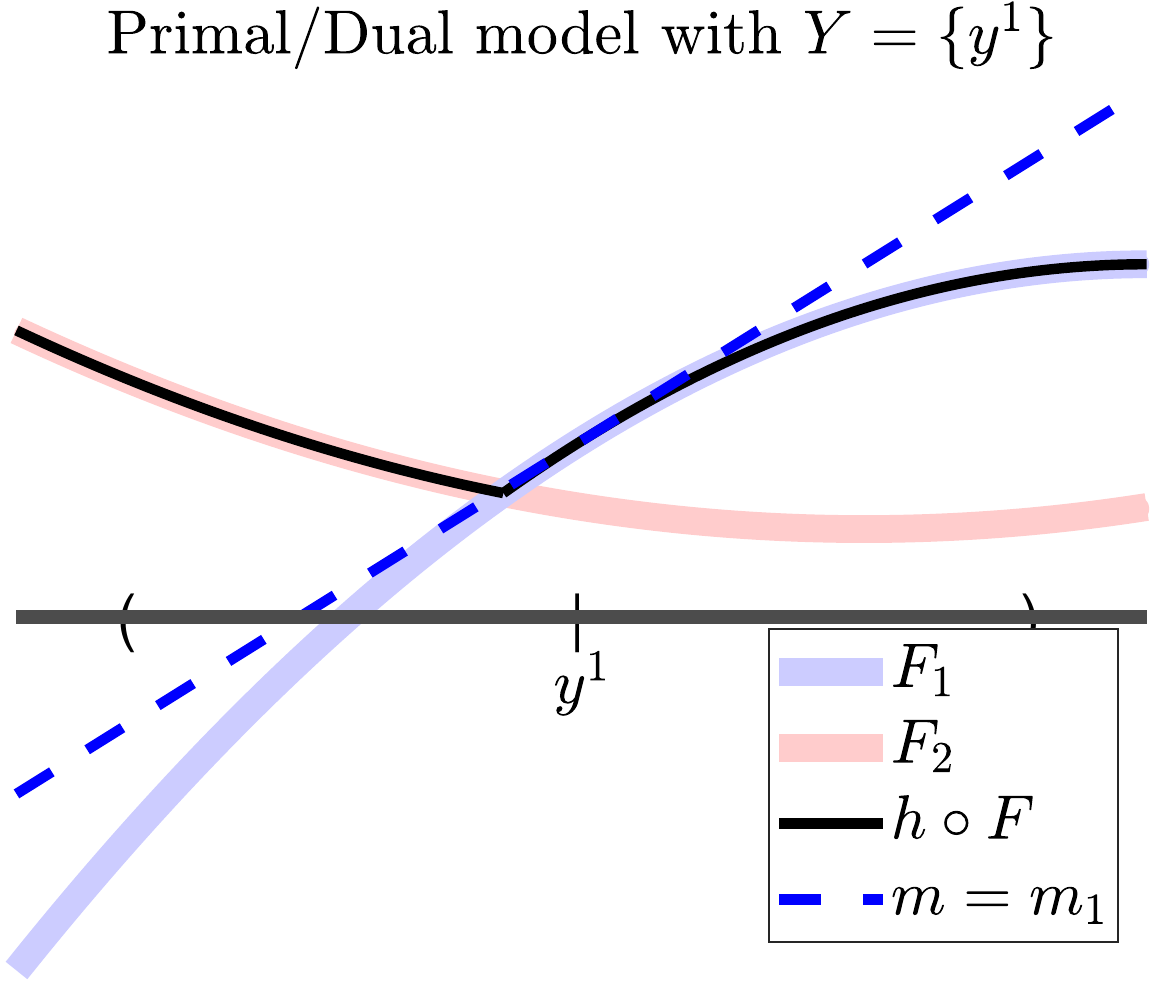}\\[10pt]
    \includegraphics[width=.45\textwidth]{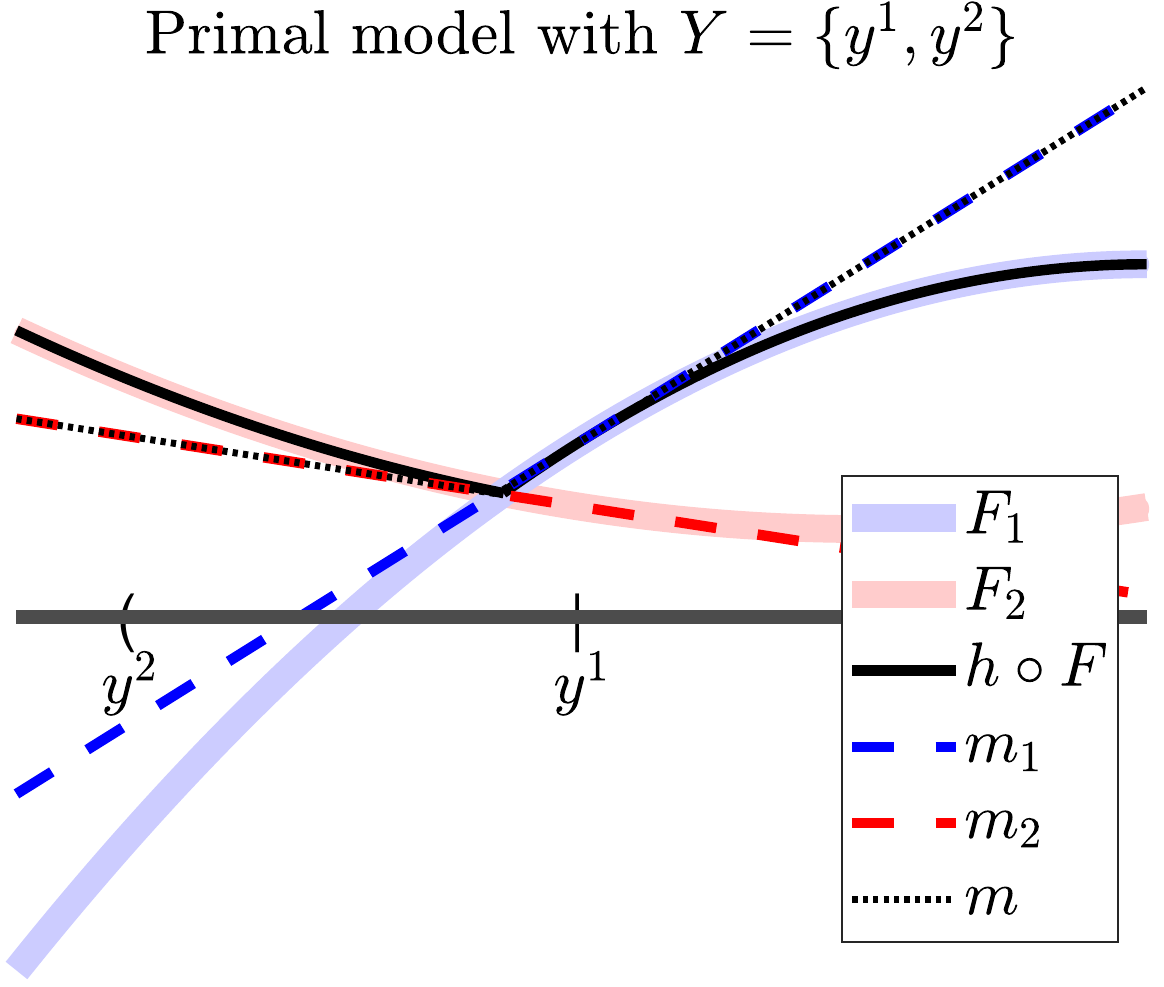}
    \hfil
    \includegraphics[width=.45\textwidth]{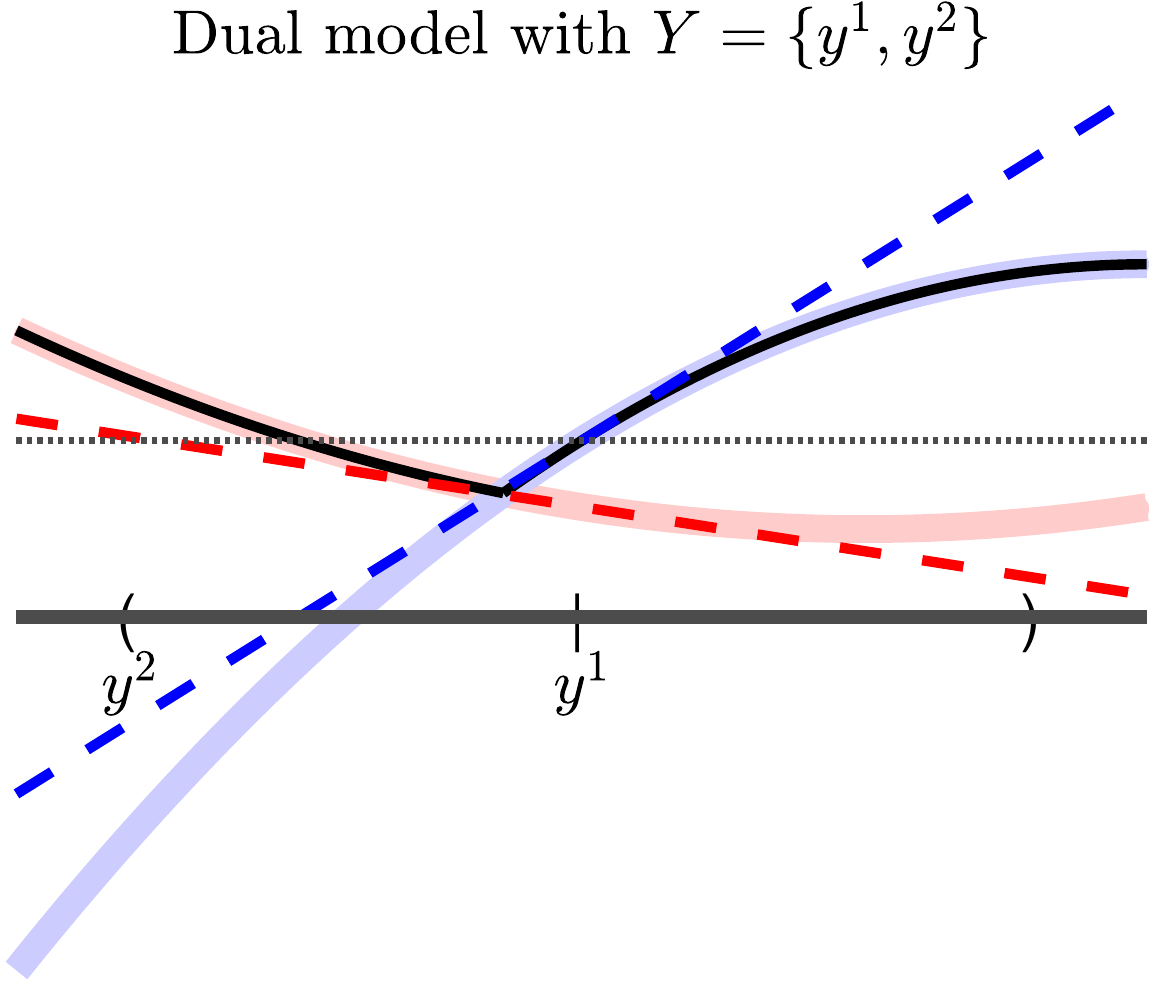}
    \caption{In the top row, with $Y=\{y^1\}$, the primal model~\eqref{eq:primal_model} employed by the
    manifold sampling method of this paper and the dual model~\eqref{eq:dual_model} of previous
    manifold sampling methods are identical. With this common model, the
    minimizer of the subproblem~\eqref{eq:subproblem} is $y^2$. Because the
    function $F_2$ is essentially active at $y^2$ (but not at $y^1$), the manifold sampling
    loop needs to include this information before proceeding. However, the primal and dual models differ when
    $Y=\{y^1,y^2\}$, as shown in the bottom row. The minimizer of~\eqref{eq:subproblem} provides a decrease in $h\circ F$,
    while~\eqref{eq:msd} produces a
    model gradient with a norm of zero, so no progress is made on the given iteration.
    \label{fig:cartoons}
    }
\end{figure}

\subsection{High-level discussion of \texttt{MS-P}}
Each iteration of \mspshortref consists of a pass through the \msploopref. 
In the \msploopref, we first construct $\Gen^k$ according
to~\eqref{eq:gen_k} and ensure that for each $j\in\Gen^k$, $f_j(x^k) +
(g_j^k)^\top s$ is a gradient-accurate model of $f_j$ on $\cB(x^k;\Delta_k)$.
We remark that, given a method to construct fully linear (and, therefore, gradient-accurate) models $M$ of $F$, it is straightforward to show under \Cref{ass:h_j} that $h(M(x))$ is a gradient-accurate model of $h(F(x))$; see~\cite[Theorem 32]{chen2021error}.
We then (approximately\footnote{The necessary conditions on approximate solution quality are given in \Cref{ass:cauchy}})
solve the subproblem~\eqref{eq:equiv_subproblem} to obtain
$(v_k,s^k)$.
We also
solve the
subproblem~\eqref{eq:chi} to measure stationarity $\rchi_k$. Provided
$\Delta_k$ is sufficiently small relative to $\rchi_k$ ($\Delta_k \leq
\eta_2\rchi_k$ for an algorithmic parameter $\eta_2$\jefflabel{eta2_def}), we continue the
iteration; otherwise we abort the iteration early, shrinking the trust region
and keeping the incumbent $x^{k+1}=x^k$. We note that, in general, we require
$\eta_2 \in (0,1/\kappabmh)$, where $\kappabmh$\jefflabel{kappabmh} is a fixed constant satisfying $\| H^k \| \le \kappabmh$ for all $k$.
Thus, in the special case where $\kappabmh=0$---which is the case in our numerical experiments---$\eta_2=\infty$ is an
acceptable setting, and thus this early termination will never happen.

Provided the \msploopref did not terminate early because of an overly large $\Delta_k$, we
evaluate $F(x^k+s^k)$ and compute the ratio
\begin{equation}
\label{eq:rho}
\rho_k \defined \displaystyle\frac{f(x^k)-f(x^k+s^k)}{f(x^k)-v_k-\frac{1}{2}s^{k\top}H^ks^k} =
\frac{f(x^k)-f(x^k+s^k)}{m_k(0) - m_k(s^k)}.
\end{equation}
If $\rho_k\geq\eta_1$ for some
algorithmic parameter $\eta_1>0$\jefflabel{eta1_def}, then the $k$th iteration is deemed
\emph{successful}, and we update $x^{k+1}$ with $x^k+s^k$ and possibly increase
$\Delta_k$ by a factor $\gammai \ge 1$ \jefflabel{gammai_def} and end the $k$th iteration.

On the other hand, if $\rho_k$ is too small and some new manifold was identified
when evaluating $F(x^k+s^k)$, we recompute $\Gen^k$ as in~\eqref{eq:gen_k}. If
$\Gen^k$ is unchanged by the addition of $x^k + s^k$ to $Y$, we shrink
$\Delta_k$ by a positive factor $\gammad < 1$ \jefflabel{gammad_def} and remain in the \msploopref if $\Gen^k \bigcap \Act{F(x^k + s^k)}$
is empty. If $\Gen^k \bigcap \Act{F(x^k + s^k)}$ is nonempty, the iteration is
declared \emph{unsuccessful}, and we set $x^{k+1}$ to $x^k$ and decrease $\Delta_k$. 
If $\Gen^k$ is in fact changed, we update models and $\Gen^k$,
resolve~\eqref{eq:subproblem} and~\eqref{eq:chi} with the new
$\Gen^k$, reevaluate $\rho_k$, and recheck for success. Note that
\Cref{lem:finite_termination} will show that the \msploopref will terminate in
finite time.

\subsection{High-level discussion of \texttt{GOOMBAH}}

As noted previously, any primal
model~\eqref{eq:primal_model} used by
\mspshortref can potentially be a poor model of
$h(F(x))$. However, the primal model captures the local first-order information
necessary to attain a convergence result in the limit of the $\{x^k\}$
generated by \mspshortref. Given that a closed form of $h$
is known and available, one can consider a
modification of \mspshortref with a trust-region subproblem of the form
\begin{equation} \label{eq:GOOMBAH_subproblem}
\begin{array}{rl}
\displaystyle\minimize_{s} & h(M(x^k + s)) \\
\text{subject to:} & \|s\|\leq\Delta_k,
\end{array}
\end{equation}
where $M$ remains a gradient-accurate model of $F$ on $\cB(x^k;\Delta_k)$.
If, for example, $M$ contains quadratic models for components of $F$ and if $h$ is a
general selection function, (approximately) solving~\eqref{eq:GOOMBAH_subproblem} may require considerable computational effort,
and no certificate of (approximate) optimality may be available. 
Importantly for deriving convergence results, there may be no known methods that can ensure Cauchy-like decrease in finite time given
general forms of~\eqref{eq:GOOMBAH_subproblem}. Therefore, assuming that
approximate solutions to~\eqref{eq:GOOMBAH_subproblem} are available (as is
done in \Cref{ass:cauchy} for the subproblem~\eqref{eq:equiv_subproblem}) would be a
somewhat hollow assumption.
Depending on the computational expense of
evaluating $F$, however, deploying a global optimization solver on~\eqref{eq:GOOMBAH_subproblem} with an appropriate time budget may be
justified. %
Thus, we propose to use the machinery of 
the \msploopref with its stationary measure
$\rchi_k$ to ensure sufficient decrease in each iteration with
respect to the stationary measure, thereby guaranteeing global convergence.
Pseudocode for this proposed algorithm is given in \goombahref, so
named because it employs a trust-region subproblem that is a glassbox optimization of a model of a blackbox in a hypersphere.

After computing the solution to~\eqref{eq:GOOMBAH_subproblem} in
each iteration, \goombahref computes the quantity
\begin{equation}\label{eq:rhoGOOMBAH}
\rhoGOOMBAH_k \defined \displaystyle\frac{f(x^k)-f(x^k+\sGOOMBAH^k)}{\Delta_k^{1+\omega}}
\end{equation}
for an algorithmic constant $\omega>0$.
If $\rhoGOOMBAH_k > \eta_1$, then we accept the subproblem~\eqref{eq:GOOMBAH_subproblem} solution as the next iterate as in a standard trust-region
method; if, on the other hand, the sufficient decrease suggested by~\eqref{eq:rhoGOOMBAH} is not attained, then we enter the \msploopref.
We study the utility of such recourse in the
numerical section by studying \goombahref that does not resort to manifold
sampling when the candidate~\eqref{eq:GOOMBAH_subproblem} solution is poor.

\subsection{Comparing/contrasting with other approaches}
Nonsmooth general (i.e., \emph{noncomposite}) optimization methods---that is,
methods that are suitable for minimizing continuous selections but that do \emph{not}
assume knowledge of $h$ as in \Cref{ass:h}---that assume access to an oracle
capable of computing $f(x)$ and a(n arbitrary) subgradient $\xi(x)\in\partial
f(x)$ are plentiful. Most fundamentally, there exist counterparts of
classical gradient descent methods,
usually termed \emph{subgradient descent methods};  see~\cite{rockafellar2009variational}
for a textbook treatment of such methods.

An alternative class of methods for the solution of nonsmooth optimization
methods has been bundle methods; see~\cite{makela2001} for a survey.
Originally designed for \emph{convex} nonsmooth optimization, bundle methods
maintain a \emph{bundle} of cutting planes and iteratively solve a subproblem
involving a piecewise-affine model,
\begin{equation}
\label{eq:bundle_sp}
\displaystyle\minimize_{s\in\Reals^n}\displaystyle\max_{y\in Y_k} f(y) + \xi(y)^\top (x^k + s - y),
\end{equation}
where, in the notation developed so far, $Y_k\subseteq Y$ is algorithmically
determined and $\xi(y)$ still refers to the arbitrary subgradient returned by
an oracle. We observe that the inner maximization in~\eqref{eq:bundle_sp} is
clearly related to the affine functions involved in the primal model~\eqref{eq:primal_model}.
Notably, however, the
affine functions of~\eqref{eq:bundle_sp} are first-order models of $f$
centered at individual points $y\in Y_k$, whereas the affine models in~\eqref{eq:primal_model} are (approximate) first-order models of individual $f_j$
centered at $x^k$. That is, the piecewise-affine model~\eqref{eq:primal_model} of manifold sampling directly exploits knowledge of the
continuous selection, whereas the piecewise-affine model in~\eqref{eq:bundle_sp} does not.

While both subgradient and bundle methods were originally intended for convex
nonsmooth optimization, there exist extensions suitable for nonconvex
optimization. For more recent work in nonconvex subgradient descent methods,
see~\cite{Bagirov2008a,Bagirov2013}; for more recent work in nonconvex
bundle methods, see~\cite{hare2010redistributed,hare2016proximal,Kiwiel96}.

A third class of methods for the solution of noncomposite nonsmooth (nonconvex)
optimization problems given access to gradients is \emph{gradient sampling
methods}~\cite{burke2020gradient,burke2002approximating,burke2005robust,curtis2013adaptive,Kiwiel2007}. Gradient sampling methods are applicable to
the minimization of locally Lipschitz functions, a class of nonsmooth functions
broader than that analyzed in the present paper. By Rademacher's theorem, a
function that is locally Lipschitz on $\Reals^n$ admits a gradient almost everywhere,
and hence one can compute $\nabla f(x)$ for almost all $x\in\Reals^n$; we will
call the full measure set of differentiable points
$\mathcal{D}\subseteq\Reals^n$. At each iteration, gradient sampling methods
maintain a finite sample of gradients $\mathcal{G}^k$ at points of
differentiability in a ball, $\mathcal{G}^k\subset
\cB(x^k;\Delta_k)\cap\mathcal{D}$. We note that we intentionally overloaded
$\Delta_k$ to denote the sampling radius of gradient sampling, since there is a
connection to the $\Delta_k$ of manifold sampling via $\Gen^k$ in~\eqref{eq:gen_k}. The primal subproblem solved in each iteration of a gradient
sampling method to obtain a trial step $s^k$ is
\begin{equation}\label{eq:gs_sp}
\displaystyle\minimize_{s\in\Reals^n} f(x^k) + \displaystyle\max_{g\in\mathcal{G}^k} \left\{g^\top s + \frac{1}{2}s^\top H^k s\right\},
\end{equation}
where $H^k$ is a (positive-definite) matrix.
Importantly, the dual problem to~\eqref{eq:gs_sp} can be written as
\begin{equation}
    \label{eq:gs_sp_dual}
    \displaystyle\minimize_{\lambda\in\Reals^{|\mathcal{G}^k|}} \frac{1}{2}\|G_k\lambda\|_{[H^k]^{-1}}^2: e^\top\lambda = 1, \lambda\geq 0,
\end{equation}
where $G_k$ is a matrix with columns given by the sampled gradients in
$\mathcal{G}^k$. The dual problem~\eqref{eq:gs_sp_dual} reveals when $H^k
= I_n$, and as $\mathcal{G}^k$ becomes dense in $\cB(x^k;\Delta_k)$, the search
direction $s^k$ approaches a steepest descent direction for the locally
Lipschitz function $f$. It is clear that when we assume~\eqref{eq:prob_statement} is an unconstrained problem (so that $\lambda_\ell =
\lambda_u = 0$),~\eqref{eq:dual} similarly provides a steepest descent direction for
$f$. However, because continuous selections involve only finitely many
selection functions (those indexed by $\Hdef$), a manifold sampling method
never needs to sample densely in $\cB(x^k;\Delta_k)$ in order to obtain
arbitrarily good approximations to the steepest descent direction: only
finitely many selection functions in $\Hdef$ active at (or near) $x^k$ must ever be
sampled.

All three of these classes of methods (subgradient descent methods, bundle methods, gradient sampling methods) admit derivative-free variants, whereby
the subgradient oracle required for these methods is replaced with approximate
(e.g., finite-difference-based) approximations. Notably, in a derivative-free
setting, Bagirov et al.~\cite{Bagirov2007} proposed the so-called discrete
gradient method, which computes approximate subgradients for use in a
subgradient descent framework; see also~\cite{riis2018geometric}.
Derivative-free bundle methods have also been proposed;
see~\cite{karmitsa2012limited}. Even closer to our proposed method is a
derivative-free bundle method employing a trust region---as opposed to a
proximal point mechanism---proposed in~\cite{liuzzi2019trust}.
Kiwiel~\cite{KIWIEL2010} proposed a gradient sampling method for the
derivative-free setting by computing approximate (finite-difference) gradients.
Moreover, as a fourth class of methods most practical for DFO,
direct-search methods have been historically concerned with
convergence to Clarke stationary points and are hence suitable for nonsmooth
optimization; see the book~\cite{AudetHare2017} for an excellent treatment of
these methods.

Special attention has been paid to \emph{composite} optimization in the
literature, that is, for cases where $h$ in \Cref{ass:f} is known and an oracle
exists for computing $F(x)$ and $\nabla F(x)$. Works from the 1980s provide
fundamental analyses for the case where $h$ is convex and $\nabla F$ is
available~\cite{fletcher1982model,fletcher1982second,womersley1986algorithm,yuan1985conditions,yuan1985superlinear}. A bundle method for composite
nonsmooth optimization with convex $h$ was proposed
in~\cite{sagastizabal2013composite}. Recently, Bareilles et al.~\cite{bareilles2022} proposed a
peculiar algorithm that identifies a smooth manifold (see also
\cite{womersley1986algorithm}) via a subproblem involving a known proximal
operator for $h$ and then solves a second-order subproblem restricted to that
manifold. The use of the term \emph{manifold} in~\cite{bareilles2022} is more
formal than that employed here, and the assumptions imposed on $h$ and the
manifold structure it creates are nontrivial when compared with the generality and simplicity of
continuous selections.

In the derivative-free setting, the authors
in~\cite{garmanjani2016trust,grapiglia2016derivative} analyze algorithms for
composite optimization where $h$ is a general convex function but $\nabla F$ is
not available.
To the best of our knowledge,~\cite{KLW18} was the first
derivative-free work to consider a nonconvex $h$ (piecewise-linear),
while~\cite{Larson2020} was the first derivative-free work to consider a
general nonconvex continuous selection $h$.

\section{Analysis}
We now state the algorithms below and the provide 
assumptions (\Cref{sec:additional_assumptions}) used in the analyses of 
convergence of \mspshortref (\Cref{sec:msp_analysis}) and \goombahref
(\Cref{sec:goombah_analysis}).

\renewcommand{\thealgocf}{}
\SetAlgorithmName{\texttt{MS-P loop}}{}{}
\SetAlgoCaptionSeparator{:}
\begin{algorithm2e}[H]
  \caption{Manifold sampling primal loop \label{alg:manifold_sampling_loop}}
  \fontsize{8}{8}\selectfont
  \SetAlgoNlRelativeSize{-4}
  \DontPrintSemicolon
  \SetKw{true}{true}
  \SetKw{return}{return}
  \SetKwFunction{proc}{manifold\_sampling\_loop}
  \let\oldnl\nl%
  \newcommand{\nonl}{\renewcommand{\nl}{\let\nl\oldnl}}%
  \SetKwInOut{input}{input}
  \SetKwInOut{output}{output}
  \input{Iterate $x^k$, radius $\Delta_k$, set of points $Y$}
  \output{$x^{k+1}$, $\Delta_{k+1}$, $Y$}

  Set $c_1,c_2\geq 0$, $\kappabmh \geq 0$, $\eta_1>0$, $\eta_2 \in(0,1/\kappabmh)$, $0<\gammad < 1 <\gammai$

  Store $\bar{\Delta} \gets \Delta_k$

    \While{\true}{
    Ensure $M$ is a gradient-accurate model of $F$ on $\cB(x^k;\Delta_k)$ using $Y$ (this may require evaluating $F$ at additional points and adding them to $Y$)

    Construct $\Gen^k$ according to~\eqref{eq:gen_k} with $Y$, $\Delta_k$, $c_1$, and $c_2$

    Choose $H^k$ such that $\|H^k\|\leq\kappabmh$

    (Approximately) solve subproblem~\eqref{eq:equiv_subproblem} to obtain $(v_k,s^k)$

    Solve subproblem~\eqref{eq:chi} to obtain $\rchi_k$

    \If{$\Delta_k > \eta_2\rchi_k$ \label{line:acceptable}}{
      $x^{k+1}\gets x^k, \Delta_{k+1} \gets \gammad \bar{\Delta}$ \tcp*{unsuccessful iteration}

        \return
    }

    Evaluate $F(x^k + s^k)$ and let $Y\gets Y\cup x^k + s^k$

    Compute $\rho_k$ according to~\eqref{eq:rho}

    \eIf{$\rho_k \geq \eta_1$}
    {$x^{k+1}\gets x^k + s^k, \Delta_{k+1} \gets \gammai\bar{\Delta}$ \tcp*{successful iteration}

    \return
    }
    {
    Compute temporary $\bar{\Gen}^k$ according to~\eqref{eq:gen_k}

    \eIf{$\bar{\Gen}^k=\Gen^k$}
    {
    \eIf{$\Gen^k \bigcap \Act{F(x^k + s^k)} \neq \emptyset$}%
    {
    $x^{k+1}\gets x^k$, $\Delta_{k+1}\gets \gammad \bar{\Delta}$ \tcp*{unsuccessful iteration} \label{line:unsuccessful_break}
    \return 
    }
    {
        $\Delta_{k} \gets \gammad\Delta_k$ \label{line:decrease_delta}
    }
    }
    {
    $\Gen^k\gets\bar{\Gen}^k$ \label{line:grow_gen_set}
    }
    }
    }
\end{algorithm2e}

\SetAlgorithmName{\texttt{MS-P}}{}{}
\SetAlgoCaptionSeparator{:}
\begin{algorithm2e}[H]
  \caption{Manifold sampling (primal) for continuous selections \label{alg:manifold_sampling}}
  \fontsize{8}{8}\selectfont
  \SetAlgoNlRelativeSize{-4}
  \SetKw{true}{true}
  \SetKw{break}{break}
  \SetKw{return}{return}
  \SetKwFunction{proc}{\msploopref}
  \let\oldnl\nl%
  \newcommand{\nonl}{\renewcommand{\nl}{\let\nl\oldnl}}%

  Choose initial iterate and radius, $x^0, \Delta_0 > 0$

  Initialize history of evaluated points $Y$ (at least with $x^0$)

   \For{$k=0,1,2,\dots$}{
   $x^{k+1}$, $\Delta_{k+1}$, $Y \gets $\proc($x^k$, $\Delta^k$, $Y$)
    }
\end{algorithm2e}

\SetAlgorithmName{\texttt{GOOMBAH}}{}{}
\begin{algorithm2e}[H]
  \caption{\small Glassbox Optimization Of Model of Blackbox in A Hypersphere \label{alg:GOOMBAH}
  }
  \DontPrintSemicolon
  \fontsize{8}{8}\selectfont
  \SetAlgoNlRelativeSize{-4}
  \SetKw{true}{true}
  \SetKw{break}{break}
  \SetKw{continue}{continue}
  \SetKw{return}{return}
  \SetKwFunction{proc}{\msploopref}
  \let\oldnl\nl%
  \newcommand{\nonl}{\renewcommand{\nl}{\let\nl\oldnl}}%

  Choose initial iterate and radius, $x^0, \Delta_0 > 0$

  Set algorithmic constants $\tilde{\eta}_1>0$, $\omega > 0$

  Initialize history of evaluated points $Y$ (at least with $x^0$)

   \For{\label{line:start_of_for_loop_j}$k=0,1,2,\dots$}{
    Ensure $M$ is a gradient-accurate model of $F$ on $\cB(x^k;\Delta_k)$
    using $Y$ (this may require evaluating $F$ at additional points and adding them to $Y$)

      (Approximately) solve~\eqref{eq:GOOMBAH_subproblem} to produce a step $\sGOOMBAH^k$

      \If{$x^k + \sGOOMBAH^k\not\in Y$}
      {Evaluate $F(x^k+\sGOOMBAH^k)$ and let $Y \gets Y \cup x^k + \sGOOMBAH^k$

        Compute $\rhoGOOMBAH_k$ as in~\eqref{eq:rhoGOOMBAH} with $\omega$

      \If{$\rhoGOOMBAH_k > \tilde{\eta}_1$}{
      $x^{k+1}\gets x^k + \sGOOMBAH^k$, $\Delta_{k+1}\gets \gammai\Delta_k$

      \continue (to \Cref{line:start_of_for_loop_j})
      }
      }
      $x^{k+1}$, $\Delta_{k+1}$, $Y \gets $\proc($x^k$, $\Delta^k$, $Y$)
    }
\end{algorithm2e}

\subsection{Additional assumptions}\label{sec:additional_assumptions}

We first state additional regularity assumptions on $h$ and $F$ that we make in
order to provide rigorous convergence guarantees about \mspshortref.
We note that even if these assumptions were violated, \mspshortref
would be well defined, but convergence guarantees may not hold, even on a computer with
infinite precision.

We first assume some regularity conditions on $F$ using some constants from \Cref{def:constants}.

\begin{assumption}\label{ass:f2}\
  \begin{enumerate}[label=\tn{\textbf{\Alph*}.},ref=\tn{\Alph*},leftmargin=*]

    \item Each component $F_i$ of $F$ is Lipschitz continuous with some
      constant $K_{F_i}$. \label{subass:Fval}

    \item Each component $F_i$ of $F$ has a Lipschitz continuous gradient with constant $\LnFi$. \label{subass:F}

    \item For a point $x^0 \in \Reals^n$, assume the set $\Level(x^0) \defined \left\{ x: f(x) \le
      f(x^0) \right\}$ is bounded. \label{subass:level}

  \end{enumerate}
\end{assumption}

We additionally cast the following assumptions concerning selection functions that comprise the outer function $h$:
\begin{assumption}\label{ass:h_j}
  On $\hdom$, each $h_j\in\Hdef$ is Lipschitz continuous, is Lipschitz
  continuously differentiable, and has bounded gradients.
  That is, for all $z,z' \in \hdom$:
  \begin{enumerate}[label=\tn{\textbf{\Alph*}.},ref=\tn{\Alph*},leftmargin=*]
    \item There exists $K_{h_j}$ such that $|h_j(z) - h_j(z')|\leq K_{h_j}\|z - z'\|$. \label{subass:func}
    \item There exists $K_{\nabla h_j}$ such that $\|\nabla h_j(z) - \nabla h_j(z')\|\leq K_{\nabla h_j}\|z - z'\|$. \label{subass:grad}
    \item There exists $\maxhgrad$ such that $\|\nabla h_j(z)\| \leq \maxhgrad$.\label{subass:bounded_grads}
  \end{enumerate}
\end{assumption}

We define the following terms for convenience.
\begin{defi}\label{def:constants}
  For the constants in \subassref{f}{F},
  \Cref{def:flmodels}, and \Cref{ass:h_j}, define
  \begin{equation*}
    \begin{array}[h]{lll}
       \Lh \defined \ds \max_{j\in \{1,\ldots,\left| \Hdef \right|\}}\left\{ K_{h_j} \right\},
       & K_F \defined \sqrt{\sum_{i=1}^p K_{F_i}^2},
       & \LnF \defined \ds \max_{i\in \{1,\ldots,p\}} \left\{\LnFi \right\},
       \\
       \Lnh \defined \ds \max_{j\in \{1,\ldots,\left| \Hdef \right|\}}\left\{ K_{\nabla h_j} \right\},
       & \ds \kappag \defined \sum_{i=1}^p \kappaieg.
      \end{array}
    \end{equation*}
\end{defi}

Our final assumption requires that the method used to solve the subproblem~\eqref{eq:equiv_subproblem}
returns a solution that satisfies a Cauchy-like decrease property. This assumption is reasonable because there exist finite-time algorithms
for producing such an approximate solution (see, e.g.,~\cite{MMSMW2017} and~\cite[Section
12.2]{TRMbook})

\begin{assumption}
\label{ass:cauchy}
The approximate solution $(v_k,s^k)$ to~\eqref{eq:equiv_subproblem} satisfies the constraints of~\eqref{eq:equiv_subproblem}
and moreover satisfies
\begin{equation}
\label{eq:suff_dec}
-\left(v_k - f(x^k) + \frac{1}{2}s^{k\top}H^ks^k\right) \geq \kappafcd \rchi_k\min\left\{\displaystyle\frac{\rchi_k}{\kappabmh},\Delta_k,1\right\}
\end{equation}
for an algorithmic parameter $\kappafcd \in(0,1)$.
\end{assumption}

\subsection{Convergence of \texttt{MS-P}}\label{sec:msp_analysis}
We begin by demonstrating that the \msploopref must terminate finitely.
\begin{lemma}\label{lem:finite_termination}
  If \Crefrange{ass:f}{ass:h_j} hold, the \msploopref will terminate.
\end{lemma}
\begin{proof}
If either $\Delta_k > \eta_2\rchi_k$ or $\rho_k \leq \eta_1$, then the loop terminates. 
Suppose that for a given $k$, on every pass through the \msploopref,
$\Delta_k \leq \eta_2\rchi_k$ and $\rho_k < \eta_1$; hence, exactly one of  \Cref{line:decrease_delta} or \Cref{line:grow_gen_set}
will be reached on each pass. 
After finitely many (at most $\left| \Hdef \right|$) consecutive passes through the \msploopref, $\bar{\Gen}_k$ must equal $\Gen_k$ and so, \Cref{line:decrease_delta} will be reached, and $\Delta_k$ will be decreased. 
By \Crefrange{ass:h}{ass:h_j},
$f$ is
  piecewise-differentiable in the sense of Scholtes~\cite{Scholtes2012} so there exists $\tilde\Delta>0$ such that
for all $\Delta\leq\tilde\Delta$
$$\Act{F(x^k)} = \bigcup_{y\in\cB(x^k;\Delta)}\Act{F(y)}.$$
For all such $\Delta \le \tilde{\Delta}$, we also have that
$\Gen^k\cap\Act{F(x^k+s^k)}\neq \emptyset$, because $\Act{F(x^k)}\subset\Gen^k$ by \eqref{eq:gen_k}. 
Thus, once $\Delta_k \le \tilde\Delta$, iteration $k$ will be deemed unsuccessful and the loop will terminate.
\qed
\end{proof}

We now show that
 the convex hull of model gradients indexed by $\mathbb{I} \gets \Gen^k$
in
some sense approximates part of the Clarke subdifferential $\partialC f(x^k)$.
The proof is virtually identical to~\cite[Lemma 4.1]{KLW18} up to differences in notation and is included for completeness.

\begin{lemma}\label{lem:weird_v_approx}
  Let \Cref{ass:f} and \Cref{ass:f2},
  hold, and let $x,y\in\Levelmax$ satisfy
  $\|x-y\|\leq \Delta$.
  Choose arbitrary subsets $\mathbb{I} \subseteq \mathbb{J} \subseteq \{1,\ldots,\left| \Hdef \right|\}$, and define
  \[
  \cG \defined \cop{\{ \nabla M(x) \nabla h_i(F(x)) : i \in \mathbb{I}\}} \mbox{ and }
  \cH \defined \cop{\{ \nabla F(y) \nabla h_j(F(y)) : j \in \mathbb{J}\}},
  \]
  where $M$ is a gradient-accurate model of $F$ on $\cB(x;\Delta)$.
  Then for each $g\in\cG$, there exists some $\sigma(g)\in \cH$
  satisfying
  \begin{equation}
    \label{eq:gapprox}
    \left\|g-\sigma(g)\right\|\leq B\Delta,
  \end{equation}
  where $B$ is defined with the constants from \Cref{def:constants} by
  \begin{equation}\label{eq:c_2}
    B \defined \Lh (\LnF + \kappag) + K_F \Lnh.
  \end{equation}
\end{lemma}

\begin{proof}[Proof (adapted from {\cite[Lemma 4.1]{KLW18}})]
  Any $g\in\cG$ may be
expressed as
  \begin{equation}
    \label{weird_v_approx_g}
    g =  \ds\sum_{i \in \mathbb{I}} \lambda_i \nabla M(x) \nabla h_i(F(x)),
  \end{equation}
  where $\sum_{i \in \mathbb{I}} \lambda_{i}=1$ and $\lambda_i \geq 0$ for each
  $i$.

  By supposition, $ \nabla F(y) \nabla h_i(F(y)) \in \cH$ for all $i \in \mathbb{I}$.
  For
  \[
  \sigma(g) \defined \ds\sum_{i \in \mathbb{I}} \lambda_i \nabla F(y) \nabla h_i(F(y)),
  \]
  using the same $\lambda_i$ as in~\eqref{weird_v_approx_g} for $i \in
  \mathbb{I}$, convexity of $\cH$ implies that $\sigma(g)\in \cH$.
  Since $y\in\cB(x;\Delta)$, the triangle inequality,  \subassref{f2}{Fval}, \subassref{f2}{F}, \subassref{h_j}{func}, and the definition of gradient-accurate give
  \begin{align*}
    \left\| \nabla M(x) \nabla h_i(F(x)) - \nabla F(y) \nabla h_i(F(y)) \right\|
    \le&\; \left\| \nabla F(y) - \nabla F(x) \right\| \left\| \nabla
    h_i(F(y))\right\| \\
    &\;+ \left\| \nabla F(x) \right\| \left\| \nabla h_i(F(x)) - \nabla h_i(F(y)) \right\| \\
    &\;+ \left\|  \nabla F(x)  - \nabla M(x) \right\| \left\| \nabla h_i(F(x))\right\| \\
    \le&\; (\Lh \LnF + K_F K_{\nabla h_i} + \kappag \Lh) \Delta,
  \end{align*}
  for each $i$.
  Using this along with~\eqref{weird_v_approx_g} and the
  definition of $\sigma(g)$ yields
  \begin{align*}
    \left\|g-\sigma(g)\right\|
    &\leq \left\|  \ds\sum_{i \in \mathbb{I}} \left[ \lambda_i \nabla
M(x) \nabla h_i(F(x)) - \lambda_i \nabla F(y) \nabla h_i(F(y)) \right] \right\|\\
&\le  \ds\sum_{i \in \mathbb{I}}\lambda_i \left\|\nabla M(x) \nabla h_i(F(x)) - \nabla
F(y) \nabla h_i(F(y)) \right\|
\le B \Delta. \tag*{\qed}
  \end{align*}
\end{proof}

We next provide a brief proposition bounding the distance between values of distinct selection functions.
\begin{proposition}
\label{prop:selection_func_change}
Let \Crefrange{ass:f}{ass:h} hold and
suppose on iteration $k$ of \mspshortref that $\bar{j}\in\Gen^k$ and $f_{\bar{j}}(x^k) > f(x^k)$. Then,
$$|f(x^k) - f_{\bar{j}}(x^k)|\leq 2 K_h c_1\Delta_k^2.$$
\end{proposition}

\begin{proof}
Because $j\in\Gen^k$, there exists some corresponding $\bar{y}\in Y\cap \cB(x^k;c_1\Delta_k^2)$ so that $f(\bar{y}) = f_{j}(\bar{y})$.
Then,
\begin{align*}
|f(x^k) - f_{j}(x^k)| & \leq |f(x^k) - f(\bar{y})| + |f(\bar{y}) - f_{j}(x^k)|\\
& = |f(x^k) - f(\bar{y})| + |f_{j}(\bar{y}) - f_{j}(x^k)| \\
& \leq 2 K_h \|x^k-\bar{y}\| \leq 2 K_h c_1\Delta_k^2. \tag*{\qed}
\end{align*}
\end{proof}

The next lemma demonstrates that iteration $k$ of \mspshortref will be successful if the
trust region $\Delta_k$ is sufficiently small with respect to the stationarity
measure $\rchi_k$.

\begin{lemma}\label{lem:success}
  Let  \Crefrange{ass:f}{ass:cauchy} hold, and define
  \[
  \bar{C} \defined \displaystyle\frac{\kappafcd (1-\eta_1)}{\LnF\Lh + \kappag \maxhgrad+\frac{\kappabmh}{2} + 2c_1 K_h}.
  \]
  If in the $k$th iteration of \mspshortref, $\Act{x^k+s^k}\cap\Gen^k\neq\emptyset$, and
  \begin{equation}\label{eq:delta_bound_linear}
    \Delta_k \leq \min\left\{1,C\rchi_k\right\},
  \end{equation}
  where
  \begin{align}  \label{eq:C}
  C \defined  \min\left\{ \eta_2, \bar{C}, \sqrt{\kappabmh^{-1} \bar{C}} \right\},
  \end{align}
  then $\rho_k \geq \eta_1$ in \mspshortref. That is, the $k$th iteration is successful.
\end{lemma}

\begin{proof}
  
  Let $j\in\Act{x^k+s^k}\cap\Gen^k$ be arbitrary, if the intersection is not a singleton.
  Then there are two cases to analyze:

  \noindent\textbf{Case 1 $h_{j}(F(x^k)) \leq f(x^k)$:}
  In this case, $\beta_{j,k} = 0$
  and so
  \begin{align*}
      1-\rho_k &=
      \displaystyle\frac{f(x^k)-v_k - \frac{1}{2}s^{k\top}H^k s^k-[f(x^k)-f(x^k+s^k)]}{f(x^k)-v_k - \frac{1}{2}s^{k\top}H^k s^k}\\
      &=
      \displaystyle\frac{f(x^k+s^k) -v_k - \frac{1}{2}s^{k\top}H^ks^k}{f(x^k)-v_k-\frac{1}{2}s^{k\top}H^ks^k}\\
      &=\displaystyle\frac{f(x^k+s^k) - f_{j}(x^k) - g_{j}^\top s^k - \frac{1}{2}s^{k\top}H^ks^k}{f(x^k)-v_k-\frac{1}{2}s^{k\top}H^ks^k}\\
      &=\displaystyle\frac{h_{j}(F(x^k+s^k)) - (h_{j}(F(x^k)) + [\nabla M(x^k)\nabla h_{j}(F(x^k))]^\top s^k) -  \frac{1}{2}s^{k\top}H^ks^k}{f(x^k)-v_k-\frac{1}{2}s^{k\top}H^ks^k},
  \end{align*}
  where the last equality is because $j\in\Act{x^k+s^k}$.
  Note that
  \begin{align*}
    & h_{j}(F(x^k+s^k)) - (h_{j}(F(x^k)) + [\nabla M(x^k)\nabla h_{j}(F(x^k))]^\top s^k)\\
    = &\; h_{j}(F(x^k+s^k)) - h_{j}(F(x^k)) - [\nabla M(x^k)\nabla h_{j}(F(x^k))]^\top s^k \\
    & + [\nabla F(x^k)\nabla h_{j}(F(x^k))]^\top s^k - [\nabla F(x^k)\nabla h_{j}(F(x^k))]^\top s^k\\
    \leq & \; h_{j}(F(x^k+s^k)) - (h_{j}(F(x^k)) + [\nabla F(x^k)\nabla h_{j}(F(x^k))]^\top s^k) \\
    & + \|\nabla F(x^k) - \nabla M(x^k)\|\|\nabla h_{j}(F(x^k))\|\Delta_k \\
    \leq &\; \LnF \Lnh\Delta_k^2 + \kappag \maxhgrad\Delta_k^2 = (\LnF \Lnh + \kappag \maxhgrad)\Delta_k^2,
  \end{align*}
  where the last inequality is due to Taylor's theorem, \subassref{f2}{F}, \subassref{h_j}{grad}, \subassref{h_j}{bounded_grads},
  and the definition of gradient-accurate.
  Thus, continuing,
  \begin{align*}
    1-\rho_k  \leq \displaystyle\frac{(\LnF\Lnh + \kappag \maxhgrad + \frac{\kappabmh}{2})\Delta_k^2}{ f(x^k)-v_k-\frac{1}{2}s^{k\top}H^ks^k}
      \leq &\; \displaystyle\frac{(\LnF\Lnh + \kappag \maxhgrad + \frac{\kappabmh}{2})\Delta_k^2}{\kappafcd \rchi_k\min\left\{\displaystyle\frac{\rchi_k}{\kappabmh},\Delta_k,1\right\}}\\
      \leq &\; \displaystyle\frac{(\LnF\Lnh + \kappag \maxhgrad + \frac{\kappabmh}{2})\Delta_k^2}{\kappafcd \displaystyle\frac{\Delta_k}{C}\min\left\{\displaystyle\frac{\Delta_k}{C\kappabmh},\Delta_k\right\}}\\
      = &\; \displaystyle\frac{(\LnF\Lnh + \kappag \maxhgrad + \frac{\kappabmh}{2})}{\kappafcd \displaystyle\frac{1}{C}\min\left\{\displaystyle\frac{1}{C\kappabmh},1\right\}}\\
     \leq &\; 1-\eta_1,
        \end{align*}
  where the second inequality uses \Cref{ass:cauchy} and the last inequality uses the fact that $\Delta_k\leq\min\left\{1,C\Delta_k\right\}$.
  Thus, in this case, $\rho_k \geq \eta_1$, and the $k$th iteration is successful.

  \noindent\textbf{Case 2 $h_{j}(F(x^k)) > f(x^k)$:} In this case, $\beta_{j,k} = h_{j}(F(x^k)) - f(x^k) > 0$.
 Thus,
 \begin{align*}
 1-\rho_k
 = & \;\displaystyle\frac{f(x^k+s^k) -v_k - \frac{1}{2}s^{k\top}H^ks^k}{f(x^k)-v_k-\frac{1}{2}s^{k\top}H^ks^k}\\
 = & \;\displaystyle\frac{f(x^k+s^k) - h_{j}(F(x^k)) - g_{j}^\top s^k + h_{j}(F(x^k)) - f(x^k) - \frac{1}{2}s^{k\top}H^ks^k}{f(x^k)-v_k-\frac{1}{2}s^{k\top}H^ks^k}\\
 = & \;\displaystyle\frac{f(x^k+s^k) - h_{j}(F(x^k)) - g_{j}^\top s^k - \frac{1}{2}s^{k\top}H^ks^k +h_{j}(F(x^k)) - f(x^k) }{f(x^k)-v_k-\frac{1}{2}s^{k\top}H^ks^k}\\
  \leq &\; \displaystyle\frac{(\LnF\Lnh + \kappag \maxhgrad + \frac{\kappabmh}{2} + 2c_1 K_h)\Delta_k^2}{\kappafcd \displaystyle\frac{\Delta_k}{C}\min\left\{\displaystyle\frac{\Delta_k}{C\kappabmh},\Delta_k\right\}}\\
 \leq &\; \displaystyle\frac{(\LnF\Lnh + \kappag \maxhgrad + \frac{\kappabmh}{2} + 2c_1 K_h)}{\kappafcd \displaystyle\frac{1}{C}\min\left\{\displaystyle\frac{1}{C\kappabmh},1\right\}}\\
 = &\; 1-\eta_1,
 \end{align*}
 where the second-to-last inequality uses the fact that $j\in\Gen^k$ and \Cref{prop:selection_func_change}.
  Thus, in this case, $\rho_k\geq\eta_1$ if $\Delta_k$ satisfies~\eqref{eq:delta_bound_linear},
  and iteration $k$ is again successful.
\qed
\end{proof}

The next lemma shows that the sequence of
trust-region radii
converges to zero.

\begin{lemma}\label{lem:delta_to_0}
  Let \Crefrange{ass:f}{ass:cauchy} hold. If
$\{x^k,\Delta_k\}_{k\in\Integers}$ is generated by
  \mspshortref, then
  $\ds\lim_{k\to\infty}\Delta_k=0$.
\end{lemma}

\begin{proof}
  If iteration $k$ is unsuccessful, then $\Delta_{k+1}<\Delta_k$, and
  $x^{k+1}=x^k$; therefore, $f(x^{k+1})=f(x^k)$. On successful iterations $k$,
  $\rho_k\geq\eta_1>0$ ensures that $f(x^{k+1})<f(x^k)$.
  Thus, the sequence $\{f(x^k)\}_{k\in\Integers}$ is nonincreasing.

  To show that $\Delta_k \to 0$, we consider the cases in which there
  are infinitely or finitely many successful iterations separately.
  First, suppose that there are infinitely many successful iterations, indexed by $\left\{
  k_j \right\}_{j\in\Integers}$. Since $f(x^k)$ is nonincreasing in
  $k$ and $f$ is bounded below
  (by \subassref{f2}{level}), the sequence $\{f(x^k)\}_{k\in\Integers}$ converges
  to some limit $f^*\leq f(x^0)$.  Thus,
  having infinitely
  many successful iterations (indexed $\{k_j\}_{j \in \Integers}$) implies that
  \begin{align}\label{eq:intermediate}
      \infty > f(x^0) - f^* \ge & \ds\sum_{j=0}^\infty f(x^{k_j}) - f(x^{k_j+1}) \nonumber \\
      > & \ds\sum_{j=0}^\infty\eta_1\kappafcd \rchi_{k_j}\min\left\{\displaystyle\frac{\rchi_{k_j}}{\kappabmh},\Delta_{k_j},1\right\} \nonumber \\
      > & \ds\sum_{j=0}^\infty \frac{\eta_1}{\eta_2}\kappafcd\Delta_{k_j} \min\left\{\displaystyle\frac{\Delta_{k_j}}{\eta_2\kappabmh},\Delta_{k_j},1\right\},
  \end{align}
  where the second-to-last inequality is due to the definition of success and \Cref{ass:cauchy}
  and the last inequality is because every successful iteration must satisfy $\Delta_{k_j}\leq \eta_2\rchi_{k_j}$.
 We note that if $1$ were the minimizer infinitely often in the right-hand side
 of~\eqref{eq:intermediate}, then a contradiction would immediately result,
 because this would imply $\rchi_{k_j} > \kappabmh$ for all such infinitely many
 $j$, violating the finiteness of the sum.
Thus, we conclude from~\eqref{eq:intermediate} that
\begin{equation*}
    \infty > \ds\sum_{j=0}^\infty \frac{\eta_1}{\eta_2}\kappafcd\Delta_{k_j}^2 \min\left\{\displaystyle\frac{1}{\eta_2\kappabmh},1\right\}.
\end{equation*}
  It follows that $\Delta_{k_j} \to 0$ for the sequence of
  successful iterations.
Observe that while multiple
decreases of $\Delta_k$ may occur inside the \msploopref during the $k$th iteration of \mspshortref,
the presence of $\bar\Delta$ in the \msploopref ensures that $\Delta_{k+1}\in\{\gammad\Delta_k,\gammai\Delta_k\}$. 
Hence,
  $\Delta_{k_j+1}\le\gammai\Delta_{k_j}$, and moreover,
  $\Delta_{k+1}=\gammad\Delta_k<\Delta_k$ if iteration $k$ is
  unsuccessful.  Thus, for any unsuccessful iteration $k>k_j$,
  $\Delta_k\leq\gammai\Delta_q$, where $q\defined\max\{k_j:j\in\Integers,\,k_j<k\}$.
  It follows immediately that
\[
  0\leq \lim_{k\to\infty}\Delta_k \leq
  \gammai\lim_{j\to\infty}\Delta_{k_j} = 0,
\]
and so $\Delta_k \to 0$ as required.

Next, suppose there are only finitely many successful iterations;
let $N \in\Integers$ be the number of successful iterations.
Since $\gammad<1\leq\gammai$, it follows that $0\leq\Delta_k\leq
\gammai^N \gammad^{k-N}\Delta_0$ for each $k\in\Integers$. Thus,
$\Delta_k \to 0$.
\qed
\end{proof}

We now show that the sequence $\{\rchi_k\}$ is not bounded away
from zero.
\begin{lemma} \label{lem:g_to_0_linear}
  Let \Crefrange{ass:f}{ass:cauchy} hold. If the sequence
  $\{x^k,\Delta_k,\rchi_k\}_{k\in\Integers}$ is generated by
  \mspshortref, then
  $\ds\liminf_{k \to \infty} \rchi_k = 0$.
\end{lemma}

\begin{proof}
  To obtain a contradiction, suppose there is an iteration $\bar{K}$ and some
  $\epsilon>0$ for which $\rchi_k\geq\epsilon$, for all $k\geq \bar{K}$.
  Any iteration $k\geq \bar{K}$ 
  that witnesses both
  $\Delta_k\leq C\rchi_k$ and 
  $\Gen_k
  \bigcap \Act{x^k + s^k}\neq\emptyset$ is guaranteed to be successful by \Cref{lem:success}, and so \Cref{line:unsuccessful_break} of the \msploopref cannot be reached. 
  Coupled with the fact that the \msploopref must terminate eventually (\Cref{lem:finite_termination}), we have that for all $k\geq \bar{K}$ such that $\Delta_k\leq C\rchi_k$, the \msploopref must yield a successful iteration, and so 
  $\Delta_{k+1} \ge \gammai \bar{\Delta} > \Delta_k$. 

  Therefore, $\Delta_k > \gammad C \epsilon$ for all $k\geq\bar{K}$,
  contradicting \Cref{lem:delta_to_0}.  Thus, no such $(\bar{K},\epsilon)$ pair
  exists, and so $\ds \liminf_{k\to\infty}\rchi_k=0$.
  \qed
\end{proof}

The next lemma shows that any convergent subsequence $\{k_j\}$ of iterates on which $\rchi_{k_j} \to 0$
 admits a Clarke stationary cluster point.
\mspshortref generates at least one
such subsequence of iterates by \Cref{lem:g_to_0_linear}.

\begin{lemma}
\label{lem:g_to_v_linear}
Let \Crefrange{ass:f}{ass:cauchy} hold.
Let
$\{x^k,\Delta_k,g^k\}_{k\in\Integers}$ be a sequence
generated by \mspshortref.
For any subsequence $\{k_j\}_{j\in\Integers}$
 iterations such
that both
\[
\lim_{j\to\infty}\rchi_{k_j}=0,
\]
and $\{x^{k_j}\}_{j\in\Integers} \to x^*$ for some cluster point $x^*$, then $0 \in \partialC
\Lag(x^*)$,
\jefflabel{Lagdef}
where $\Lag(x)$ is the Lagrangian of~\eqref{eq:prob_statement},
$$\Lag(x) \defined f(x) + \lambdal^\top(\lb-x) + \lambdau^\top(x-\ub).$$
\end{lemma}

\begin{proof}
  Let $\mathbb{I}^k \defined \Gen^k$, and let $\mathbb{J}^k\defined \Act{F(x^*)}$.  Because
  \begin{itemize}
  \item $\Delta_k \to 0$ by
  \Cref{lem:delta_to_0},
  \item $\{x^{k_j}\}_{j\in\Integers}$
  converges to $x^*$ by assumption,
  \item $f$ is
  piecewise-differentiable~\cite{Scholtes2012} due to \Crefrange{h}{h_j},
  and \item the definition of $\Gen^k$ in~\eqref{eq:gen_k},
  \end{itemize}
 we conclude that for $j$ sufficiently
  large, only selection functions that are essentially active at $x^*$ are represented in
  $\Gen^{k_j}$.
  Consequently, $\mathbb{I}^{k_j} \subseteq \mathbb{J}^{k_j}$ for all $j$ sufficiently large.
  Hence, in the definition of $\rchi_k$~\eqref{eq:chi}, we see that $a^{k_j} = 0$ for all $j$ sufficiently large.

  Let $(\lambdaa^{k_j},\lambdal^{k_j},\lambdau^{k_j})$ denote the minimizing
  $(\lambdaa,\lambdal,\lambdau)$ in the definition of $\rchi_k$~\eqref{eq:chi},
  and define
  $$g^{k_j}\triangleq G^{k_j}\lambdaa^{k_j}.$$
  By \Cref{lem:weird_v_approx} with $\mathbb{I} \gets \mathbb{I}^{k_j}$,
  $\mathbb{J}\gets \mathbb{J}^{k_j}$, $x \gets x^{k_j}$, $y \gets x^*$, and
  $\Delta\gets\Delta_{k_j}$,
  there exists $\sigma(g^{k_j})\in\partialC f(x^*)$ for each $g^{k_j}$ so that
  \[
  \|g^{k_j}-\sigma(g^{k_j})\|\leq B\Delta_{k_j},
  \]
  with $B$ defined by~\eqref{eq:c_2}.
Thus,
  $$\|g^{k_j}-\lambdal^{k_j} + \lambdau^{k_j} -(\sigma(g^{k_j}) - \lambdal^{k_j} + \lambdau^{k_j})\|\leq B\Delta_{k_j},$$
and so
  \[
  \|\sigma(g^{k_j}) - \lambdal^{k_j} + \lambdau^{k_j}\|\leq B\Delta_{k_j} + \|g^{k_j}-\lambdal^{k_j} + \lambdau^{k_j}\| = B\Delta_{k_j} + \rchi_{k_j},
  \]
  where the latter equality is true for all $j$ sufficiently large since we have shown that $a^{k_j}=0$ for $j$ sufficiently large.
  Since $\rchi_{k_j}\to 0$ and $\Delta_{k_j}\to 0$ by assumption, we conclude that $\left\| \sigma(g^{k_j}) -\lambdal^{k_j} + \lambdau^{k_j} \right\|
  \to 0$.
  Because $\sigma(g^{k_j}) -\lambdal^{k_j} + \lambdau^{k_j} \in \partialC \Lag(x^{k_j})$,
  Proposition~7.1.4 in~\cite{Facchinei2003}  yields the
  claimed result, by establishing that $\partialC \Lag$ is
  \emph{outer-semicontinuous} and therefore $0 \in \partialC \Lag(x^*)$.
  \qed
\end{proof}

We can now present our final result: that the limit of any subsequence of
\mspshortref iterates is a Clarke stationary point of the Lagrangian of~\eqref{eq:prob_statement}.

\begin{theorem} \label{thm:cluster_linear}
  Let \Crefrange{ass:f}{ass:cauchy} hold. If
  $x^*$ is a cluster point of a sequence $\{x^k\}$ generated by
  \mspshortref, then $0\in\partialC \Lag(x^*)$.
\end{theorem}

\begin{proof}
  First, suppose that there are only finitely many successful iterations and
  $k'$ is the last.
  Suppose for contradiction that $0\notin\partialC \Lag(x^{k'})$.
  By continuity of $F_i$ (\Cref{ass:f}), there exists
  $\bar\Delta > 0$ so that for all $\Delta\in [0,\bar\Delta]$, the manifolds
  active in $\cB(x^{k'};\bar\Delta)$ are precisely the manifolds active at
  $x^{k'}$; that is,
  \[
  \Act{F(x^{k'})} = \bigcup_{y \in \cB(x^{k'};\Delta)} \Act{F(y)} \qquad
    \mbox{ for all } \Delta \le \bar{\Delta}.
  \]

  By assumption,
  $\Delta_k$ decreases by a factor of $\gammad$ in each iteration after
  $k'$ because every iteration after $k'$ is unsuccessful. Thus there
  is a least iteration $k''\geq k'$ so that
  $\Delta_{k''} \le \bar\Delta$.  By the definition of $\Gen^k$ in~\eqref{eq:gen_k}, for each
  $k\geq k''$, $\Gen^k$ contains all
  manifolds at $x^{k'}$, and therefore
  $\nabla M(x^k) \nabla h_j(F(x^k)) \in\Gen^k$ for all $j\in\Act{F(x^k)}$.
  Since $k'$ is the last
  successful iteration, $x^{k} = x^{k'}$ for all $k \ge k'' \ge k'$.
  Consequently,
  the conditions for \Cref{lem:weird_v_approx} hold for $x\gets x^{k}$, $y \gets
x^{k'}$ (noting that $x^k = x^{k'}$) $\Delta\gets 0$, $\Gen \gets \Gen^k$, and
  $\cH \gets \partialC f(x^{k'})$; thus, for each $k\geq k''$,
  $g^k - \lambdal^k + \lambdau^k \in\partialC \Lag(x^{k'})$.

  Since $0\notin\partialC \Lag(x^{k'})$ by supposition,
  $\pi^*\defined\proj{0}{\partialC \Lag(x^{k'})}$ \jefflabel{pi_def} is nonzero, and so
  \begin{equation}\label{eq:v_approx_conclusion_linear}
    \|g^k - \lambdal^k + \lambdau^k \|\geq\|\pi^*\|>0 \qquad
      \mbox{ for all } k\geq k''.
  \end{equation}
  Since $\Delta_k \to 0$, $\Delta_k$ will satisfy the conditions of
  \Cref{lem:success} for $k$ sufficiently large: there will be a
  successful iteration contradicting $k'$ being the last.

  Next, suppose there are infinitely many successful iterations.
  We will demonstrate that there exists a subsequence of successful iterations
  $\{k_j\}$ that simultaneously satisfies both
  \begin{equation}\label{eq:two_conditions}
    x^{k_j}\to x^* \mbox{ and } \|g^{k_j} -\lambdal^{k_j} + \lambdau^{k_j}\|\to 0.
  \end{equation}
  If the sequence $\{x^k\}_{k\in\Integers}$ converges, then the subsequence
$\left\{ x^{k_j} \right\}_{j\in\Integers}$ from \Cref{lem:g_to_0_linear}
  satisfies~\eqref{eq:two_conditions}.
  Otherwise, if the sequence $\{x^k\}_k$ is not convergent, we will show that
\linebreak[4] $\liminf_{k\to\infty}(
  \max\{\|x^k-x^*\|,\|g^k-\lambdal^{k} + \lambdau^{k}\|\})=0$ for each cluster point $x^*$. Suppose for
  contradiction that there exist $\bar\theta>0$, an iteration $\bar{k}$,
  and a cluster point $x^*$ of the sequence $\{x^k\}$ with the
  following property:
given the infinite set \jefflabel{specialK}
  \[
  \mathbb{K} \defined \{k: k\geq\bar{k}, \|x^k-x^*\|\leq \bar\theta\}, 
  \]
  the subsequence $\left\{ x^k \right\}_{k \in \mathbb{K}}$ converges
  to $x^*$ and $\|g^k-\lambdal^{k} + \lambdau^{k}\| > \bar\theta$
  for all $k\in \mathbb{K}$. Thus,
  \begin{equation}
    \label{eq:finite_sum_linear}
    \eta_1\displaystyle\sum_{k\in \mathbb{K}} \|g^k-\lambdal^{k} + \lambdau^{k}\|\|x^{k+1}-x^k\| \leq
    \eta_1\displaystyle\sum_{k=0}^\infty \|g^k-\lambdal^{k} + \lambdau^{k}\|\|x^{k+1}-x^k\|
    < \infty,
  \end{equation}
  since on successful iterations, $\|x^{k+1}-x^k\|\leq \Delta_k$, while on
  unsuccessful iterations, $\|x^{k+1}-x^k\|=0$.  Since $\|g^k-\lambdal^{k} + \lambdau^{k}\|>\bar\theta$ for
  all $k\in \mathbb{K}$, we conclude from~\eqref{eq:finite_sum_linear} that
  \begin{equation}
    \label{eq:finite_sum2_linear}
    \displaystyle\sum_{k\in \mathbb{K}}\|x^{k+1}-x^k\| < \infty.
  \end{equation}

  Since $x^k\not\to x^*$, there exists some $\hat\theta\in(0,\bar\theta)$
  for which, for each $k' \in
  \mathbb{K}$, there exists
  \[
    q(k')\defined \min\{\kappa\in\Integers:\kappa>k',\quad \|x^{\kappa}-x^{k'}\| > \hat\theta\}.
  \]
  From this construction, since $\hat\theta<\bar\theta$, then $\{ k',
  k' + 1, \ldots, q(k')-1\} \subset \mathbb{K}$.

  By~\eqref{eq:finite_sum2_linear}, for $\hat\theta$ there exists $N \in \Integers$ such that
  \[
  \sum_{\substack{k \in \mathbb{K} \\ k \ge N}} \left\| x^{k+1} - x^k \right\| \le
  \hat\theta.
  \]

  Taking $k' \ge N$, by the triangle inequality, we have
  \begin{equation*}
    \label{eq:triangle_ineq_linear}
    \hat\theta < \|x^{q(k')}-x^{k'}\| \leq
    \displaystyle\sum_{i\in\{k',k'+1,\dots, q(k')-1\}} \|x^{i+1}-x^i\| \le
  \sum_{\substack{k \in \mathbb{K} \\ k \ge N}} \left\| x^{k+1} - x^k \right\| \le
  \hat\theta.
  \end{equation*}
  Therefore, $\hat\theta < \hat\theta$, a contradiction. Therefore $\liminf_{k\to\infty}
  (\max\{\|x^k-x^*\|,\|g^k-\lambdal^{k} + \lambdau^{k}\|\})=0$ for all cluster points $x^*$, and there is a
  subsequence satisfying~\eqref{eq:two_conditions}. By \Cref{lem:g_to_v_linear},
  $0\in\partialC \Lag(x^*)$ for all such subsequences.
  \qed
\end{proof}

\subsection{Convergence of \texttt{GOOMBAH}}\label{sec:goombah_analysis}
Because \goombahref essentially reverts to a manifold sampling step
whenever the (approximate) solution to~\eqref{eq:GOOMBAH_subproblem} does not provide sufficient decrease according to~\eqref{eq:rhoGOOMBAH}, \goombahref retains all the same convergence properties
as guaranteed by \Cref{thm:cluster_linear} for \mspshortref. To see why~\eqref{eq:rhoGOOMBAH} works as a sufficient decrease
measure, look at~\eqref{eq:intermediate} in the proof of
\Cref{lem:delta_to_0}. Partition the infinite set of successful iterations
$\{k_j\}$ into the set of successful iterations that are solutions of~\eqref{eq:GOOMBAH_subproblem}, $\succh$, and the set of successful iterations
from the manifold sampling loop, $\succm$. That is, $\{k_j\} = \succh \cup
\succm$. With this partition,
  \begin{align*}
      \infty &
      > f(x^0) - f^*
      \ge \ds\sum_{j=0}^\infty f(x^{k_j}) - f(x^{k_j+1}) \\
     & > \ds\sum_{k \in\succh}\tilde\eta_1\Delta_k^{1+\omega} +  \ds\sum_{k\in\succm}\eta_1\kappafcd \rchi_{k}\min\left\{\displaystyle\frac{\rchi_{k}}{\kappabmh},\Delta_{k},1\right\}\\
      & \geq \ds\sum_{k\in\succh}\tilde\eta_1\Delta_k^{1+\omega} + \ds\sum_{k\in\succm} \eta_1\kappafcd \min\left\{\frac{\Delta_{k}^2}{C^2\kappabmh},\frac{\Delta_{k}^2}{C} \right\}.
  \end{align*}
Regardless of the cardinalities of $\succh$ and $\succm$ (both infinite, or
exactly one infinite), we still conclude that $\Delta_{k_j}\to 0$, and so the
proof of \Cref{lem:delta_to_0} still follows.
The remaining proofs are unaffected.

\section{Testing}
We now discuss the performance of implementations of the numerical optimization
methods presented in this manuscript.
We compare \mspshortref, \goombahref, and the previous manifold
sampling code, which we denote \msdshortref because it employs the dual model~\eqref{eq:dual_model}.
In inspecting some \goombahref runs, we observed that the recourse to using the
\msploopref occurred on relatively few iterations. This motivates the inclusion
in our set of benchmarked implementations of a modified \goombahref that does not
resort to any manifold sampling logic but instead shrinks $\Delta_k$ on
iterations in which
\begin{equation*}
\frac{h(F(x^k))-h(F(x^k+s^k))}{h(M(x^k)) - h(M(x^k+s^k))} \le \eta_1.
\end{equation*}

The theory and implementation of \msdshortref have
been developed only for unconstrained problems. And, because of the
relatively poor performance of \msdshortref shown below, we did not seek to
extend \msdshortref to address bound-constrained problems.
We also tested but do not present the (relatively poor)
performance of other general-purpose nonsmooth optimization methods for both
bound-constrained and unconstrained problems. Because
such methods do not exploit the composite problem structure, their performance
was understandably poor, and we find such a comparison to be unfair.

\subsection{Test problems}
We test the four manifold sampling implementations on problems of the form
\[
\minimize_{x \in \Omega} h(F(x))
\]
with variously defined $\Omega$, $h$,  $F$, and starting points $x^0$. The
specific values of $F$ and $x^0$ are defined by the 53 vector-mapping problems in the
Mor\'{e}--Wild~\cite{JJMSMW09} benchmark set; the dimension of the domain
$F$ is between 2 and 12 and its output is between 2 and 65 dimensions. The
mappings of $F$ are smooth with known gradients (which are used only for
benchmarking purposes).

For $h$ we consider the following four nonsmooth\jefflabel{Qdef} mappings:
\begin{equation*}
\begin{aligned}[c]
h_1 & \defined  \min_i z_i^2\\
h_2 & \defined  \max_i z_i^2
\end{aligned}
\qquad\qquad
\begin{aligned}[c]
h_3 & \defined \sum_{i=1}^p \left|d_i - \max\left\{z_i,c_i\right\} \right| \\
h_4 & \defined \max_{i \in \{1,\ldots,l\}} \left\{ \left\| z - z_i \right\|_{Q_i}^2 + b_i \right\}.
\end{aligned}
\end{equation*}

The $h_3$ mapping is the piecewise-linear, censored-L1 loss
function~\cite{Womersley1986a} that measures how far $z_i$ is from target data
$d_i$, but only if $z_i$ is more than the censor value $c_i$. The values for
$c_i$ and $d_i$ are randomly generated for each for a given $(F, x^0)$ pair
following the approach outlined in~\cite[Section~5.1]{KLW18}. Similarly, the $l$ quadratics
defining the nonconvex mapping $h_4$ are randomly generated for each $(F,x^0)$
pair following the approach in~\cite[Section~6.2]{Larson2020}. While 10
instances of $h_3$ and 20 instances of $h_4$ were originally generated for each
$(F,x^0)$ pair, we used only the first instance for the current benchmark
studies.  These four $h$ mappings have known subdifferentials.

We consider an unconstrained setting ($\Omega \defined
\Reals^n$) and a bound-constrained setting $\Omega \defined \left\{ x: \lb \le x \le \ub \right\}$, where
$\lb$ and $\ub$ are defined for each $(F,x^0)$ pair via the following procedure.
We first run all the unconstrained experiments and record for each $(F,x^0)$ pair the approximate minimizer
$\tilde{x}$ from among the three solvers $\{\mspshortref, \goombahref, \goombahref \text{  w/o } \mspshortref\}$
that minimizes $h(F(x))$.
We then identify the midpoint $x^{\mathrm{mid}}$ on the line segment $[x^0,\tilde{x}]$ and define, coordinate-wise for $i=1,\dots,n$,
\begin{equation}\label{eq:testbounds}
\lb_i = x^0_i - \max\{x^0_i - x^{\mathrm{mid}}_i, x^{\mathrm{mid}}_i - x^0_i \}
\quad
\ub_i = x^0_i + \max\{x^0_i - x^{\mathrm{mid}}_i, x^{\mathrm{mid}}_i - x^0_i \}.
\end{equation}
Choosing the bounds as in~\eqref{eq:testbounds} guarantees that on at least one solver per $(F,x^0)$ pair, at least one bound must become active at some point in a run.

Thus, in total, we have 424 benchmark problem instances from $53$ choices of $(F,x^0)$, $4$ choices of $h$, and bound-constrained and unconstrained settings.

\subsection{Implementation details}

Implementations of the four tested methods were developed in
\texttt{Matlab}\footnote{All software is available at
\url{https://github.com/POptUS/IBCDFO/}}.
All methods used the same parameters (e.g., $\Delta_0$, $\eta_1$) where possible.
The \msdshortref implementation was run (only for the unconstrained problems)
using the default settings and subproblem solvers outlined in~\cite{Larson2020}.
Each method was given a budget of $100(n+1)$ evaluations of $F$ with tolerances
set to be as small as possible. For \goombahref, $\omega =1$ was used.
Our implementations of both \goombahref versions and \mspshortref do not explicitly
verify whether \Cref{ass:cauchy} is satisfied since we have not found it necessary
in practice.
Finally, we set $c_1 = c_2 = 1.0 + 10^{-8}$ in the definition of $\Gen_k$, see  \eqref{eq:gen_k}.

The \goombahref solutions to~\eqref{eq:GOOMBAH_subproblem} were produced by
calling various \texttt{GAMS} 38.3 solvers via the \texttt{Matlab} \texttt{GAMS} \texttt{GDXMRW}
interface~\cite{GAMS}. 
The particular \texttt{GAMS} models that encode our four nonsmooth
mappings can be viewed in our linked repository.
Unfortunately, we were unable to find one solver that
worked universally for all the subproblems and all the functions $h$.
Some of this is due to limitations on the forms of nonsmoothness supported by
the various optimization solvers. We also found that numerical issues occasionally arose
when $\Delta$ was tiny or when models $m^{F_i}$ were poorly scaled.
As a remedy, we attempted to solve each instance of~\eqref{eq:GOOMBAH_subproblem}  by multiple \texttt{GAMS}
solvers, including \texttt{CONOPT}~\cite{conopt}, \texttt{MINOS}~\cite{minos},
\texttt{KNITRO}~\cite{knitro}, %
and \texttt{BARON}~\cite{baron}. All solvers were limited to 30 seconds for each
subproblem solve, but rarely did any solve take more than a few
seconds in its entirety. These solvers were also used to compute $\rchi$ in
\mspshortref. All model Hessians $H_k$ were set to zero; therefore $\eta_2$ in the \msploopref
was effectively chosen to be $\infty$.
We remark that because of this choice, the if conditional beginning in
\Cref{line:acceptable} of the \msploopref is never entered, and hence there
is no need to compute $\rchi_k$. Nevertheless, we compute $\rchi_k$ in our
implementation anyway, for the sake of being able to monitor a stationarity
measure.

\subsection{Comparing implementations with approximate stationarity}

We observe that many of the compositions $h \circ F$ have numerous
stationary points. Since the methods being compared are all local optimization
methods, therefore, we find that comparing performance in terms of objective
value to be possibly misleading. That is, we believe a local optimization method should get credit
for solving a problem when it has identified an (approximate) stationary point,
even if that point has a worse function value than some other stationary point.

To measure a method's progress, we find it necessary to have an approximate
stationary value for each point $x^t$ evaluated by a given method. Obtaining
this quantity is somewhat difficult, with various concerns that must be
addressed. To begin, we randomly generate $50$ points uniformly within
$\cB(x^t;10^{-5})\subset \Reals^n$ and denote them $S^t$.\jefflabel{Sdef} The same initial
random seed is used when comparing all methods on a given problem so the
starting pattern of points is equal for all methods. We add to $S^t$ all points
evaluated by the method within $\cB(x^t;10^{-5})$. These points include $x^t$
and possibly any points that were used by the method to determine approximate
stationarity.

With $S^t$ in hand, we can then compute
\begin{equation}
\label{eq:gradient_bundle}
D(x^t)\defined \{\nabla F(s) \nabla h_j(F(s))\colon j\in\Act{F(s)}, \; s\in S^t\}
\end{equation}
and the corresponding
\begin{equation}
\label{eq:function_bundle}
a(x^t)\defined \{h_j(F(s))\colon j\in\Act{F(s)} \; s\in S^t\}.
\end{equation}
The gradient values in~\eqref{eq:gradient_bundle}
can be computed (in
postprocessing) because $\nabla F$ is computable in closed
form for the problems considered and each $h_j$ has a known gradient.
We consider a problem to be solved to an absolute level $\tau$ when the
optimal value of~\eqref{eq:chi} with problem data $D(x^t)$ and $a(x^t)$ in place of $G^k$ and $a^k$, respectively, denoted $\rchi_t$, satisfies
\begin{equation}
\label{eq:converged}
\rchi_t \leq \tau.
\end{equation}

We used the same routines used to compute $\rchi_k$ in \mspshortref to compute $\rchi_t$ for benchmarking.
Data profiles~\cite{JJMSMW09} are used to compare the performance of the tested
methods. Data profiles display how many evaluations of $F$ are required by each
method to solve a certain fraction of the benchmark set of problems to a level
$\tau$ for criterion~\eqref{eq:converged}. If a method satisfies~\eqref{eq:converged} on any problem for the first time after $t$ evaluations
of $F$, the data profile is incremented by $\frac{1}{424}$ at the point
$\frac{t}{(n_p+1)}$ (where $n_p$ is the dimension of the problem) on the
horizontal axis. In other words, a method's data profile displays the
cumulative fraction of problems solved by that method as a function of the
number of evaluations of $F$ (scaled by $n_p + 1$).

\begin{remark}
The condition in~\eqref{eq:function_bundle} was chosen for testing stationarity
because it was the stationary measure used in the analysis of \texttt{MS-P}. One also can consider a projection onto zero of the convex hull of
$D(x^t)$ and active constraint normals as an analogous stationary condition.
For unconstrained problems, this latter stationary measure would be the same as tested in, for instance,~\cite{Larson2020}.
We found no meaningful difference between the data profiles when using either
metric.
\end{remark}

\begin{figure}[t]
   \begin{center}
     \subfigure[][$\tau = 10^{-1}$, unconstrained]{\includegraphics[width=0.48\textwidth]{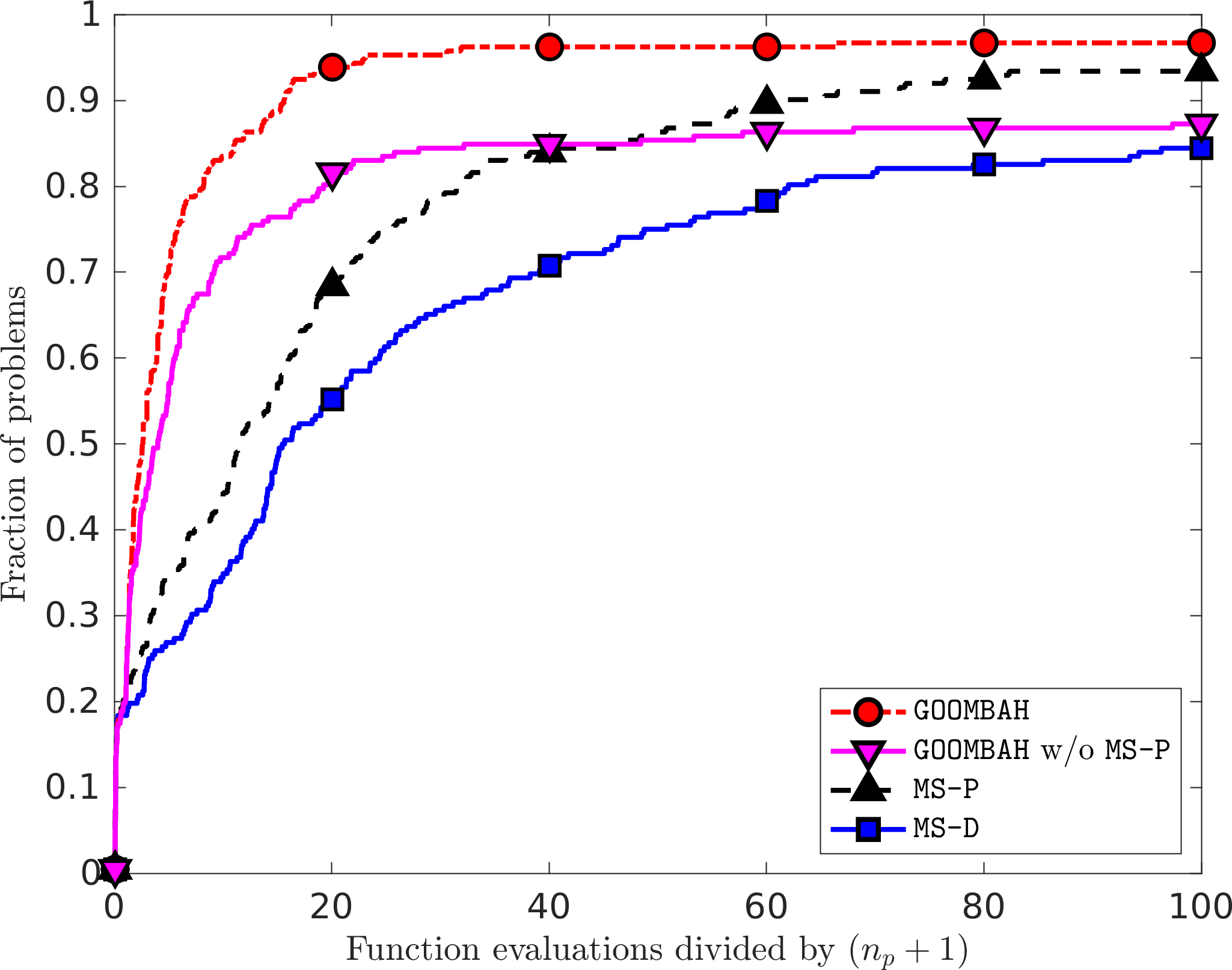}}
     \hfil
     \subfigure[][$\tau = 10^{-5}$, unconstrained]{\includegraphics[width=0.48\textwidth]{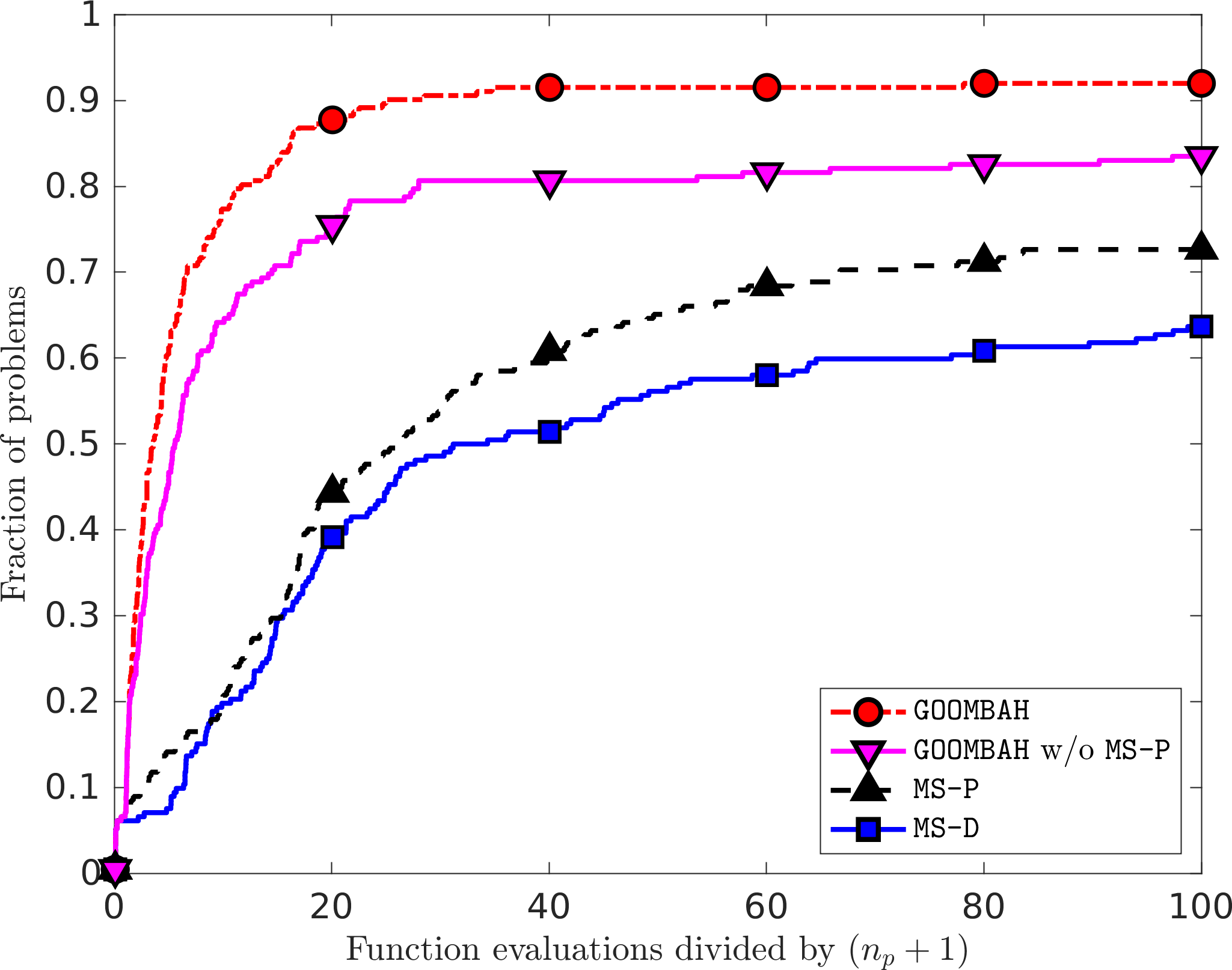}}\\
     \subfigure[][$\tau = 10^{-1}$, bound-constrained]{\includegraphics[width=0.48\textwidth]{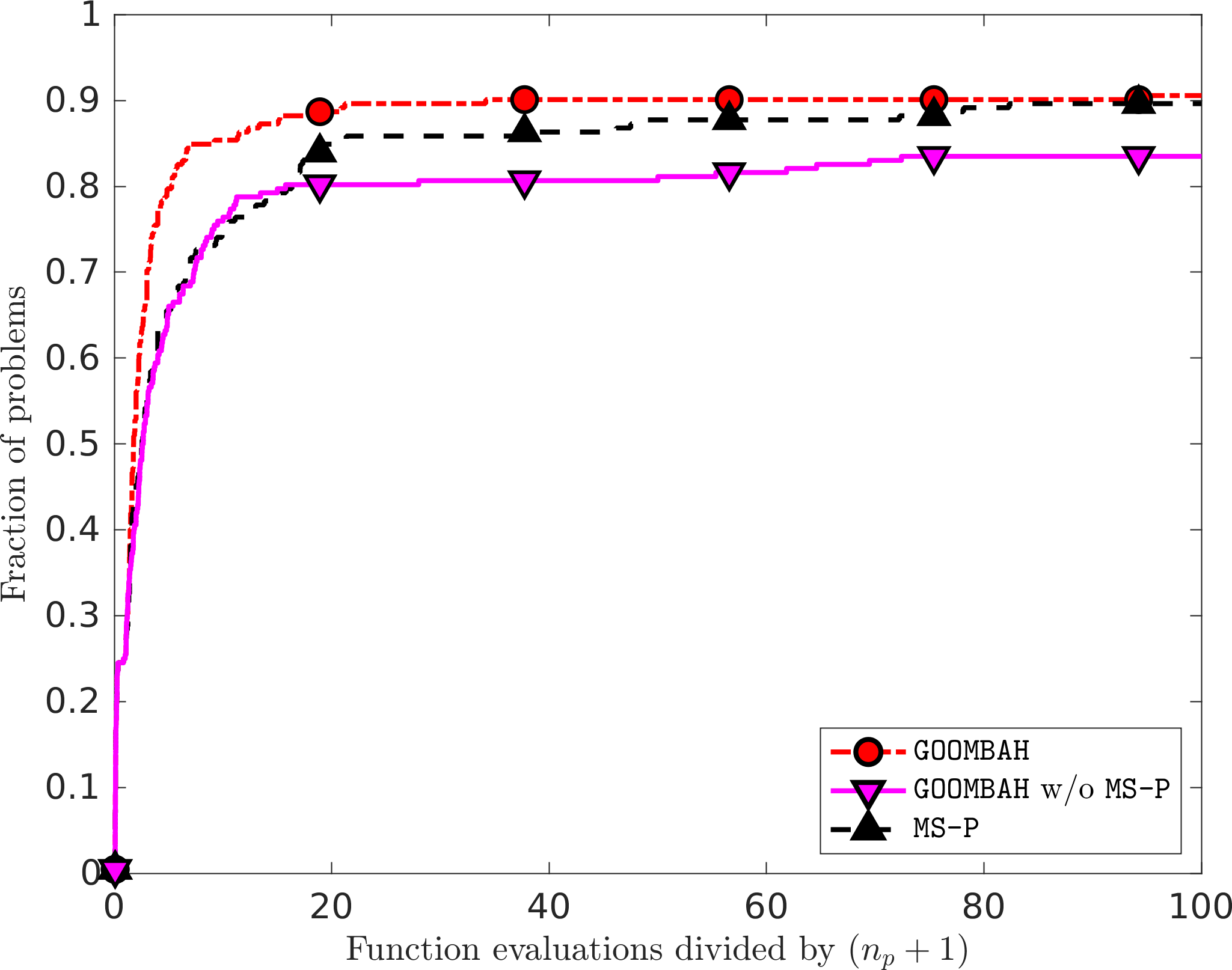}}
     \hfil
     \subfigure[][$\tau = 10^{-5}$, bound-constrained]{\includegraphics[width=0.48\textwidth]{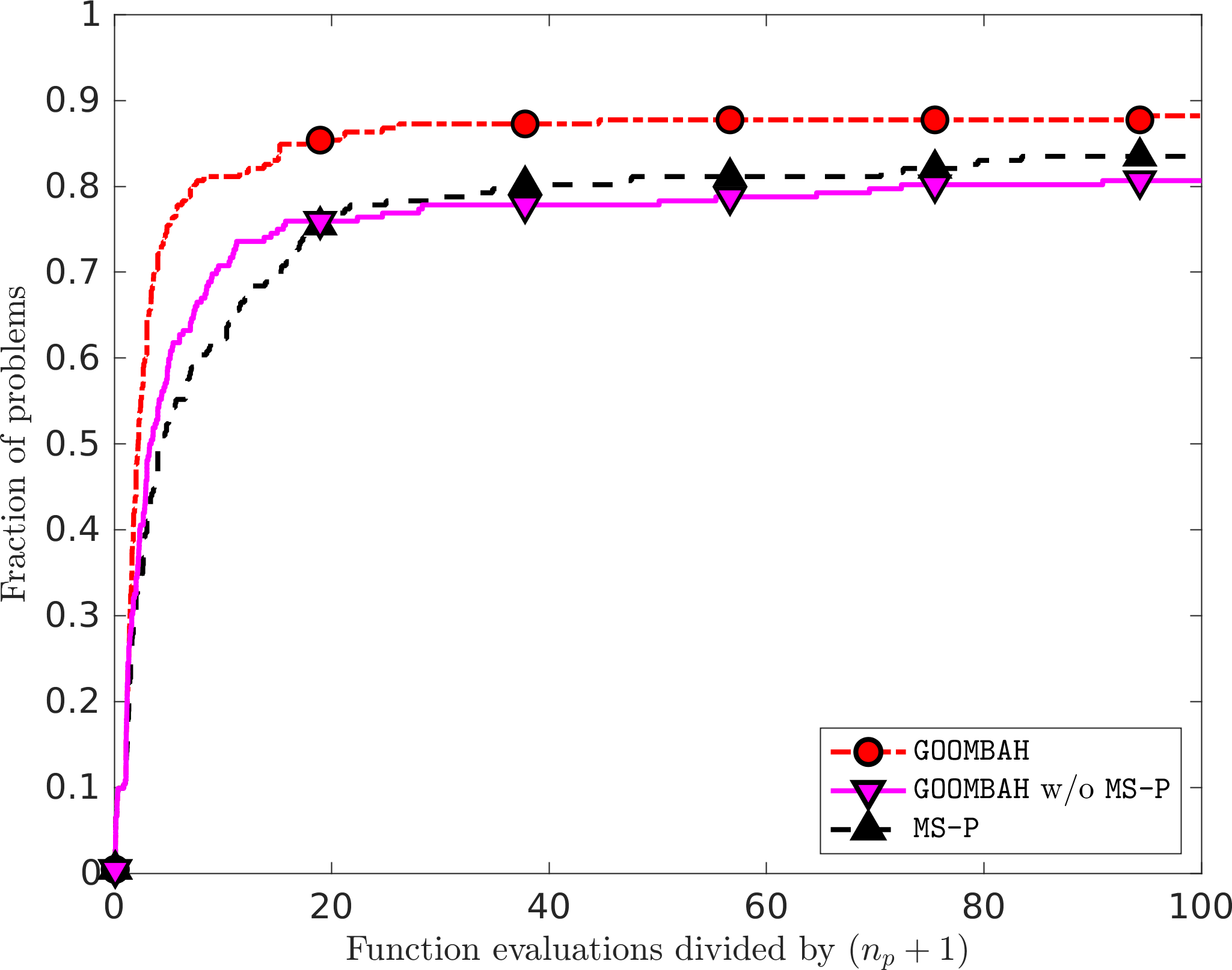}}
  \end{center}
   \caption{Data profiles using the stationary metric~\eqref{eq:converged}  with $\tau = 10^{-1}$ (left) and $\tau = 10^{-5}$ (right) for the $53 \times 4$ bound-constrained and unconstrained problems.\label{fig:all_four}}
   \end{figure}

The performance of the four benchmarked implementations is presented in
\Cref{fig:all_four}. For the most part, we
see that the \goombahref implementations considerably outperform the manifold
sampling implementations on both bound-constrained and unconstrained problems. We
observe that the (infrequent) recourse to manifold sampling logic helps
\goombahref solve approximately 5\% of the benchmark problems.

As mentioned, \msdshortref is seen to be outperformed by its primal counterpart
\mspshortref for unconstrained problems for all levels of $\tau$ that we
considered. The supplemental material in \Cref{sec:more_figs}
provides a greater dissection of these
results for each of the $h$ functions considered. We do note that \msdshortref
and \mspshortref are seen to have nearly identical performance for the
$h_1$-problems. 
This may be explained by the fact that the pointwise minimum structure of $h_1$ can still be addressed
by the \mspshortref analysis, but
it is clear that~\eqref{eq:primal_model}, which involves a pointwise maximum, is a poor model of such an $h_1$. 

\section{Discussion and Future Directions}
Many open research questions are related to the methods presented in this
manuscript.
One of the most obvious is the need to quantify and empirically study
the trade-off between the cost of producing the next iterate $x^k + s^k$ and the cost of performing
objective evaluations $F(x^k + s^k)$.
While many derivative-free optimization papers may make
a blanket statement that the cost of the objective is assumed to outweigh any
cost of the operations of the optimization method itself, such an assumption
seems less reasonable for the \goombahref method, especially if $h$ is complicated,
$m^{F_i}$ are nonlinear models, and a global optimization method is being
called to produce iterates.
The development of a ``convention'' or ``rule'' for limiting such algorithmic
effort does not appear obvious.
It seems natural to limit the wall-clock time used to solve~\eqref{eq:GOOMBAH_subproblem}---which may need to invoke global optimization solvers---to less than some fraction (say, one-tenth) of the expected time required
to evaluate $F$ on a given computational resource.
From another vantage, such a limitation may not be proper,
especially if the simulation uses massive amounts of parallel compute resources and
the trust-region subproblem solver cannot utilize such resources as efficiently.
In fact, one could argue that \emph{more} time should be spent on the solution of~\eqref{eq:GOOMBAH_subproblem}
when $F$ requires considerable
computational resources;
the time spent improving the next iterate $x^k+s^k$ could greatly reduce
the required number of calls to $F$.
Such a trade-off calculation is even further confused by the fact that, in
practice, global optimization methods often find high-quality solutions quickly;
additional effort does not improve the solution quality but instead only reduces
the gap between a lower and upper bound on the best possible objective value.

Another feature we seek to add to the implementation of \mspshortref or
\goombahref is the ability to use (approximate) derivatives of $F$ when they are available.
Because \mspshortref requires only gradient-accurate models via
\Cref{def:flmodels}, we remark that given access to computable gradients
$\nabla F_i(x)$, we have only to define $m^{F_i}(x)$ as the first-order Taylor
model centered at the current iterate $x^k$, namely,
$$m^{F_i}(x^k+s) = F_i(x^k) + \nabla F_i(x^k)^\top s.$$
Such a choice of $m^{F_i}$ is automatically gradient-accurate with constant
$\kappaieg = \LnFi$. Although we have not yet implemented or tested a gradient-based
method, this is a trivial extension, and we intend to release it eventually.

In future work, we can also investigate nontrivial model Hessians $H_k$ in~\eqref{eq:subproblem}.
In our
experiments, we employed only $H_k = 0$, but our analysis permits the use of
any model Hessian satisfying $\|H_k\|\leq\kappabmh$. Practical experience (and
limited theoretical results; see, e.g.,~\cite{lewis2013nonsmooth}) leads one to
believe that BFGS matrices may be appropriate choices for $H_k$. Alternatively,
in a (primal) gradient sampling context, Curtis and Que~\cite{curtis2013adaptive} employed an
overestimating Hessian strategy that may prove useful in the context of the
subproblem~\eqref{eq:subproblem}. Additionally, at the expense of even more
difficult subproblems (QCQPs), Schichl and Fendl~\cite{schichl2020second} proposed---in a bundle
method context---incorporating Hessians that approximate $\nabla^2 F_i(x)$
into the constraints of~\eqref{eq:subproblem}. Although this would require a
more computationally difficult subproblem, this is in line with our discussion
acknowledging tradeoffs between subproblem difficulty and
the expense of evaluating $F$. All these Hessian-building methods are of
interest to us in future testing and releases of software.

In closing, gentle reader, we'd like to thank you. What's that, you say? Us
thanking you? No, it's not a misprint. For you see, we enjoyed writing this
manuscript as much as you enjoyed reading it. The end.

\section*{Acknowledgments}
We thank Geovani Nunes Grapiglia for initial discussions of convergence analysis
results.
This work was supported in part by the U.S.~Department of Energy, Office of
Science, Office of Advanced Scientific Computing Research and Office of
High-Energy Physics, Scientific Discovery through Advanced Computing (SciDAC)
Program through the FASTMath Institute and the CAMPA Project under Contract
No.~DE-AC02-06CH11357.

\printbibliography

\appendix
\section{Table of notation}\label{sec:table}
\begin{description}[before={\renewcommand\makelabel[1]{##1:}}]
\item[$\Act{z}$] Set of indices of essentially active functions at a point $z$ \dotfill \Cref{def:pc1manifold}
\item[$\cB$] Euclidean ball \dotfill \newref{cBdef}
\item[$D$] Sampled gradients of $f$ used in numerical tests \dotfill \Cref{eq:gradient_bundle}
\item[$F$] Expensive inner function, with components $F_i$ \dotfill \Cref{ass:f}
\item[$G$] Matrix with columns of vectors from $g_j^k$ \dotfill \Cref{eq:dual}
\item[$\Gen$] Set of indices used to generate $G$ \dotfill \Cref{eq:gen_k}
\item[$H$] Model Hessian $H^k$ \dotfill \Cref{eq:primal_model}
\item[$\Hdef$] Set of selection functions \dotfill \Cref{def:continuous_selection}
\item[$I_n$] The identity matrix for $\Reals^n$  \dotfill \Cref{eq:dual}
\item[$K$] Used for Lipschitz constants, (e.g., $\Lh, \LnF, \LnFi$ ) \dotfill e.g., \Cref{def:constants}
\item[$\mathbb{K}$] Special sets of iterates \dotfill \newref{specialK}
\item[$\Level$] Level set \dotfill \newref{Leveldef}
\item[$\Levelmax$] Level set plus $\Deltamax$ padding \dotfill \newref{eq:Lmaxdef}
\item[$\Lag$] The Lagrangian function \dotfill \newref{Lagdef}
\item[$\Lset$] $\{i\in 1,\dots,n: \ell_i = -\infty\}$ \dotfill\newref{LandUdef}
\item[$M$] Vector mapping of the $n$ models $m^{F_i}$ \dotfill\newref{Mdef}
\item[$\Integers$] The set of Integers
\item[$Q$] $Q_i$ is used to define quadratics in test functions
  \dotfill\newref{Qdef}
\item[$\Reals$] The set of real numbers
\item[$S$] Set of points $S^t$ sampled around each $x$ evaluated by a method \dotfill\newref{Sdef}
\item[$\Uset$] $\triangleq \{i\in 1,\dots,n: u_i = \infty\}$ \dotfill\newref{LandUdef}
\item[$Y$] A collection of points $y$ from the domain of $F$ that have been evaluated \dotfill\newref{Ydef}

~\\

\item[$a$] The value $ [a^k]_j \defined f(x^k) - f_j(x^k) + \beta_{j,k}$ \dotfill\newref{a_def}
\item[$b$] $b_i$ is used to define quadratics in test functions \dotfill\newref{Qdef}
\item[$c$] $c_i$ are the censors for the censored-L1 loss function, Also $c_1,
  c_2$ are algorithmic constants \dotfill\newref{Qdef} and \Cref{eq:gen_k}
\item[$d$] $d_i$ are the data for the censored-L1 loss function \dotfill\newref{Qdef}
\item[$e$] Vector of all ones
\item[$f$] Composite objective function $f \defined h \circ F$. Also, sometimes $f_j \defined h_j \circ F$ \dotfill\Cref{ass:f}
\item[$g$] The generators, $g_j^k = \left[\nabla m^F_k\right]\nabla h_j(F(x^k))$ \dotfill\newref{gen_def}
\item[$h$] Nonsmooth, outer piecewise-selection function \dotfill\Cref{ass:f}
\item[$h_j$] Smooth selection functions defining $h$ \dotfill\Cref{def:continuous_selection}
\item[$i$] General index
\item[$j$] General index
\item[$k$] Iteration of the algorithm
\item[$\lb$] Lower bounds  \dotfill\newref{Omega_def}
\item[$m^{F_i}$] A model of $F_i$ \dotfill\Cref{def:flmodels}
\item[$n$] Dimension of domain of $F$ (and $f$) \dotfill\Cref{ass:f}
\item[$p$] Dimension of domain of $h$ \dotfill\Cref{ass:f}
\item[$s$] The trust-region subproblem step, sometimes $s^*$ or $s^k$ or $\sGOOMBAH^k$ 
  \dotfill\Cref{eq:equiv_subproblem} or \Cref{eq:GOOMBAH_subproblem}
\item[$t$] Index for points evaluated by methods (not necessarily the iterate $k$) \dotfill\newref{Sdef}
\item[$\ub$] Upper bound on domain \dotfill\newref{Omega_def}
\item[$v$] Primal variables for the problem \dotfill\Cref{eq:equiv_subproblem}
\item[$x$] Points in the domain of $F$
\item[$y$] Points in the domain of $F$
\item[$z$] Points in the domain of $h$

~\\

\item[$\beta$] A nonnegative offset added to affine functions in primal model \dotfill\Cref{eq:primal_model}
\item[$\gammad$] Trust-region decrease factor \dotfill\newref{gammad_def}
\item[$\gammai$] Trust-region increase factor \dotfill\newref{gammai_def}
\item[$\Delta$] Trust region radius \dotfill\newref{Delta_def}
\item[$\Deltamax$] Upper bound on trust region radius \dotfill\newref{Deltamax_def}
\item[$\eta$] Algorithmic acceptability tolerances \dotfill\newref{eta1_def}, \newref{eta2_def}
\item[$\kappa$] Bounds on errors (between models/functions, fraction of Cauchy decrease) $\kappaieg, \kappag, \kappabmh, \kappafcd$
  \dotfill \Cref{def:flmodels}, \Cref{def:constants}, \newref{kappabmh}, \Cref{ass:cauchy}
\item[$\lambda$] Dual variables ($\lambdaa, \lambdal, \lambdau$) \dotfill \Cref{eq:dual}
\item[$\pi$] Used to denote projection problem solution \dotfill\newref{pi_def}
\item[$\rho$] Ratio of actual-versus-predicted decrease \dotfill\Cref{eq:rho}, \Cref{eq:rhoGOOMBAH}
\item[$\sigma$] A mapping between two convex sets \dotfill\Cref{lem:weird_v_approx}
\item[$\tau$] Tolerance used for data profiles \dotfill\Cref{eq:converged}
\item[$\rchi$] Stationary measure\dotfill\Cref{eq:chi}
\item[$\Omega$] Domain of test problems, either $\Reals^n$ or $[\ell,u]$\dotfill\newref{Omega_def}
\item[$\partialC$] Clarke subdifferential\dotfill\newref{C_def}

\item[Operations]

\item[$\clp{\cS}$] Closure of a set $\cS$
\item[$\intp{\cS}$] Interior of a set $\cS$
\item[$\cop{\cS}$] Convex hull of a set $\cS$
\item[$\proj{0}{\cS}$] Projection of zero onto a set $\cS$
\item[$\image{\cS}$] Image of set $\cS$ under $F$

\end{description}

\section{Additional data profiles}\label{sec:more_figs}
\begin{figure}[h]
  \begin{center}
    \subfigure[][$\tau = 10^{-1}$, unconstrained]{\includegraphics[width=0.45\linewidth]{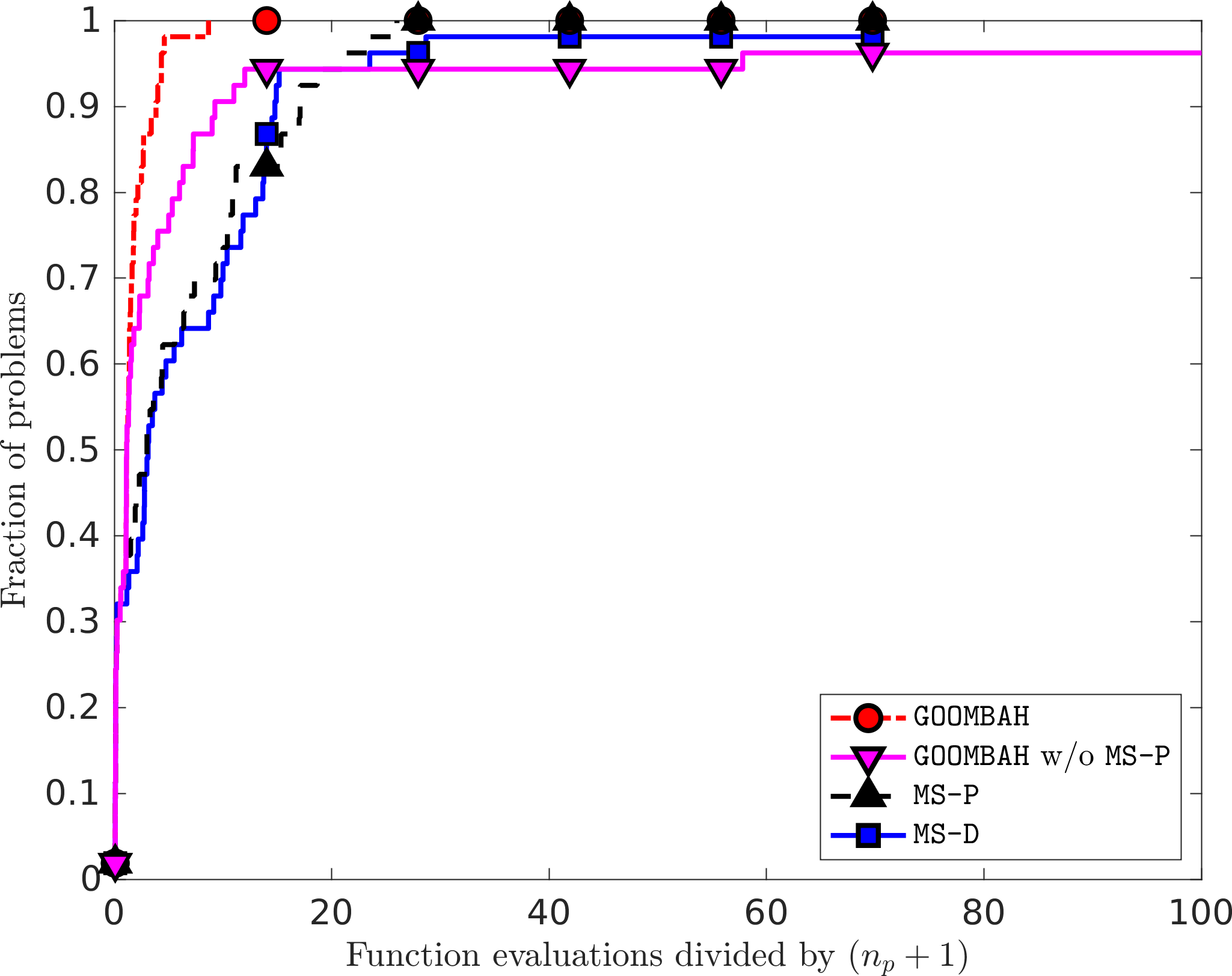}}
    \hfil
    \subfigure[][$\tau = 10^{-1}$, constrained]{\includegraphics[width=0.45\linewidth]{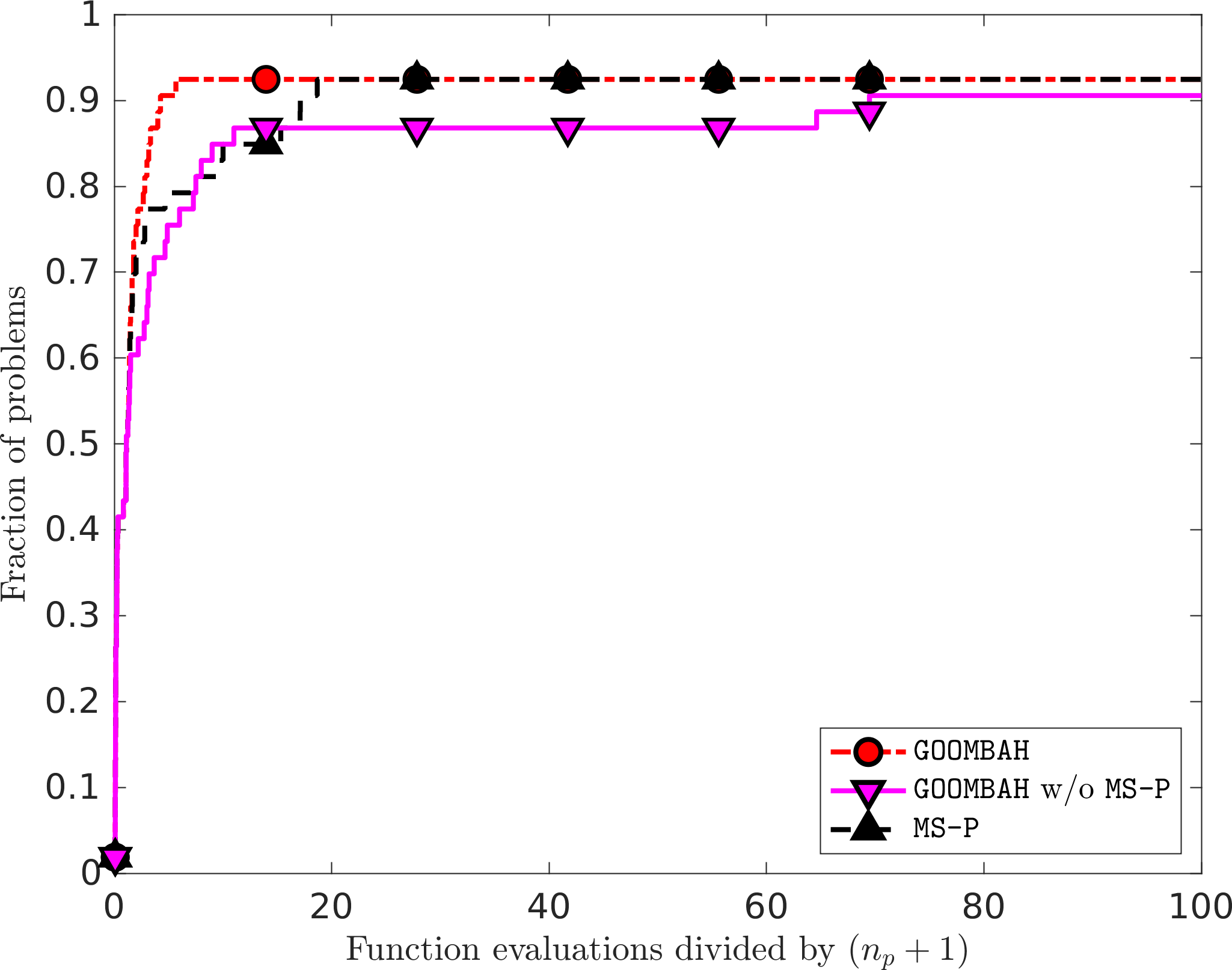}}\\[10pt]
    \subfigure[][$\tau = 10^{-3}$, unconstrained]{\includegraphics[width=0.45\linewidth]{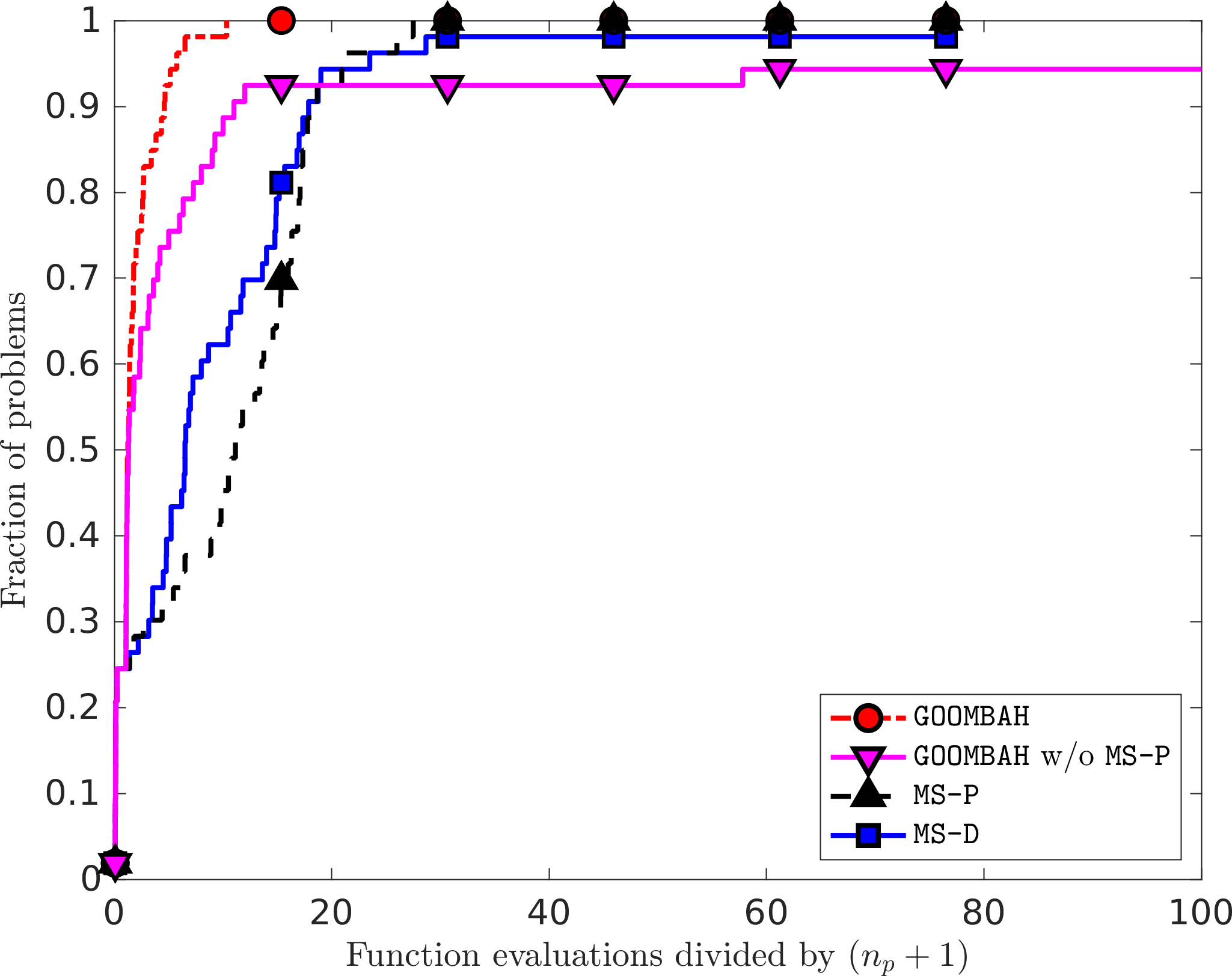}}
    \hfil
    \subfigure[][$\tau = 10^{-3}$, constrained]{\includegraphics[width=0.45\linewidth]{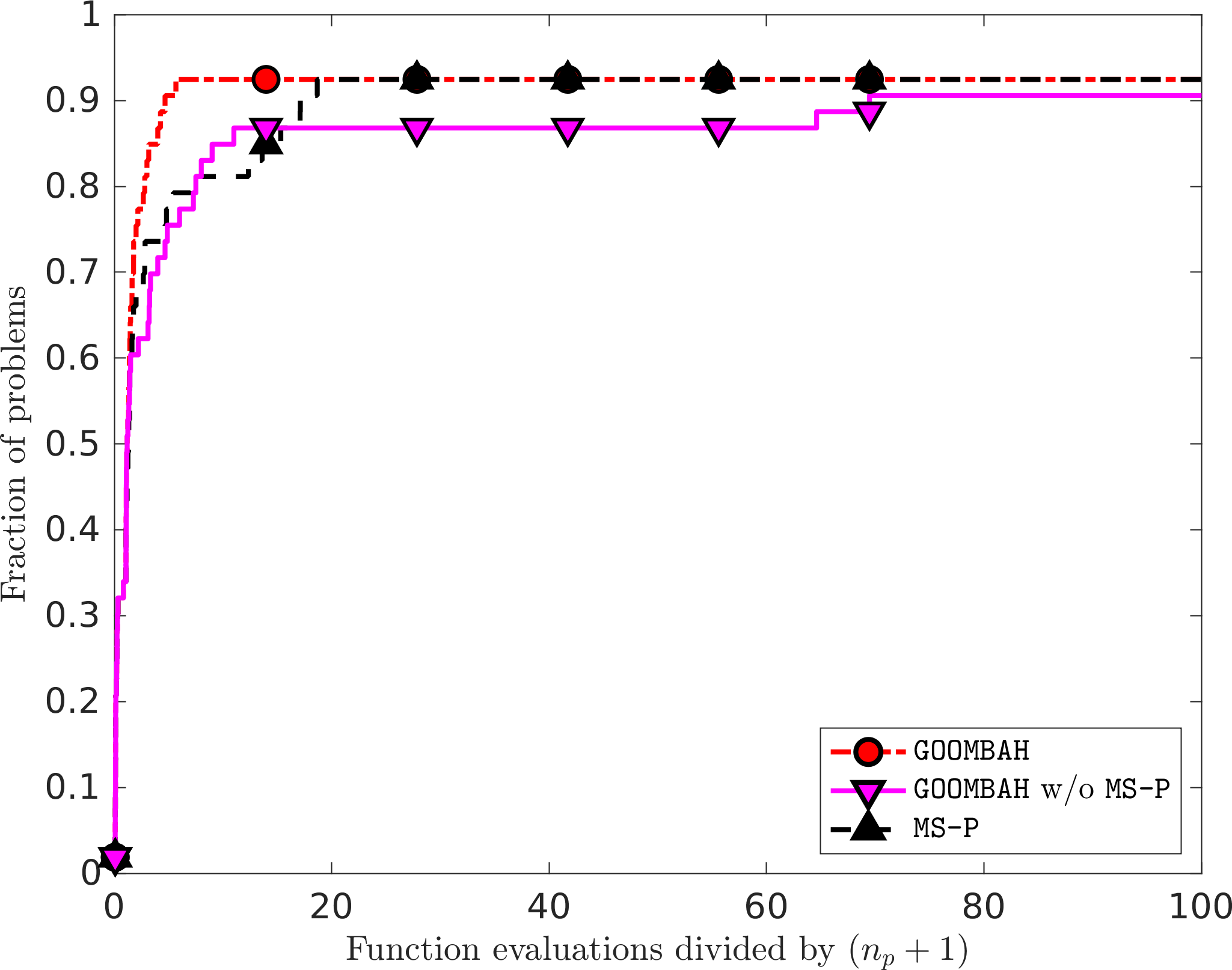}}\\[10pt]
    \subfigure[][$\tau = 10^{-5}$, unconstrained]{\includegraphics[width=0.45\linewidth]{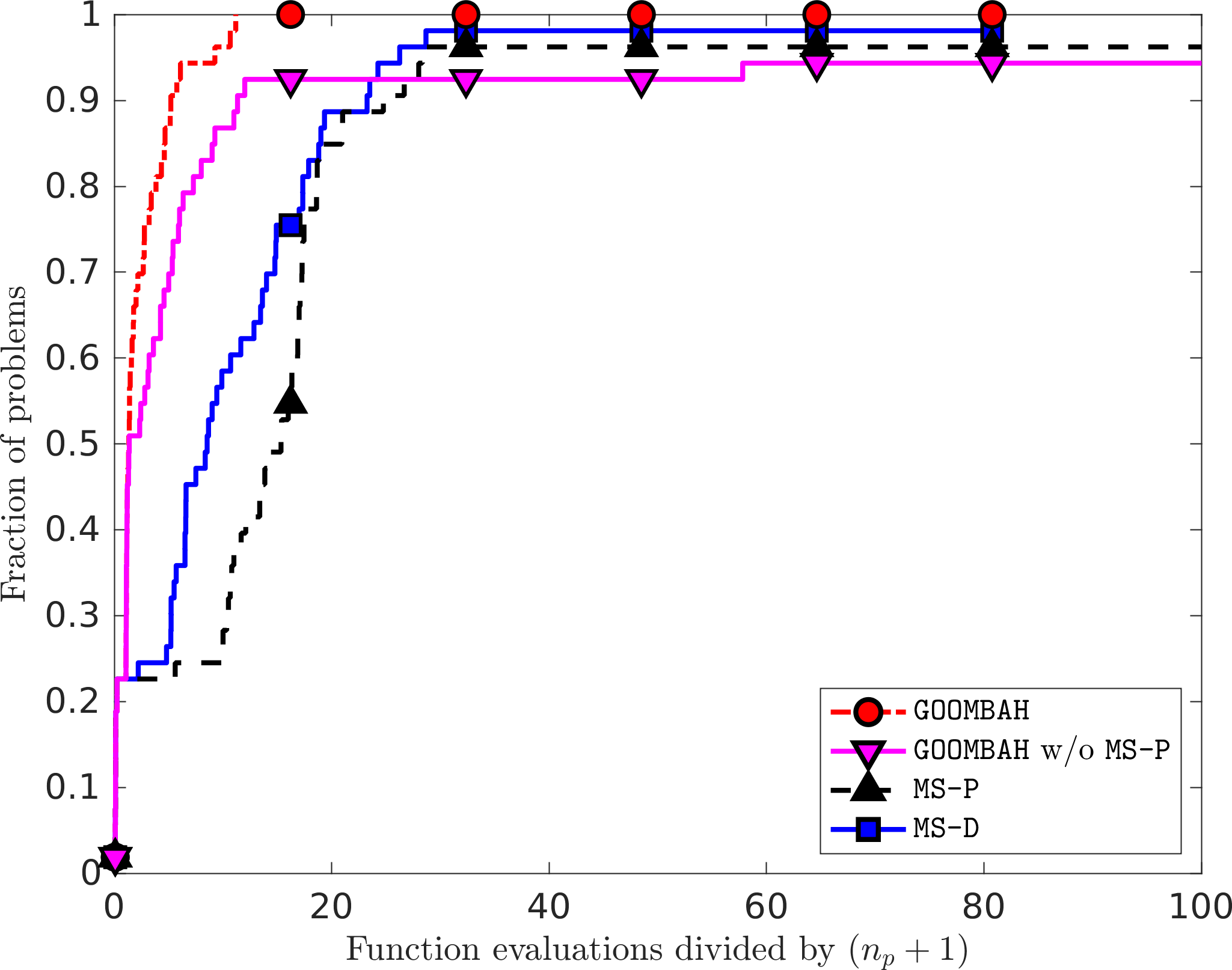}}
    \hfil
    \subfigure[][$\tau = 10^{-5}$, constrained]{\includegraphics[width=0.45\linewidth]{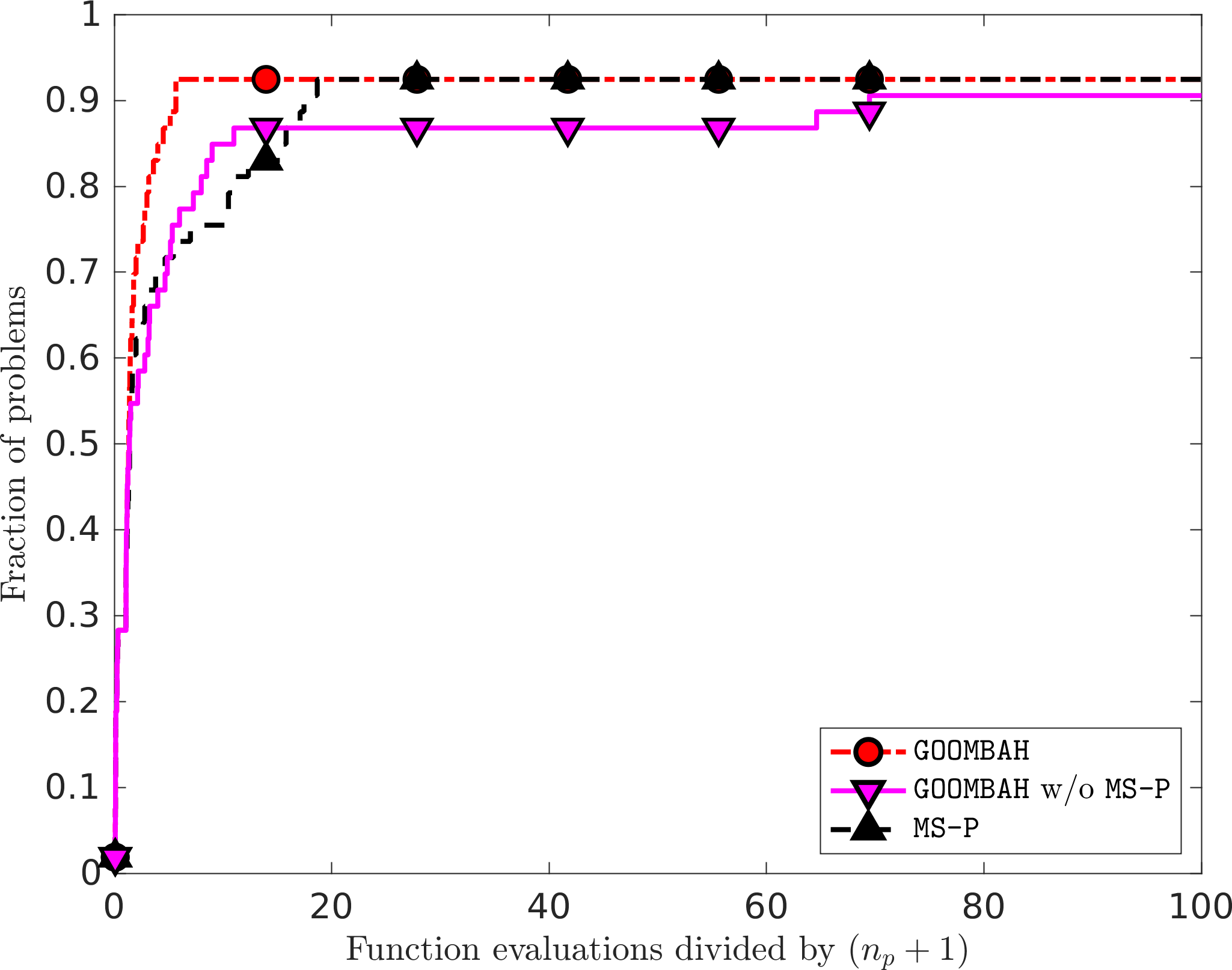}}
  \end{center}
  \caption{Data profiles on the pointwise-minimum-squared function $h_1$, constrained and unconstrained, for three values of $\tau$.}
\end{figure}

\begin{figure}[h]
  \begin{center}
    \subfigure[][$\tau = 10^{-1}$, unconstrained]{\includegraphics[width=0.45\linewidth]{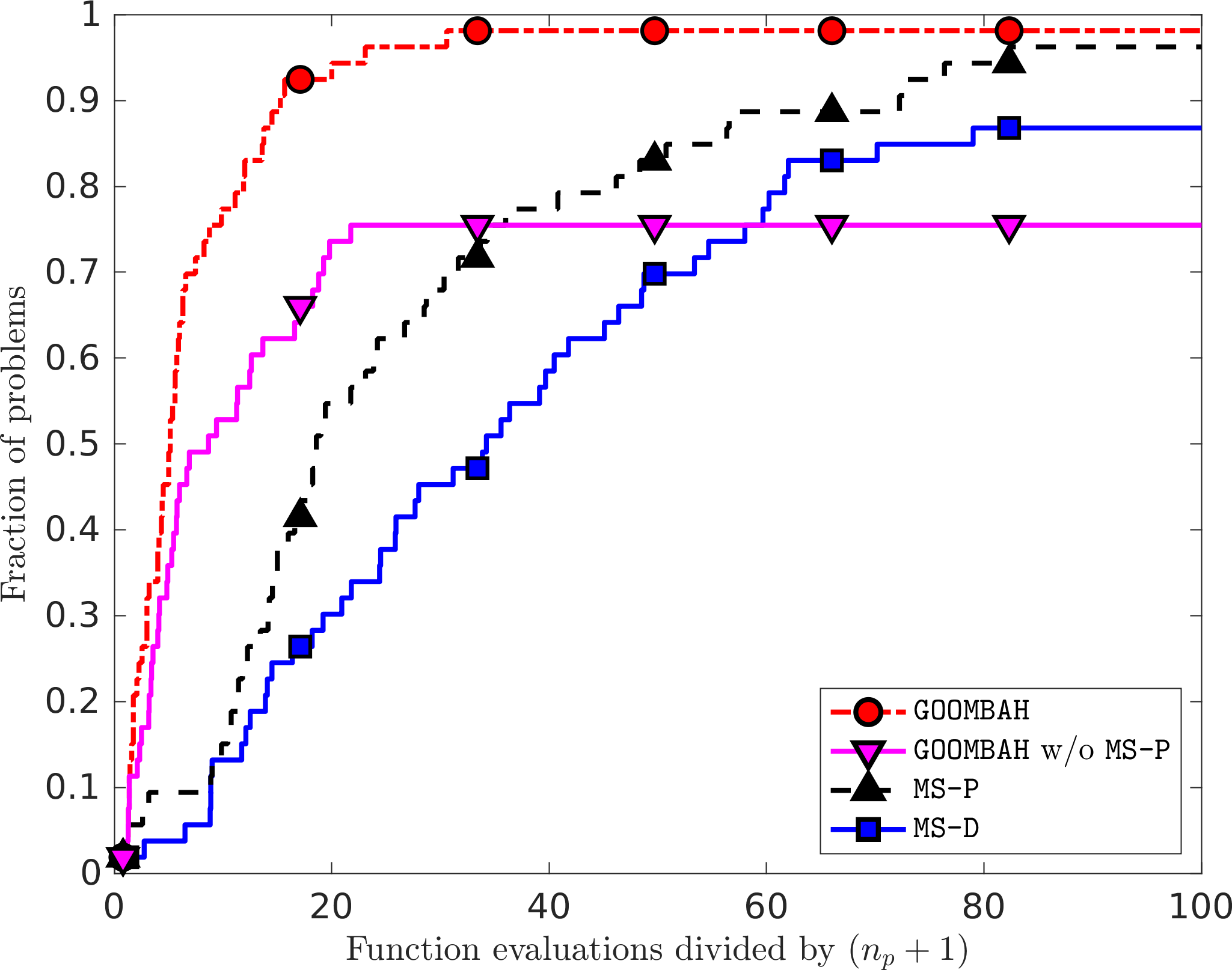}}
    \hfil
    \subfigure[][$\tau = 10^{-1}$, constrained]{\includegraphics[width=0.45\linewidth]{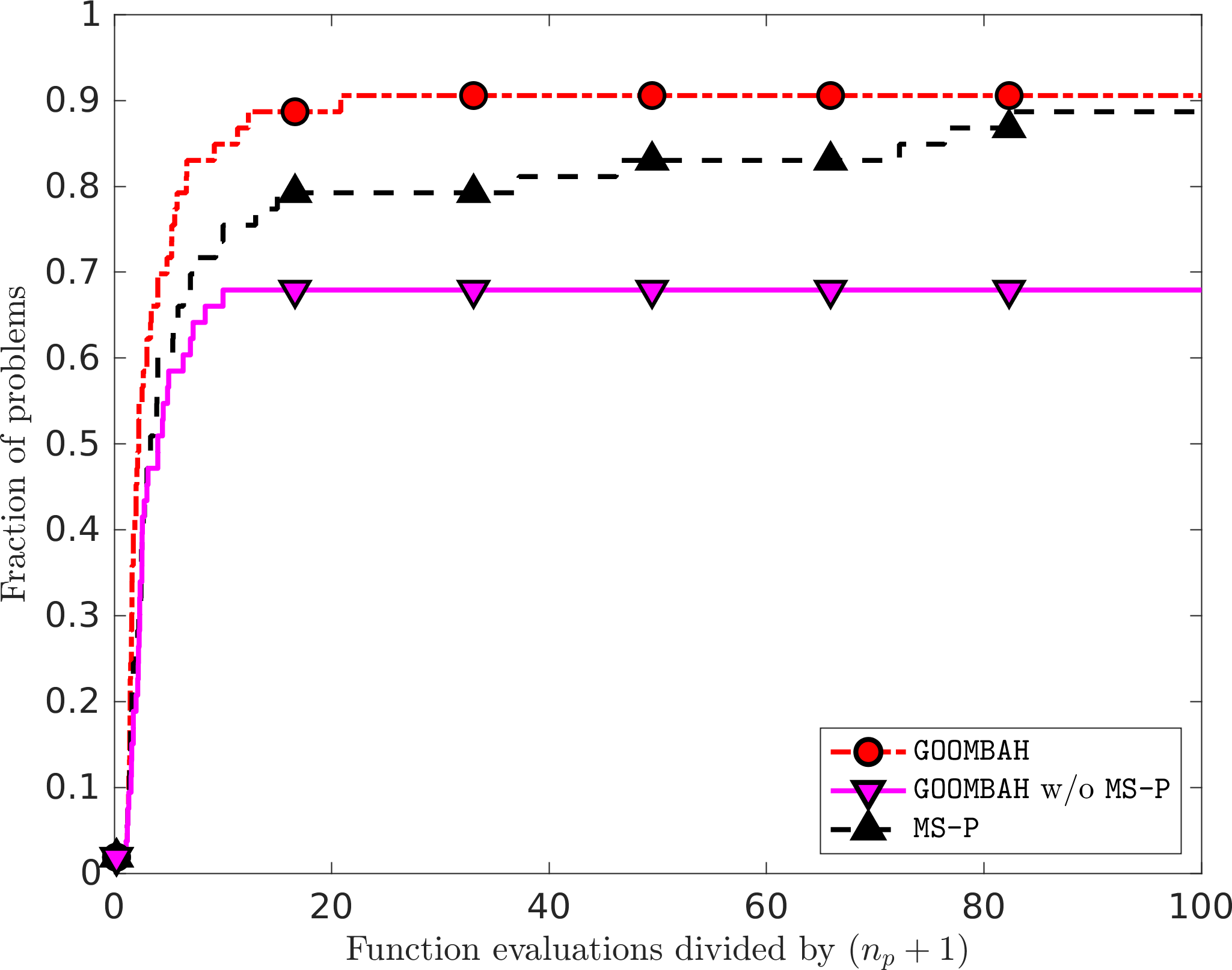}}\\[10pt]
    \subfigure[][$\tau = 10^{-3}$, unconstrained]{\includegraphics[width=0.45\linewidth]{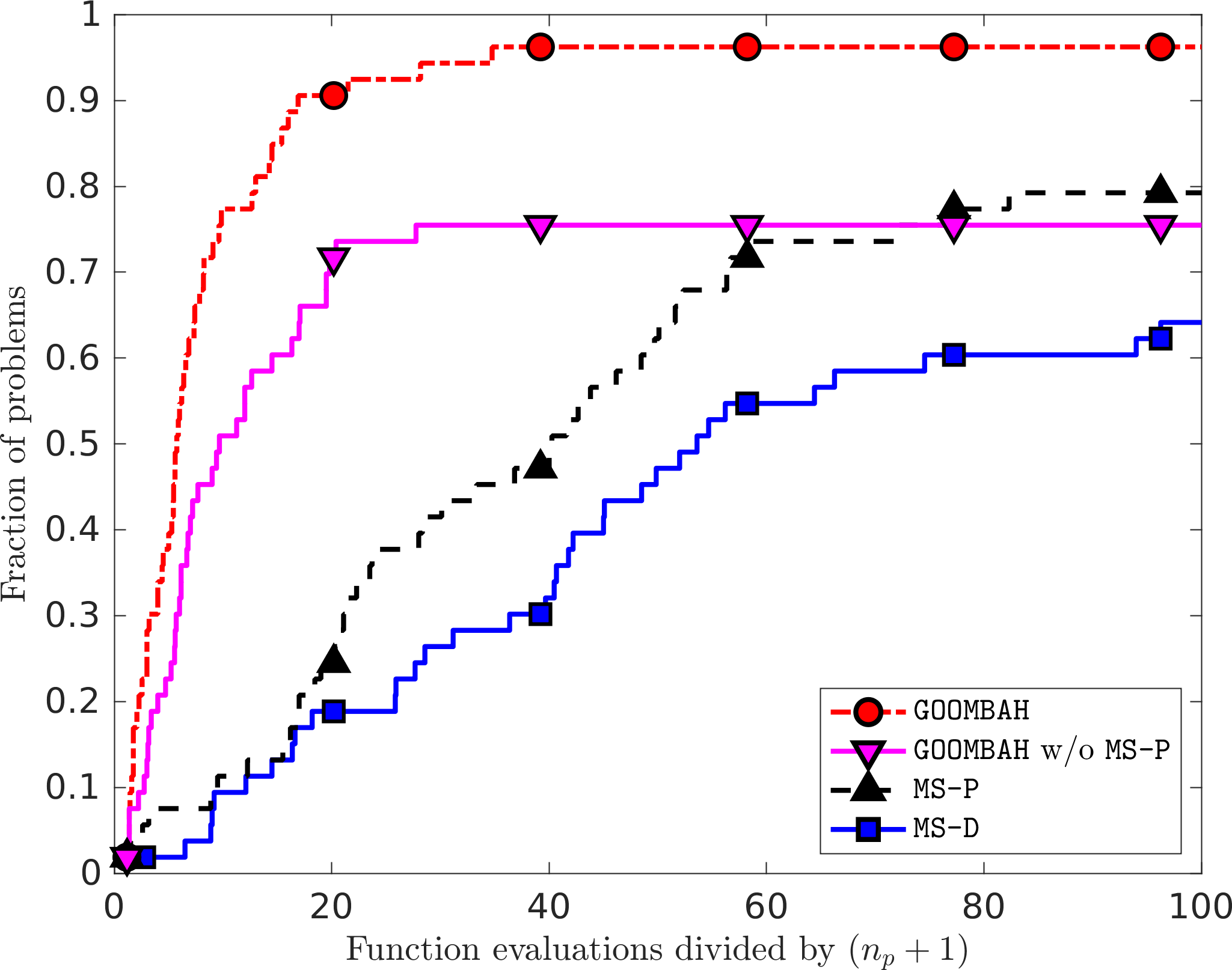}}
    \hfil
    \subfigure[][$\tau = 10^{-3}$, constrained]{\includegraphics[width=0.45\linewidth]{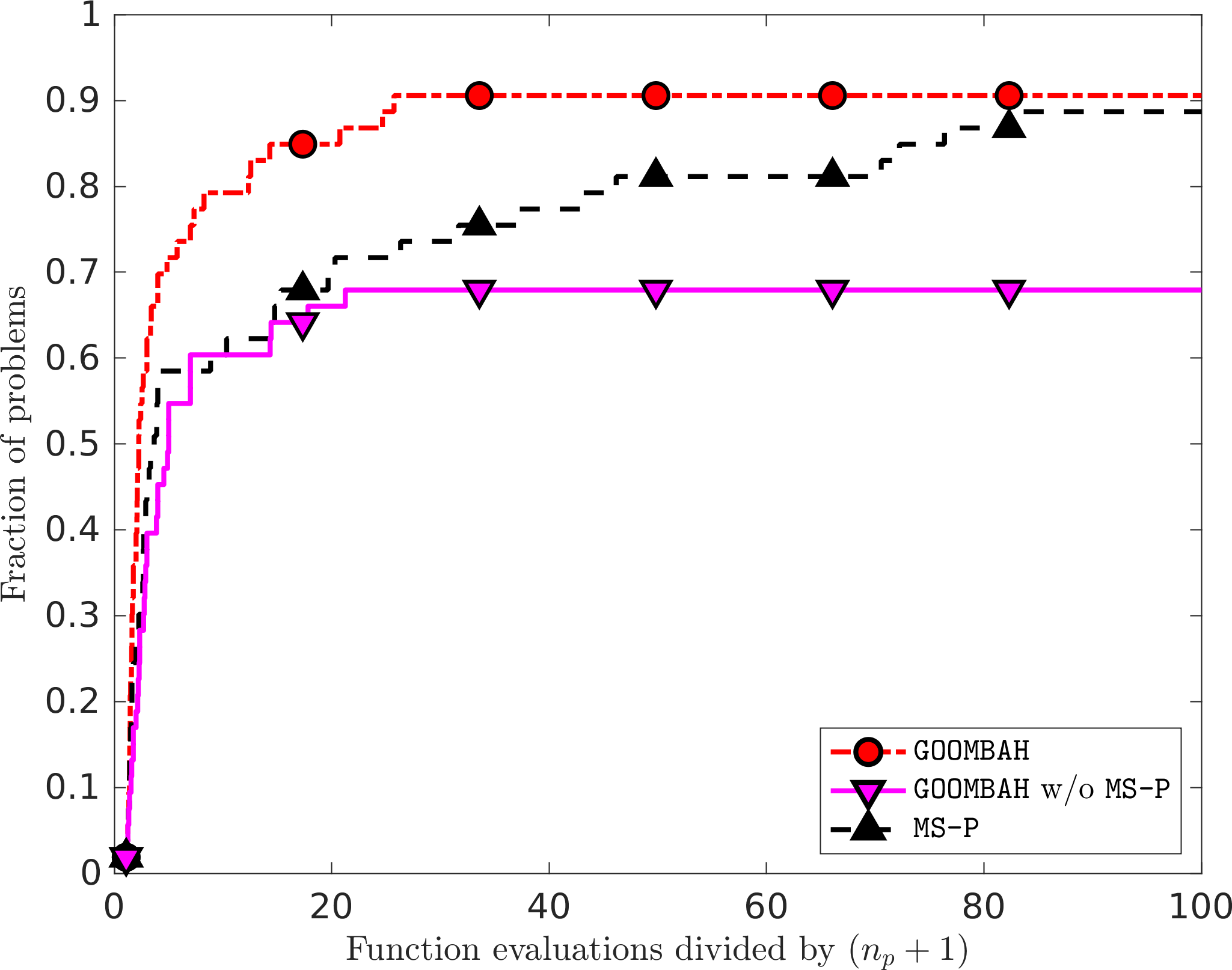}}\\[10pt]
    \subfigure[][$\tau = 10^{-5}$, unconstrained]{\includegraphics[width=0.45\linewidth]{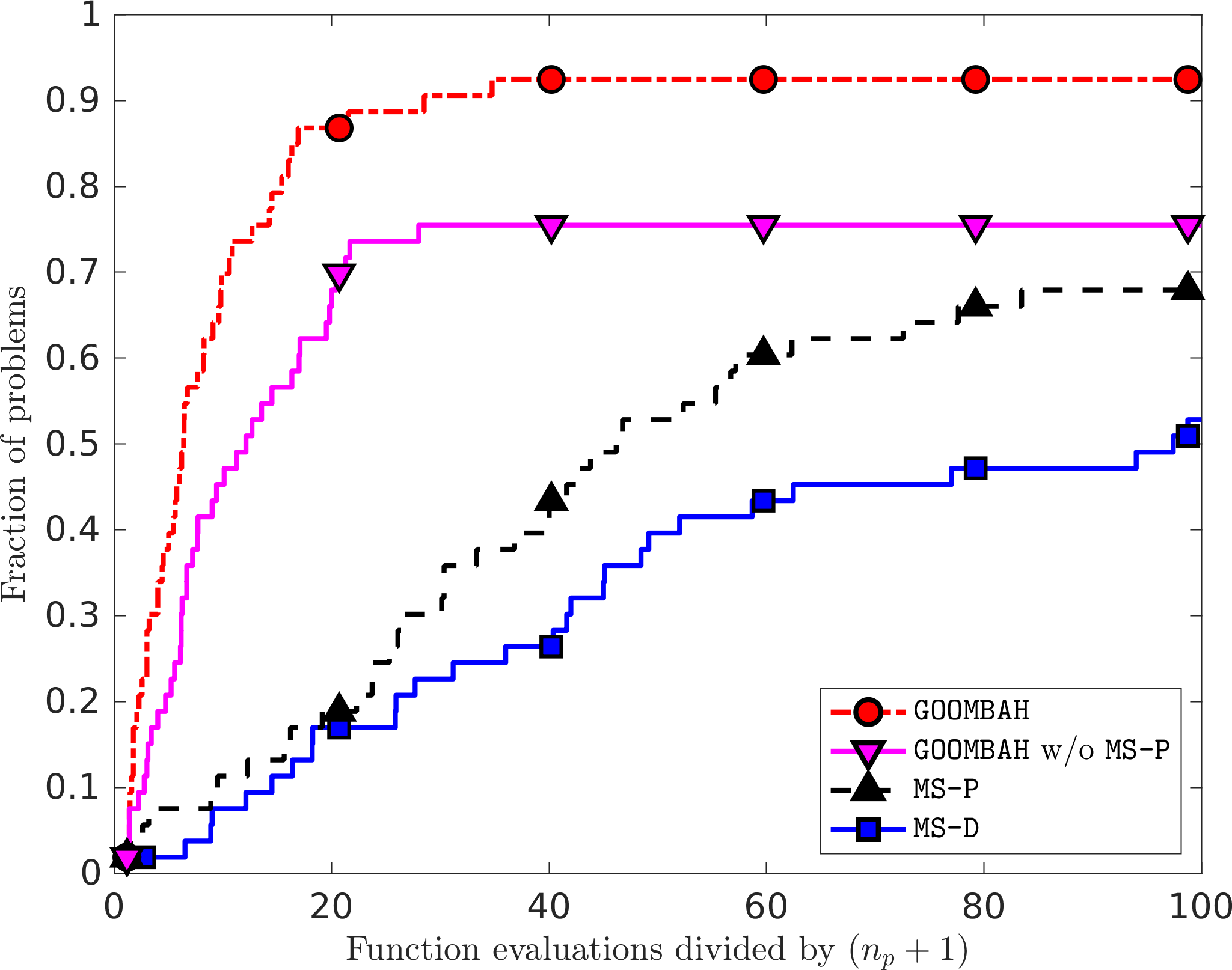}}
    \hfil
    \subfigure[][$\tau = 10^{-5}$, constrained]{\includegraphics[width=0.45\linewidth]{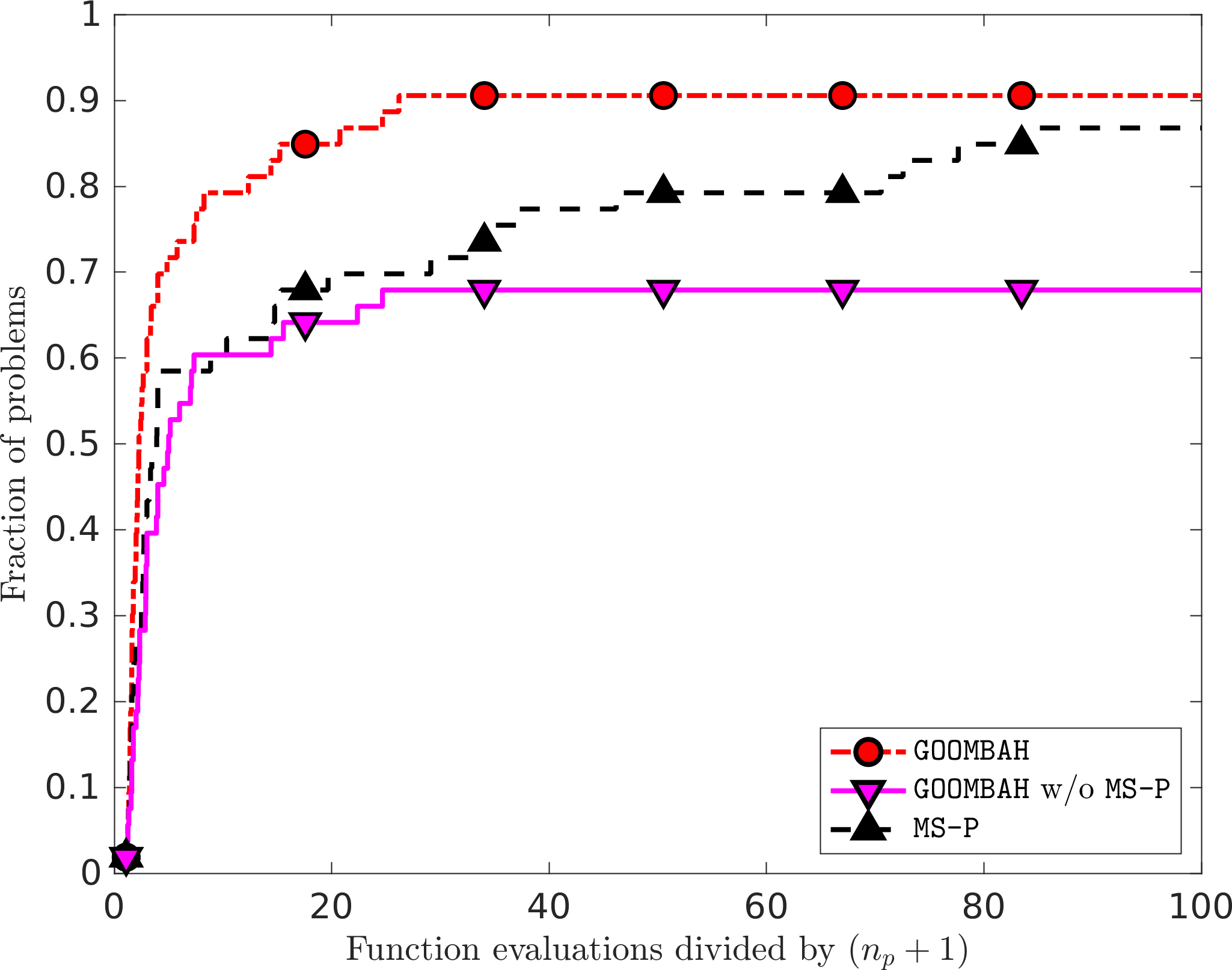}}
  \end{center}
  \caption{Data profiles on the pointwise-maximum-squared function $h_2$, constrained and unconstrained, for three values of $\tau$.}
\end{figure}

\begin{figure}[h]
  \begin{center}
    \subfigure[][$\tau = 10^{-1}$, unconstrained]{\includegraphics[width=0.45\linewidth]{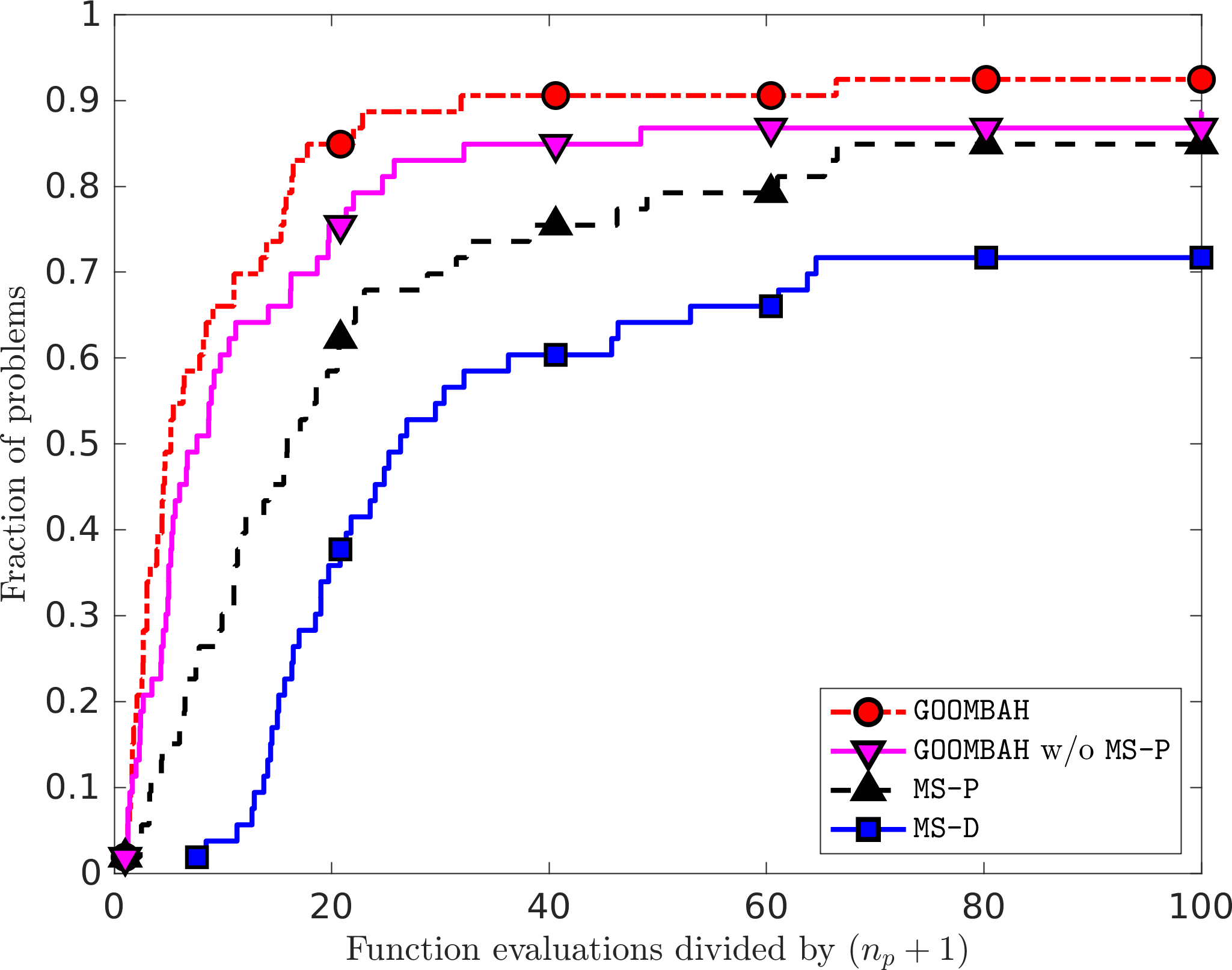}}
    \hfil
    \subfigure[][$\tau = 10^{-1}$, constrained]{\includegraphics[width=0.45\linewidth]{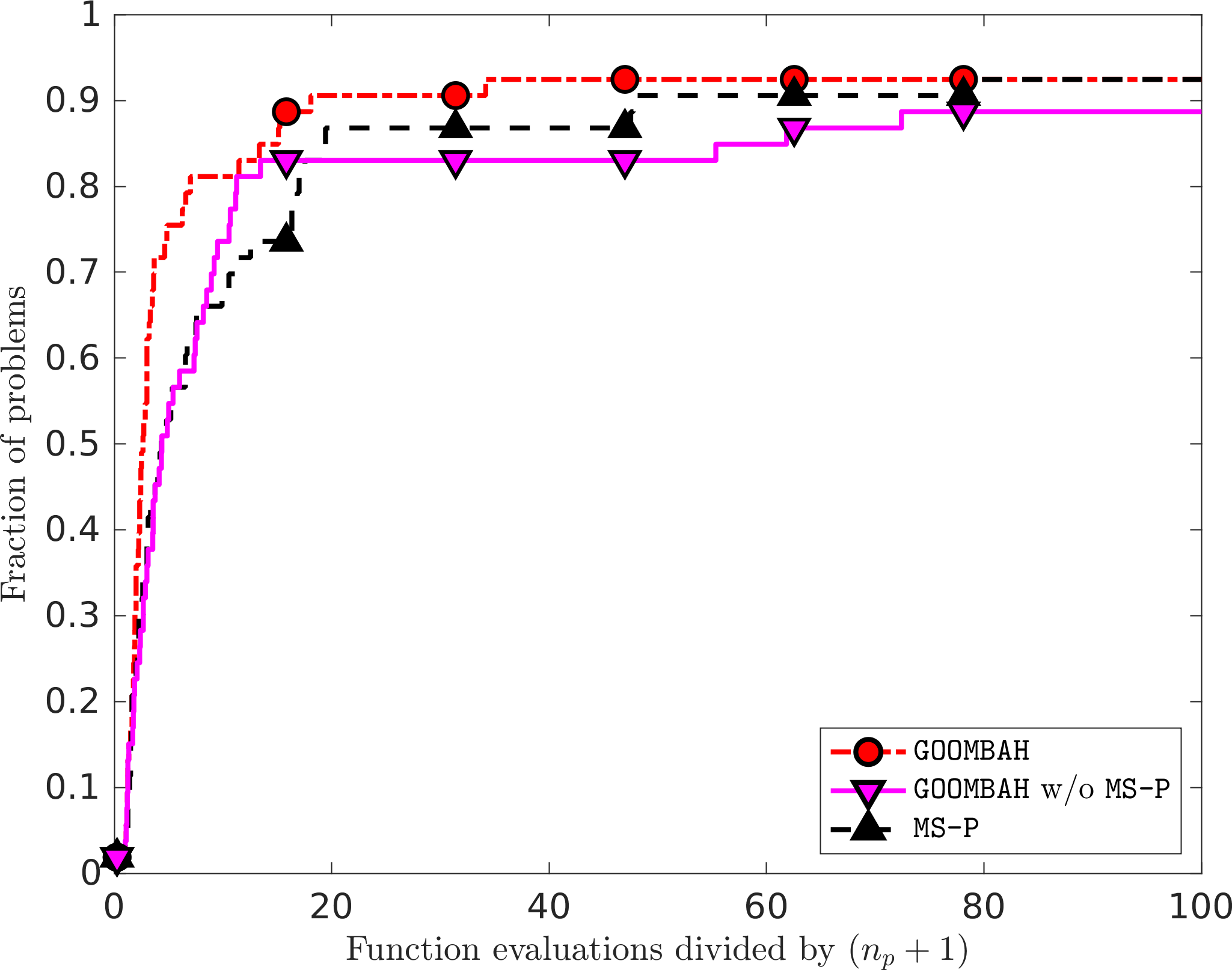}}\\[10pt]
    \subfigure[][$\tau = 10^{-3}$, unconstrained]{\includegraphics[width=0.45\linewidth]{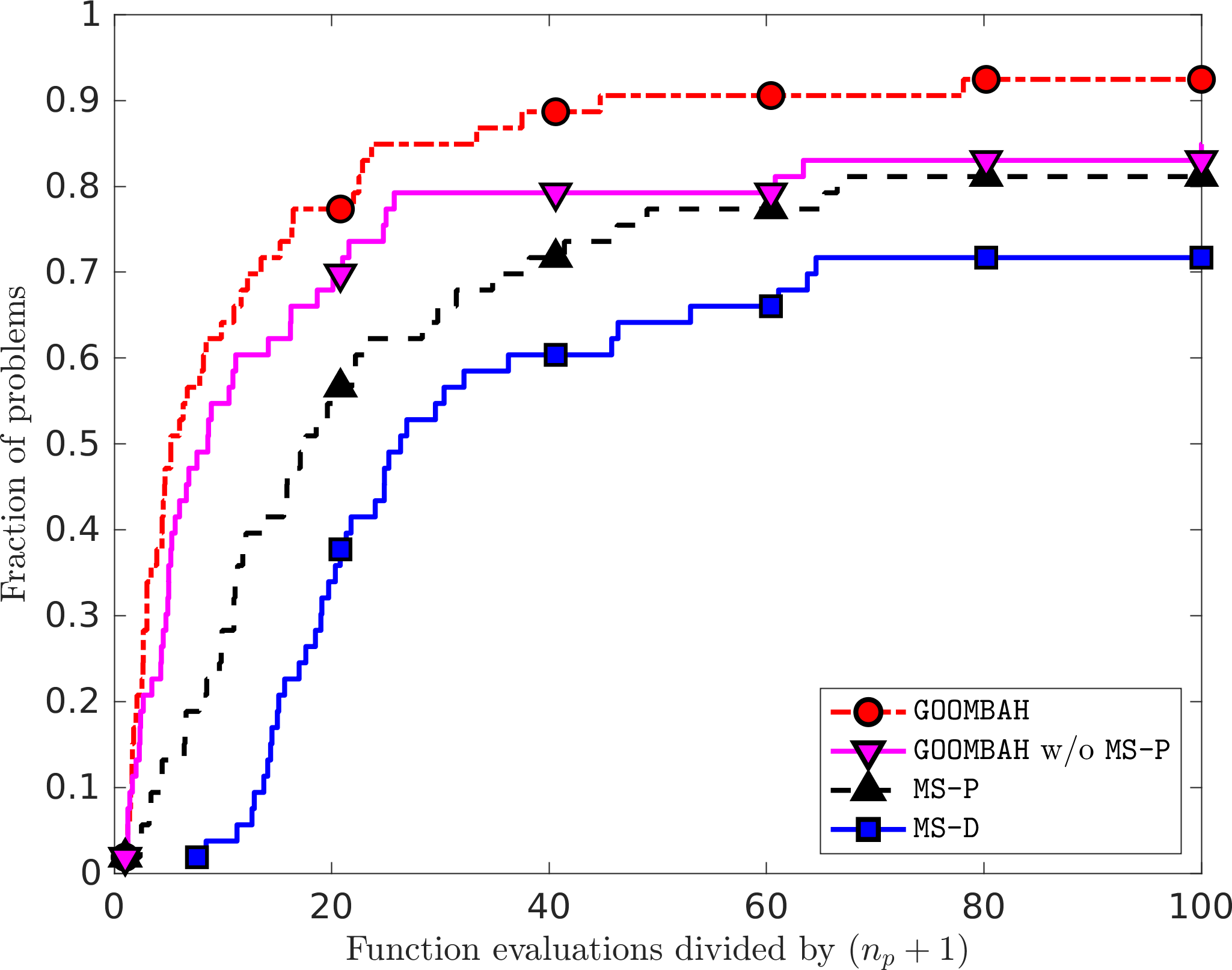}}
    \hfil
    \subfigure[][$\tau = 10^{-3}$, constrained]{\includegraphics[width=0.45\linewidth]{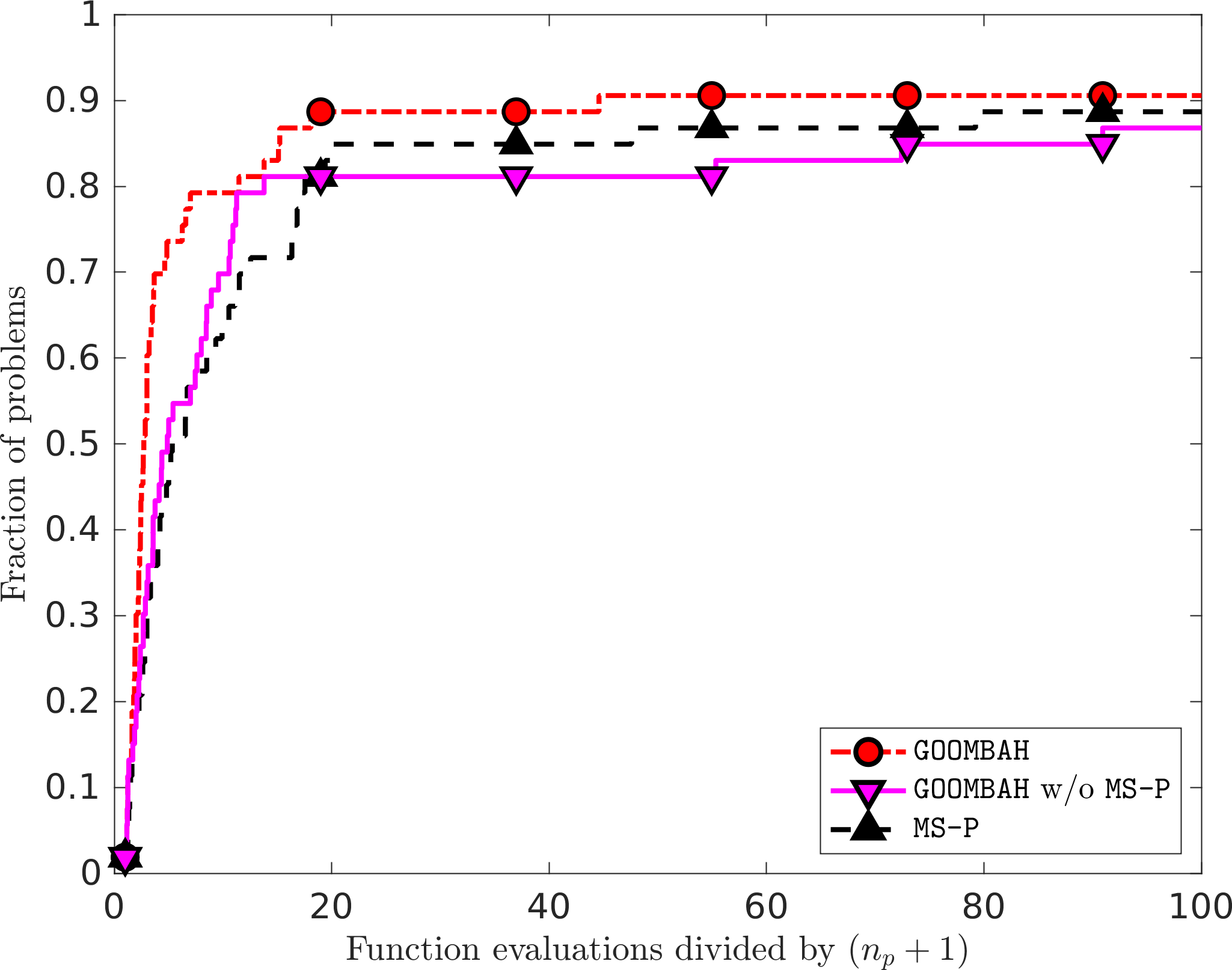}}\\[10pt]
    \subfigure[][$\tau = 10^{-5}$, unconstrained]{\includegraphics[width=0.45\linewidth]{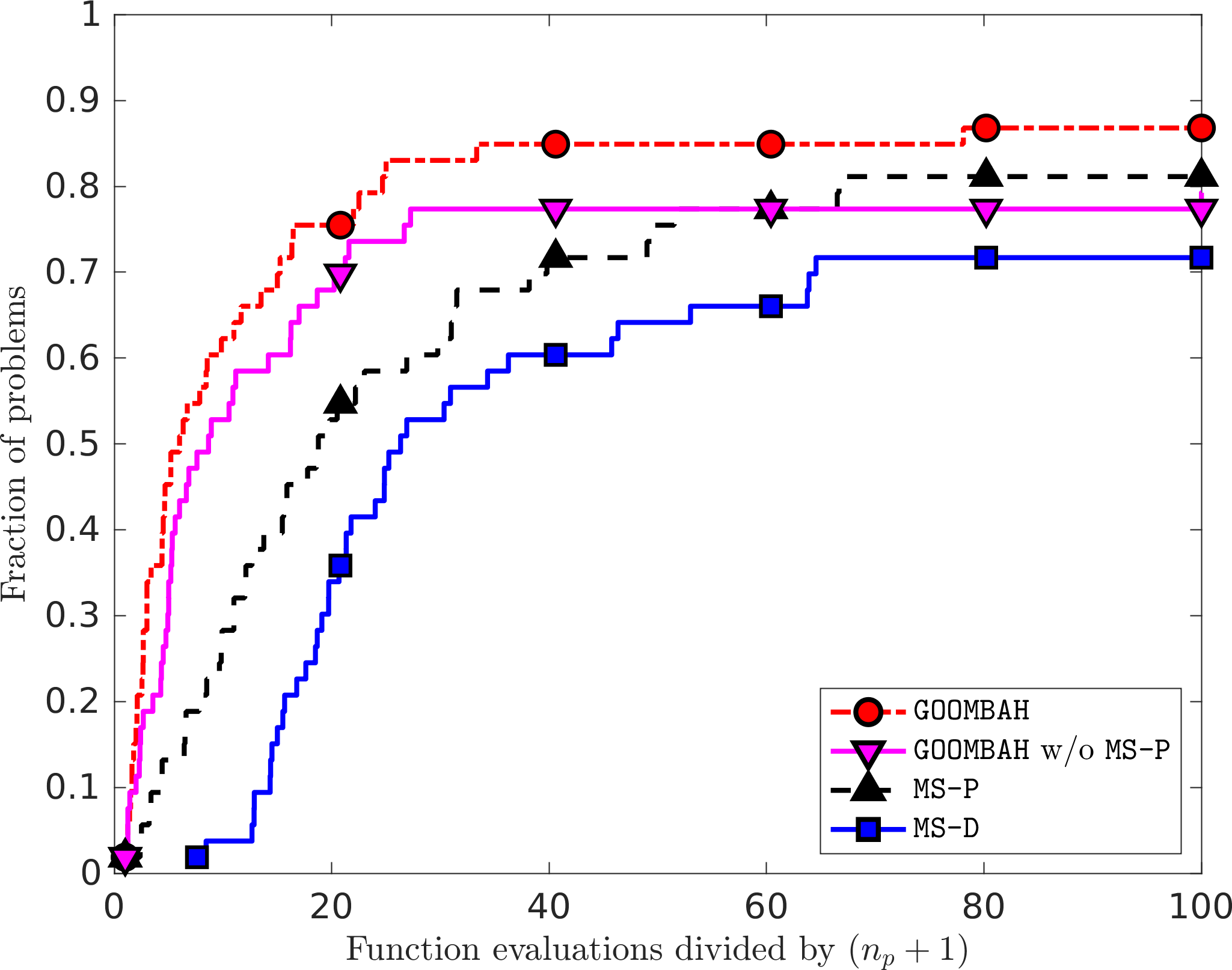}}
    \hfil
    \subfigure[][$\tau = 10^{-5}$, constrained]{\includegraphics[width=0.45\linewidth]{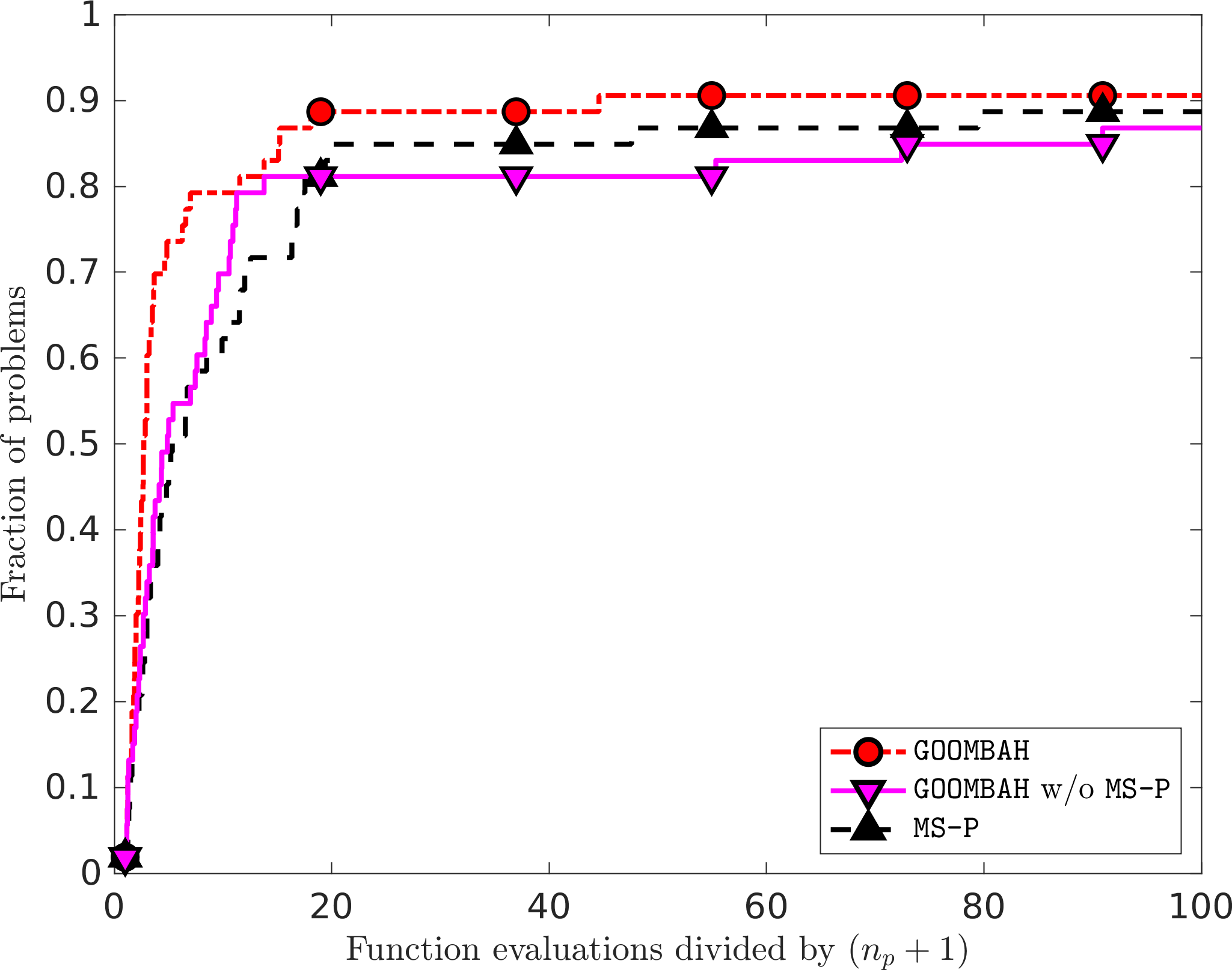}}
  \end{center}
  \caption{Data profiles on the censored-L1-loss function $h_3$, constrained and unconstrained, for three values of $\tau$.}
\end{figure}

\begin{figure}[h]
  \begin{center}
    \subfigure[][$\tau = 10^{-1}$, unconstrained]{\includegraphics[width=0.45\linewidth]{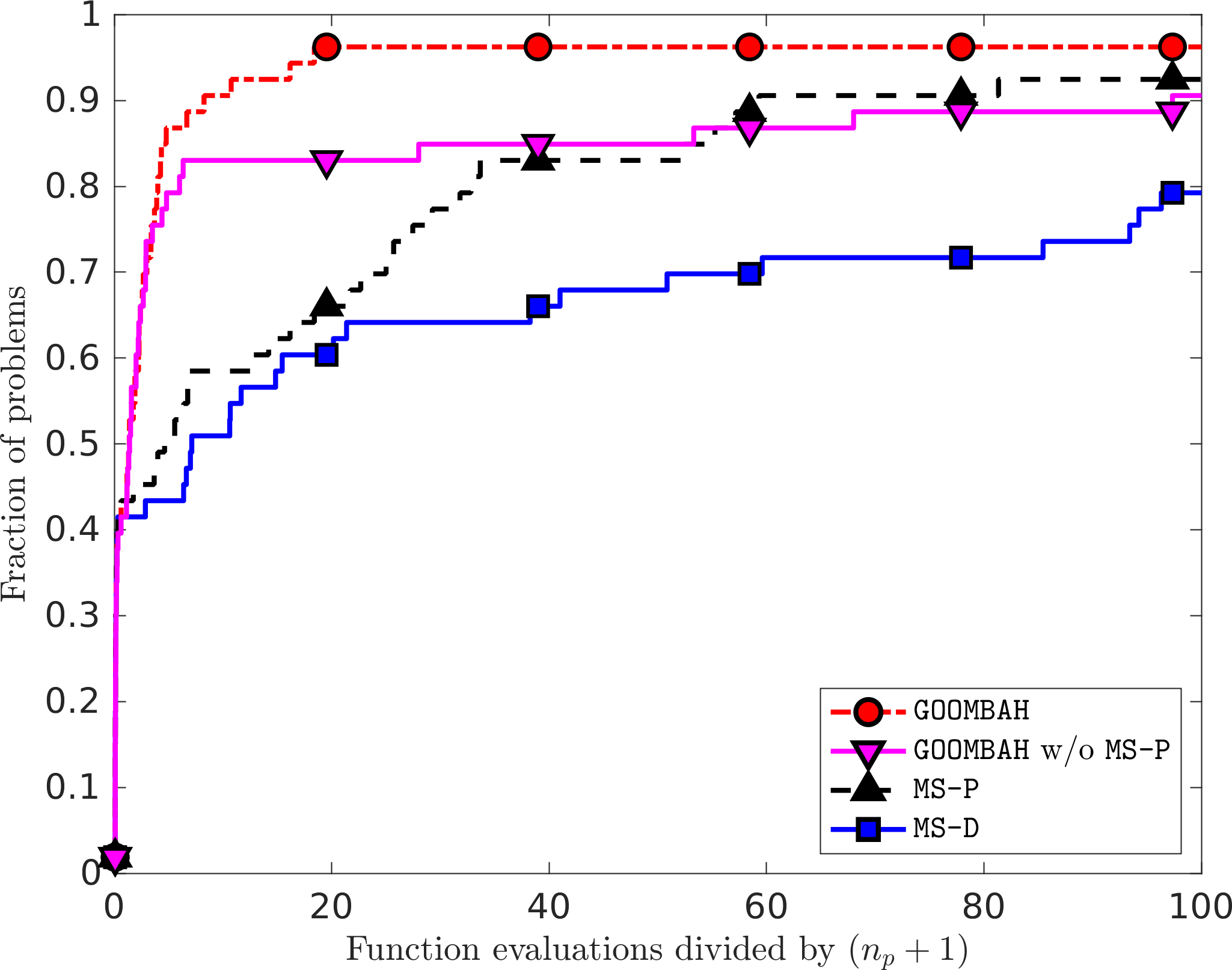}}
    \hfil
    \subfigure[][$\tau = 10^{-1}$, constrained]{\includegraphics[width=0.45\linewidth]{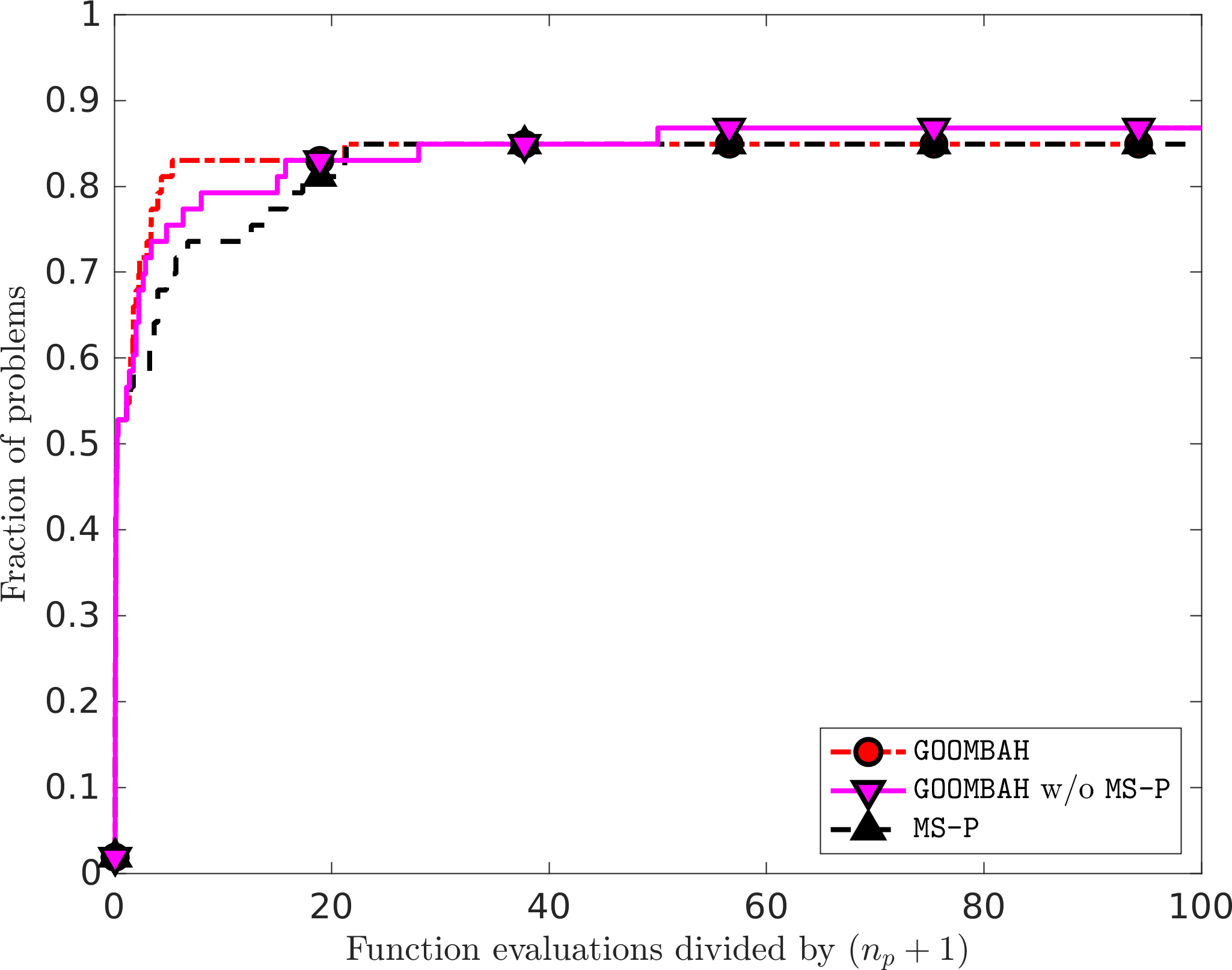}}\\[10pt]
    \subfigure[][$\tau = 10^{-3}$, unconstrained]{\includegraphics[width=0.45\linewidth]{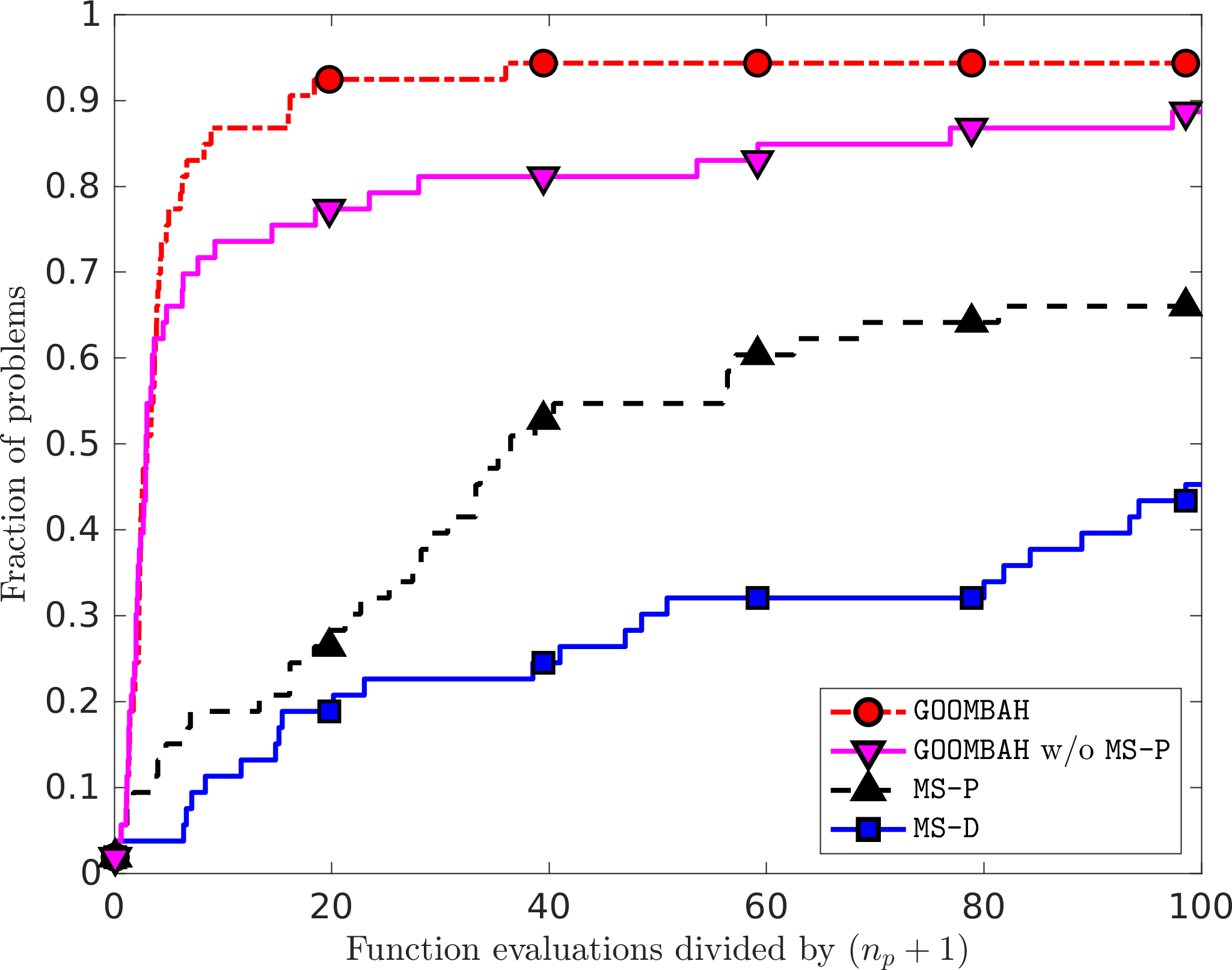}}
    \hfil
    \subfigure[][$\tau = 10^{-3}$, constrained]{\includegraphics[width=0.45\linewidth]{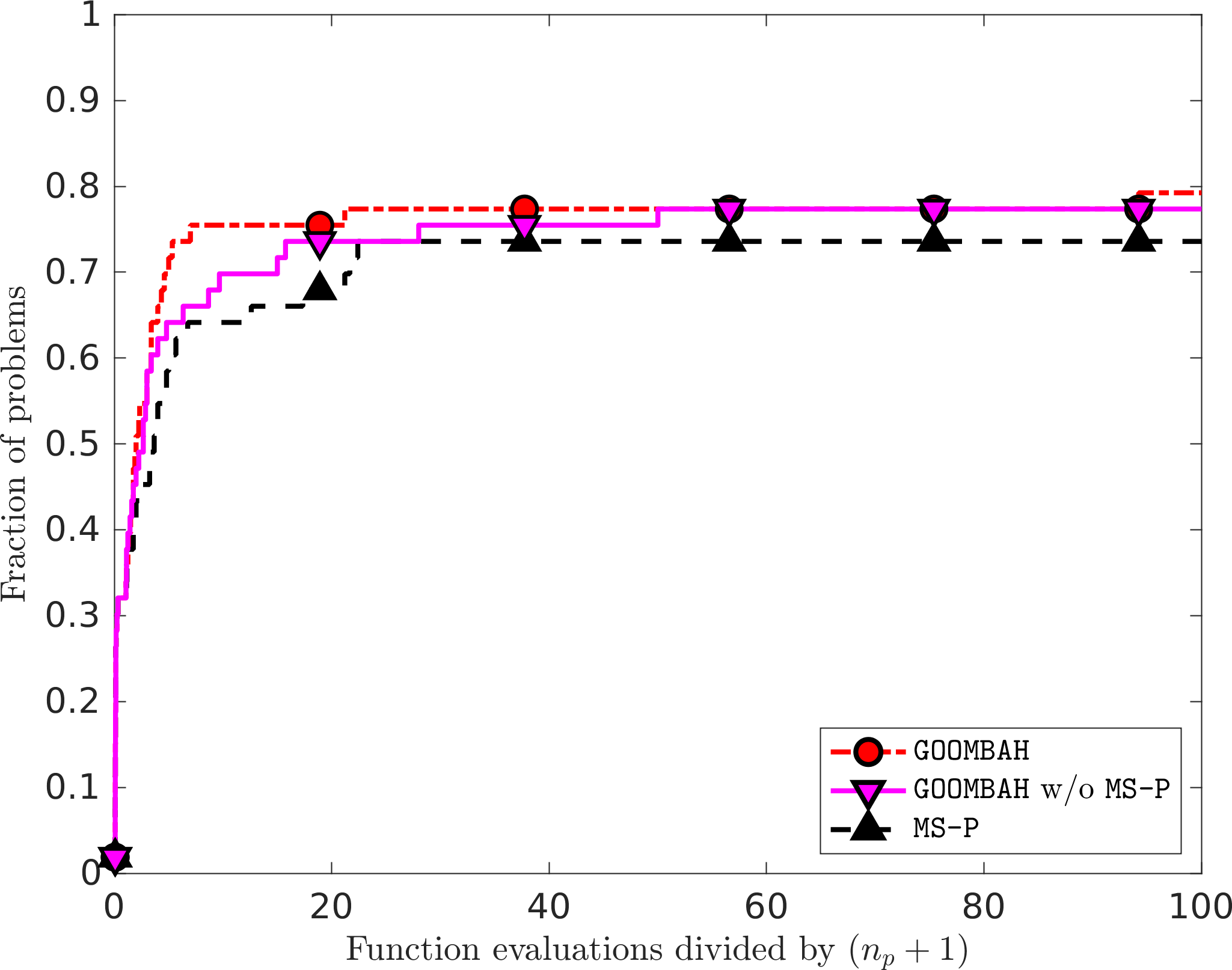}}\\[10pt]
    \subfigure[][$\tau = 10^{-5}$, unconstrained]{\includegraphics[width=0.45\linewidth]{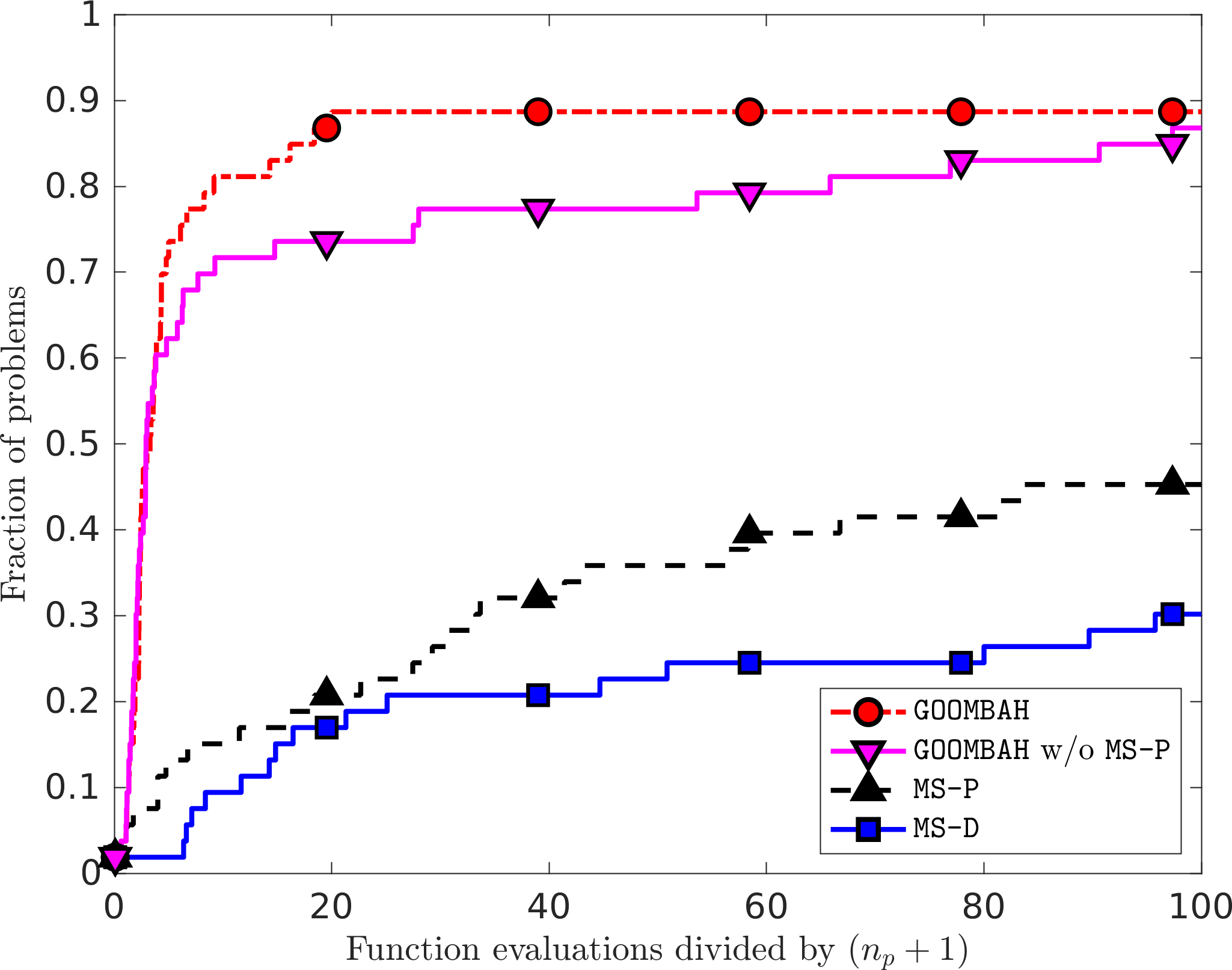}}
    \hfil
    \subfigure[][$\tau = 10^{-5}$, constrained]{\includegraphics[width=0.45\linewidth]{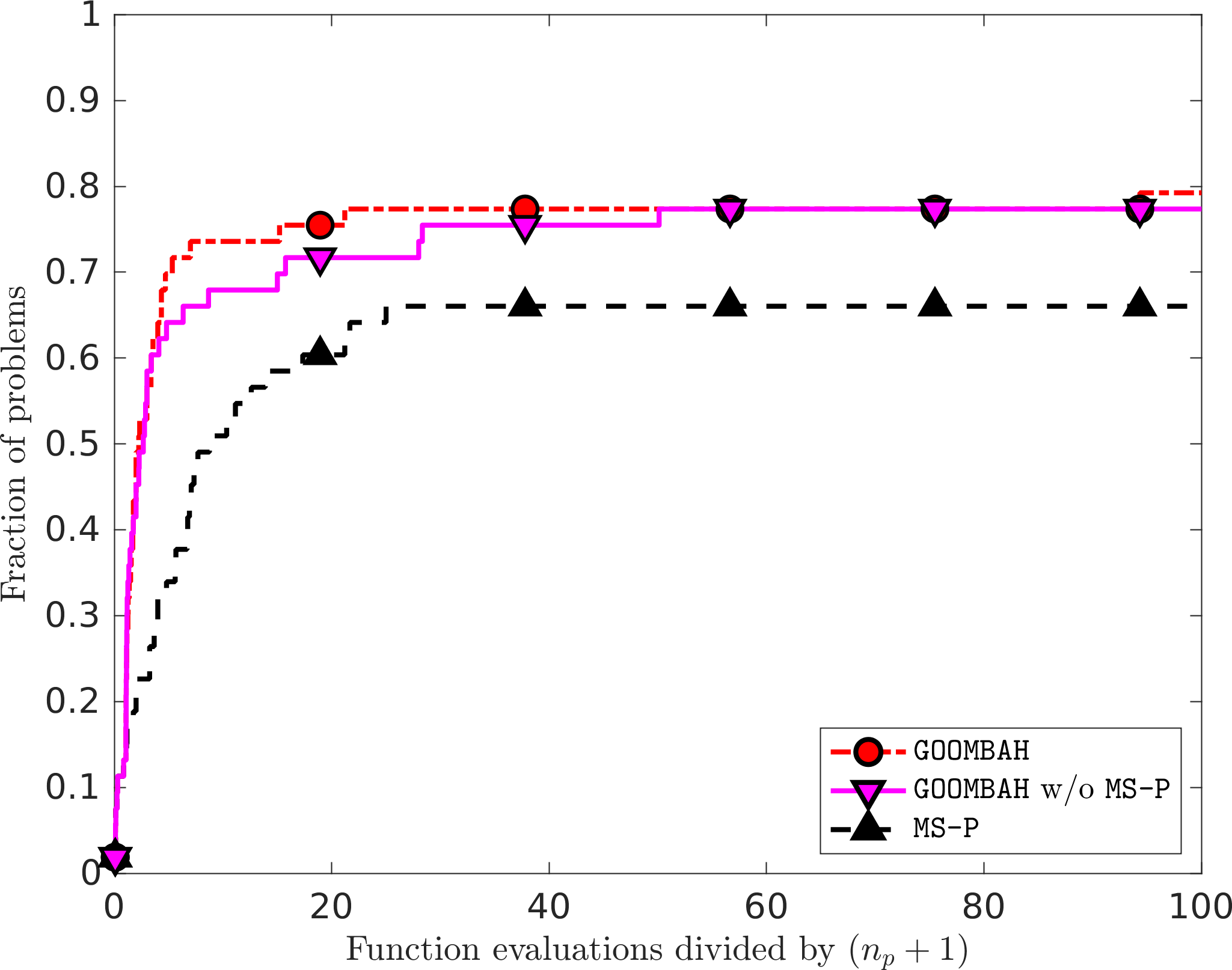}}
  \end{center}
  \caption{Data profiles on the piecewise-quadratic function $h_4$, constrained and unconstrained, for three values of $\tau$.}
\end{figure}

\vfill
\framebox{\parbox{.90\linewidth}{\scriptsize The submitted manuscript has been created by
        UChicago Argonne, LLC, Operator of Argonne National Laboratory (``Argonne'').
        Argonne, a U.S.\ Department of Energy Office of Science laboratory, is operated
        under Contract No.\ DE-AC02-06CH11357.  The U.S.\ Government retains for itself,
        and others acting on its behalf, a paid-up nonexclusive, irrevocable worldwide
        license in said article to reproduce, prepare derivative works, distribute
        copies to the public, and perform publicly and display publicly, by or on
        behalf of the Government.  The Department of Energy will provide public access
        to these results of federally sponsored research in accordance with the DOE
        Public Access Plan \url{http://energy.gov/downloads/doe-public-access-plan}.}}
\end{document}